\documentclass[11pt,twoside]{amsart}
\usepackage{amssymb,amsmath,amsthm}
\usepackage{graphicx}
\usepackage[colorlinks,linkcolor=blue,citecolor=blue]{hyperref}
\usepackage{tikz-cd}

\numberwithin{equation}{section}
\newtheorem{thm}{Theorem}[section]
\newtheorem{lem}[thm]{Lemma}

\newtheorem{pro}[thm]{Proposition}
\newtheorem{rmk}[thm]{Remark}

\newcommand{\bess}{\begin{eqnarray*}}
	\newcommand{\eess}{\end{eqnarray*}}
\newcommand{\C}{\mathbb{\widehat{C}}}
\newcommand{\E}{\widetilde{\mathcal{E}}}
\newcommand{\g}{\boldsymbol g}
\newcommand{\bs}{\boldsymbol b}

\renewcommand{\setminus}{-}
\renewcommand{\epsilon}{\varepsilon}

\usepackage{fancyhdr} %Ò³ÃŒºÍÒ³œÅ
\input xy
\xyoption{all}

\title{Continuity of Julia sets and invariant rays}

\author{Xiaoguang Wang}
\address{School of Mathematical Sciences, Zhejiang University, Hangzhou, 310027, China}
\email{wxg688@163.com}
\date{}

\begin{document}

	\begin{abstract}   For  certain  typical   perturbations $(f_n)_n$ of a rational map $f$ with parabolic cycles,   we investigate the relations between
		\begin{itemize}
			\item 	the Hausdorff convergence of Julia sets and  invariant rays,  and
			\item    the horocyclic convergence of multipliers of periodic points.
	       			\end{itemize}
	  
	  We establish several equivalent characterizations by means of  parabolic implosion theory.    This builds upon an analysis of the edge dynamics  on the  tree for the gate structure induced by the  perturbation. The edge  dynamics  which are driven by Oudkerk's algorithm, are used to trace the orbits for the near parabolic perturbations.
	\end{abstract}

	
	%{Primary 37F45; Secondary 37F10, 37F15}
	
	% AMS keywords (used in AMS journals)
	
	\date{\today} 
	
	\maketitle
	
	%\tableofcontents

	\section{Introduction}\label{int}
	
	Let ${\rm Rat}_d$ be the space of rational maps of degree $d\geq 2$, equipped with the algebraic topology. That is, we say  $f_n\rightarrow f$ {\it algebraically} if ${\rm deg}(f_n)={\rm deg}(f)$ and the coefficients of  $f_n$ (as a ratio of polynomials) can be chosen to converge to those of $f$.  The Julia set $J(f)$ is the closure of all repelling periodic points of $f\in {\rm Rat}_d$. 
	 The Julia set $J(f)$ determines a map
	 $$J: {\rm Rat}_d\rightarrow  \mathcal C(\C)$$
	from  ${\rm Rat}_d$ to the space $\mathcal C(\C)$ of non-empty compact subsets of the sphere $\C$, endowed with the Hausdorff topology. 
	
%	Base on the work of Douady \cite{D}, Yin \cite{Y} and Wu \cite{W}, it is known that 
	It  is  known from  Douady, Yin and Wu  \cite{D,Y, W} that
	 $J(f)$ varies continuously at   $f\in {\rm Rat}_d$ if and only if  $f$ has neither parabolic cycles nor rotation domains (Siegel disks and Herman rings).   
 When $f$ has parabolic cycles,  parabolic implosions provide a source of discontinuity  \cite{D}.  However,  for certain  sequences $(f_n)_n$ so that  $f_n\rightarrow f$  {algebraically},  it is still possible to have  $J(f_n)\rightarrow J(f)$.   
 This paper aims to determine suitable conditions for such sequences when the limit map $f$ has parabolic  points, and to study their continuity properties, including the continuity of invariant rays.
% This paper  aims to find the 
 %suitable conditions for these sequences, and to study the continuity properties (i.e. continuity of  invariant rays) for these sequences. 
%  study the continuity of Julia set and invariant rays at maps with parabolic cycles for some typical perturbations.  
 
 %The underlying idea is the parabolic implosion theory for parabolic germs with general multiplicity  developed by Oudkerk \cite{Ou99}.  

% Before introducing the main theorems, we  look at one of the applications, 
 %as an appetizer to the paper: 
 
% As an appetizer to the main results, we first introduce a motivating example as one of their applications.
 
 To introduce the main results and to signal one of their key applications, we begin with a motivating example.
 % before stating the principal theorems. 
 
 %Before stating the main theorems, we first present an application as an appetizer to the paper.
 
 %which is an apiti  the example for quadratic polynomial:  Let $f_n(z)=\lambda z+z^2$ with $\lambda_n\rightarrow 1$, and $f_0(z)=z+z^2$. If 
 
 	\begin{thm} \label{appe} Let $f_\lambda(z)=\lambda z+z^2$ with $\lambda\in \mathbb C$. Assume $\lambda_n\rightarrow  \lambda_0=e^{2\pi i p/q}$, which is a primitive $q${th} root of unity.  Then we have 
 		$$J(f_{\lambda_n})\rightarrow J(f_{\lambda_0}) \Longleftrightarrow  |{\rm Re}\big(1/(1-\lambda_n^q)\big)|\rightarrow +\infty.$$
 \end{thm}

%  \Big|{\rm Re}\Big(\frac{1}{1-\lambda_n^q}\Big)\Big|\rightarrow +\infty.

	\begin{figure}[h]
	\begin{center}
		\includegraphics[height=3.725cm]{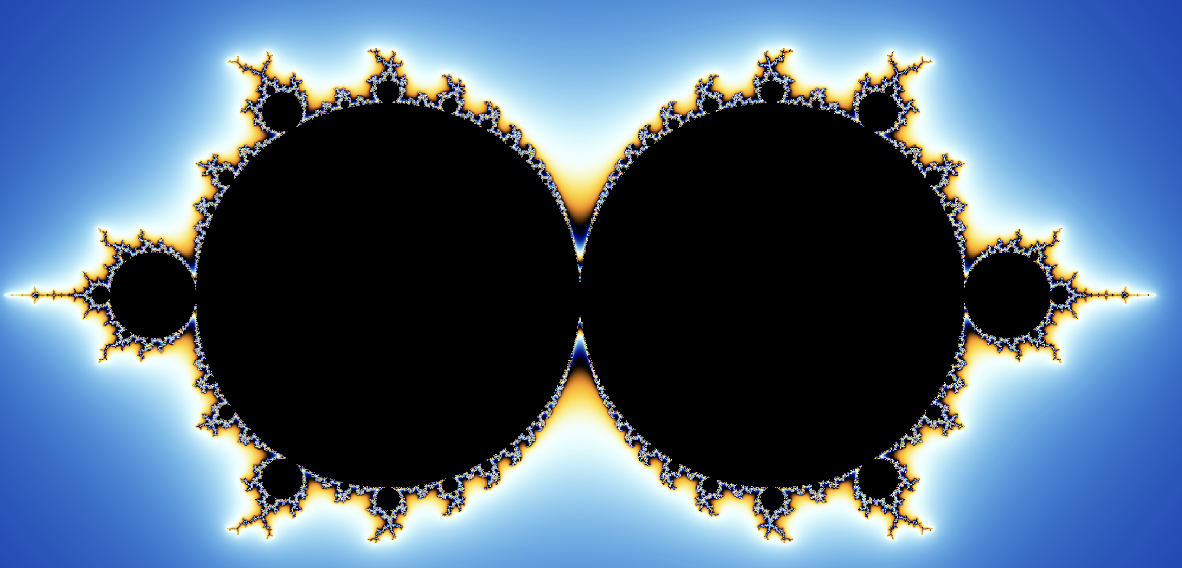}\\
			\includegraphics[height=4cm]{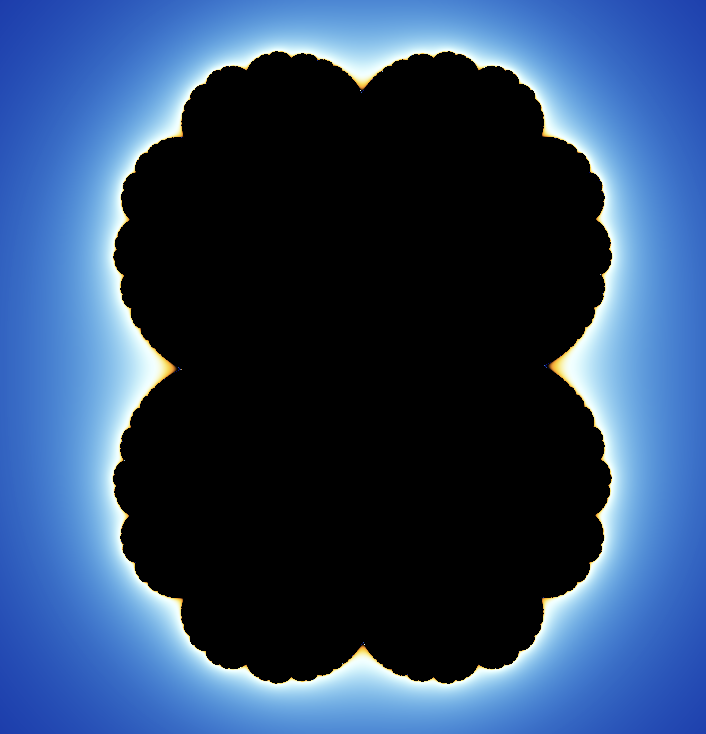} \includegraphics[height=4cm]{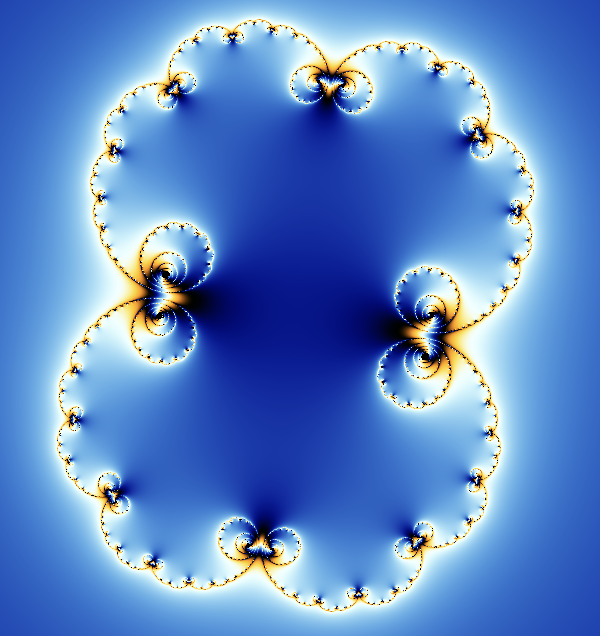}
	\end{center}
	\caption{The Double Mandelbrot set $\{\lambda\in \mathbb C; J(\lambda z+z^2) \text{ is connected}\}$ (up),  the Julia sets $J(\lambda z+z^2)$ for $\lambda=1-\epsilon$ and $\lambda=1+\epsilon i$ with small $\epsilon>0$. According to Theorem \ref{appe}, 
		$J(\lambda_n z+z^2)\rightarrow J(z+z^2) \Longleftrightarrow |{\rm Re}\big(1/(1-\lambda_n)\big)|\rightarrow +\infty$.
% $\lambda=1-\epsilon$ and $\lambda=1+\epsilon i$	
}
	\label{fig:mandelbrot-lambda}
\end{figure}

%The case $\lambda_0=1$ is implicit in many people's work, for example, Lavaurs \cite{Lav},  Douady \cite{D}, Bodart-Zinsmeister \cite{BZ},   Shishikura \cite{Sh98},  Oudkerk \cite{Ou99},   McMullen \cite{Mc00},  and Buff-Tan \cite{BT}.

The case $\lambda_0=1$ 
  appears implicitly in the work of many authors,
   including Lavaurs \cite{Lav}, Douady \cite{D}, Bodart–Zinsmeister \cite{BZ}, Shishikura \cite{Sh98}, Oudkerk \cite{Ou99}, McMullen \cite{Mc00}, and Buff–Tan \cite{BT}.
Recently, Simanjuntak  \cite{S25} proves this case using parabolic implosion for Blaschke products.  For  general $\lambda_0=e^{2\pi i p/q}$,  McMullen \cite{Mc00}  shows the continuity of Hausdorff dimensions 
 ${\rm H.dim}(J(f_{\lambda_n}))\rightarrow {\rm H.dim}(J(f_{\lambda_0}))$ under a slightly stronger condition (i.e., $\lambda_n \rightarrow \lambda_0$ radially).

%\begin{rmk} If $\lambda=1$, then $(p,q)=(0,1)$. The  right part of the equivalence condition is $|{\rm Re}(\lambda_n-1)^{-1}|\rightarrow +\infty.$  This case is proved independently by Simanjuntak \cite{S25} using different methods.  
	%\end{rmk}

Let $(\lambda_n)_n$ be a sequence in $\mathbb C^*$ with $\lambda_n\rightarrow 1$.   We say
$\lambda_n\rightarrow 1$ {\it  horocyclically} if $\lambda_n=e^{L_n+i\theta_n}$ satisfies $\theta_n^2/L_n\rightarrow 0$ (see \cite{Mc00}), 
or equivalently
$$\Big|{\rm Re}\Big(\frac{1}{1-\lambda_n}\Big)\Big|\rightarrow +\infty.$$

 Now let $f_n\rightarrow f_0$ algebraically. Assume $f_0$ has  a parabolic periodic point $\zeta$. 
Then there is a minimal integer $l\geq 1$ so that 
 $f_0^l(\zeta)=\zeta$ and $(f_0^l)'(\zeta)=1$ (note that the $f_0$-period of $\zeta$ is a divisor of $l$). Suppose
 \begin{itemize}
 	\item there are fixed points $\zeta_n$ of $f^l_n$ converging to $\zeta$; and
 	\item their multipliers $\lambda_n=(f^l_n)'(\zeta_n)$ % $\lambda_n=(f_n^l)′(\zeta_n)$ 
 	converge to $1$ horocyclically.
 	\end{itemize}

%(a) There are fixed points $\zeta_n$ of $f^l_n$ converging to $\zeta$; and

%(b) 
 
 Then we say $f_n\rightarrow f_0$  {\it horocyclically} at $\zeta$. 
If these conditions hold for all $f_0$-parabolic points  $\zeta$, we say  $f_n\rightarrow f_0$  {\it horocyclically}.

We call $(f_n)_n$ 
% A sequence  with $f_n\rightarrow f_0$ is called
  a {\it generic perturbation of $f_0$ at $\zeta$}, if  the $f_0$-parabolic  point $\zeta$  splits  into 
   non-parabolic periodic points for $f_n$ (i.e., there are no parabolic periodic points for $f_n$ near $\zeta$); a  {\it generic perturbation of $f_0$} if it is a generic perturbation of $f_0$ at  all $f_0$-parabolic points.

McMullen \cite{Mc00} proved that if $f_0$ is geometrically finite, and $f_n\rightarrow f_0$  both horocyclically and preserving critical relations on Julia sets, then 
$J(f_{n})\rightarrow J(f_0)$ (it is possible  that ${\rm H.dim}(J(f_n))\not\rightarrow {\rm H.dim}(J(f_{0}))$).   
We  prove a converse statement:

%	Let $f$ be a rational map with parabolic points.

\begin{thm} \label{hdc-horo} 
	Let $f_0$ be a rational map with parabolic points.   Let   $(f_n)_n$ be a generic perturbation of $f_0$.  If $J(f_{n})\rightarrow J(f_0)$, then 
	$f_n\rightarrow f_0$ horocyclically.
\end{thm}

%However, 
The converse of Theorem \ref{hdc-horo} is false.  %  as demonstrated by a counter due to 
Oudkerk \cite{Ou02} gave the example:
%\begin{rmk} 
%	The converse of Theorem \ref{hdc-horo} is false.  Oudkerk \cite{Ou02} gave the example  
	$$f_n(z)=z+z^4+1/n, \ n\geq 1;  \ f_0(z)=z+z^4.$$
	One may verify that $f_n\rightarrow f_0$ horocyclically, but the Julia set $J(f_n)$ does not converge to $J(f_0)$\footnote{A rigorous proof of the statement $J(f_n)\not\rightarrow J(f_{0})$ can be extracted from \S \ref{h-e}.}. %The essential reason for discontinuity 
In this example, the $f_0$-parabolic  fixed point $0$ splits into two repelling fixed points and two attracting fixed points for $f_n$.   
%\end{rmk}

 Let   $(f_n)_n$ be generic perturbation of $f_0$ at $\zeta$, which is an $f_0$-parabolic periodic point   with multiplicity\footnote{Let $l\geq 1$ be the minimal integer  so that 
 	$f_0^l(\zeta)=\zeta$ and $(f_0^l)'(\zeta)=1$, the multiplicity of $\zeta$ is the order of the zero of $f_0^l(z)-z$ at $\zeta$.} say $\nu+1\geq 2$.
%For each
 %	Let $f$ be a rational map with  $\zeta$  of 
We say $(f_n)_n$ is a
\begin{itemize}
	\item      {\it leaned sequence} of $f_0$ at $\zeta$, if  $\zeta$  splits into  $\nu+1$ non-attracting points or  $\nu+1$ non-repelling   points  for $f_n$;
		
			\item    {\it  balanced sequence} of $f_0$ at $\zeta$, if  $\zeta$  splits into  at least two  attracting points  and at least two repelling   points  for $f_n$;
				
			\item     {\it  $*$-sequence}  of $f_0$ at $\zeta$, if  $\zeta$  splits into  
			either a single attracting point accompanied by $\nu$ non-attracting points, or a single repelling point accompanied by $\nu$ non-repelling points  for $f_n$.
			% one attracting point together with $\nu$ non-attracting points, or one  repelling  point together with  $\nu$  
		%	 non-repelling points
	\end{itemize}

We call  {\it $(f_n)_n$   a $\omega$ sequence of $f_0$} if   $(f_n)_n$  is a  $\omega$ sequence of $f_0$ at all $f_0$-parabolic points, here $\omega\in \{\text{leaned},   \text{balanced}, * \text{-}\}$.  Note that any generic perturbation  $(f_n)_n$  of $f_0$  contains a subsequence that falls into one of the three types above.

%will  contain a  subsequence belonging to one of the three types above.
% which is one of the above three types.

If   $(f_n)_n$ is a  leaned sequence, it can be shown that $J(f_0)$ is strictly smaller than any Hausdorff limit of  $J(f_n)$, hence $J(f_n) \not\rightarrow J(f_0)$ (see Remark \ref{leaned}).  For balanced sequence, %the situation is subtle.
by Oudkerk's example above  and Theorem \ref{hdc-horo},  the horocyclic convergence
is strictly weaker than the Hausdorff   convergence  of Julia sets.
%for balanced sequences. %   
%  there are examples $(f_n)_n$  for which	$f_n\rightarrow f_0$ horocyclically but  $J(f_n)$ does not converge to $J(f_0)$ (see \cite{Ou02}).
%one can also construct examples $(f_n)_n$  for which $f_n\rightarrow f_0$ horocyclically and  $J(f_n)$   converges to $J(f_0)$ (\textcolor{red}{see Section}). 
%This example, together with
 %, implies that 

Our next theorem shows that for  $*$-sequences,  the two notions of convergence are equivalent.
% kinds of convergence are equivalent:
%  continuity of Julia sets.

%For these three types sequence, the continuity of Julia sets 

%We say  {\it $f_n$ is a $*$-sequence of $f$} if  $f_n$ is a $*$-sequence of $f$ at all parabolic points of $f$.

\begin{thm} \label{h-equiv}
 	Let $f_0$ be a rational map with parabolic points and without rotation domains.   Let   $(f_n)_n$ be a $*$-sequence of $f_0$. Then
	$$J(f_{n})\rightarrow J(f_0) \Longleftrightarrow  f_n\rightarrow f_0   \  \text{ horocyclically}.$$
\end{thm}

%\begin{thm}  \label{cont-j}
%	Let $f\in {\rm Rat}_d$ have parabolic points and no rotation domains.  If $f_n\rightarrow f$   horocyclically, then  $J(f_{n})\rightarrow J(f)$.
%\end{thm}

%We shall  prove a more general version of
Our approach to proving  Theorem \ref{h-equiv}
%The ideas in the proof 
allows us  to study the relations between
 the (partial) horocyclic convergence of multipliers,  the kernel convergence of Fatou components and the Hausdorff convergence of invariant curves.   The latter  is motivated by  studying  the Hausdorff limits of external rays for polynomials, %which is fundamental in polynomial dynamics, 
as demonstrated by the recent work of Petersen and Zakeri  \cite{PZ24,PZ24b}.

In the following, without loss of generality,   let $f_0$  have  a parabolic fixed point $\zeta$ with multiplier $1$ and   multiplicity $\nu+1$. Let $A_1, \cdots, A_\nu$ be all   immediate parabolic basins of $f_0$ at $\zeta$, labeled in  positive cyclic order.   Let $f_n\rightarrow f_0$ be  a  generic   perturbation  of $f_0$ at $\zeta$.

For each $1\leq k\leq \nu$, let $\gamma_k: [0, +\infty)\rightarrow A_k$ be an $f_0$-invariant curve  $f_0(\gamma_k(t))=\gamma_k(t+1)$, so that $\gamma_k(t)$ converges to $\zeta$ along the $k$-th attracting direction as $t\rightarrow +\infty$. 
%$\lim_{t\rightarrow +\infty}\gamma_k(t)=0$ and $\gamma_k$ is asymptotic to the $k$-th attracting.
Let $\gamma_{n,k}: [0, +\infty)\rightarrow \mathbb C$ be an $f_n$-invariant curve  $f_n(\gamma_{n,k}(t))=\gamma_{n,k}(t+1)$ so that $\gamma_{n,k}$ converges locally and uniformly to $\gamma_k$ in $[0, +\infty)$ as $n\rightarrow +\infty$.
%For each $1\leq k\leq \nu$, let $\gamma_k: (-\infty, 0]\rightarrow \mathbb D(0, r_0)$ be an $f$-invariant curve  $f(\gamma_k(t))=\gamma_k(t+1)$ converging to $\zeta$ (i.e. $\lim_{t\rightarrow -\infty}\gamma_k(t)=\zeta$) in  the $k$-th repelling direction. Let $\gamma_{n,k}: (-\infty, 0]\rightarrow \mathbb C$ be an $f_n$-invariant curve  $f_n(\gamma_{n,k}(t))=\gamma_{n,k}(t+1)$ so that $\gamma_{n,k}$ converges locally and uniformly to $\gamma_k$ in $(-\infty, 0]$. 
%It is possible that  $\gamma_{n,k}$ does not  converge  to $\gamma_k$  uniformly in $(-\infty, 0]$. 
We denote by $\gamma_{n,k} \rightrightarrows \gamma_k$ the  uniform convergence of $\gamma_{n,k}$ to $\gamma_k$ in $[0, +\infty)$. 
Note that the uniform convergence $\gamma_{n,k} \rightrightarrows \gamma_k$ implies the Hausdorff convergence $\overline{\gamma_{n,k}}\rightarrow \overline{\gamma_k}$, and is strictly stronger than the  local and uniform convergence.
 If $\gamma_{n,k}$ does not converge  to $\gamma_k$ uniformly in $[0, +\infty)$,  it is possible that no subsequence of     $(\gamma_{n,k})_n$   converges to $\gamma_k$ uniformly in $[0, +\infty)$, we denote this extreme case by  $\gamma_{n,k} \not\rightrightarrows  \gamma_k$.  
 %and by $\gamma_{n,k} \not\rightrightarrows  \gamma_k$    the failure of such uniform convergence.
 %If $\gamma_{n,k} \not\rightrightarrows  \gamma_k$, 
%We   denote  by $\gamma_{n,k} \rightrightarrows \gamma_k$ the uniform convergence of  $\gamma_{n,k}$   to $\gamma_k$   in $(-\infty, 0]$; $\gamma_{n,k} \not\rightrightarrows \gamma_k$ the non uniform convergence of  $\gamma_{n,k}$   to $\gamma_k$   in $(-\infty, 0]$.

	For  a sequence of compact sets $(X_n)_n$ in $\C$,  let $\mathcal{L}((X_n)_n)\subset \mathcal C(\C)$ be the set of   Hausdorff limits of all convergent subsequences of  $(X_n)_n$. \footnote{By Blaschke selection theorem, any sequence of compact sets  $(X_n)_n$  in $\C$ admits a subsequence converging to a compact set in Hausdorff metric. Hence $\mathcal{L}((X_n)_n)\neq \emptyset$.}
	
%	 its accumulation set is defined by
%\begin{equation}\label{accu}
%	$$\mathcal{A}((X_n)_n)=\bigcap_{k=1}^\infty\overline{\bigcup_{n\geq k} X_n}.$$
%\end{equation}

%for  any   $*$-sequence  $(f_n)_n$ of $f$ at $\zeta$, there exist an integer $0\leq \ell\leq \nu$ and a subsequence $(f_{n_j})_j$ such that   $f_{n_j}\rightarrow f$ $\ell$-horocyclically at $\zeta$. 

\begin{thm} \label{limit-A}   	Let $f_0\in {\rm Rat}_d$ have a parabolic fixed point $\zeta$ of multiplicity $\nu+1$, $f'_0(\zeta)=1$.  
%	Let $f_0$ be a rational map with a parabolic fixed point $\zeta$, and $f'_0(\zeta)=1$.  
	 Let   $(f_n)_n$ be a generic perturbation of $f_0$ at $\zeta$. Then the following three statements are equivalent:
	
	(1).  For each $n$, there is an $f_n$-fixed point  $\zeta_n$ so that   
	$$\zeta_n\rightarrow \zeta,  \text{ and } f_n'(\zeta_n)\rightarrow 1  \text{ horocyclically}.$$

	% and the    multiplier $\lambda_n=f_n'(\zeta_n)$ converges to $1$ horocyclically.
	
	% (2). For any  subsequence of $(f_n)_n$, there   exist a further subsequence $(f_{n_j})_j$  and  an immediate parabolic basin $A$ of $f$ at $\zeta$ so that 
	
	(2). Every subsequence of $(f_n)_n$ admits a further subsequence $(f_{n_j})_j$ and an immediate parabolic basin $A$ of $f_0$ at $\zeta$ such that
	$$ X \cap A=\emptyset, \ \forall X\in \mathcal{L}((J(f_{n_j}))_j). $$
			% \begin{equation} \label{limJ-basin}
		%\end{equation}
		
		(3). Every subsequence of $(f_n)_n$ admits a further subsequence $(f_{n_j})_j$ and  some   $  k\in [1, \nu]\cap \mathbb N$,  so that $\gamma_{n_j,k} \rightrightarrows \gamma_k$.
	\end{thm}
%	\begin{rmk} %Assuming (1),  the parabolic basin $A$   in (2) depends on the choice of the subsequence $(f_{n_j})_j$. 
	
		%	one can not expect there are more than one immediate parabolic basins  of $f$ at $\zeta$ 
		%	Theorem \ref{limit-A} also implies that if we can not find a 
%	\end{rmk}

We remark that in the statement (2),  when $(f_{n_j})_j$ and $A$  are chosen, it is possible  that  
for any other immediate parabolic basin $B\neq A$ of $f_0$ at $\zeta$, 
we have $ X \cap B\neq \emptyset, \ \forall X\in \mathcal{L}((J(f_{n_j}))_j)$.
%$\liminf_{j\rightarrow \infty} J(f_{n_j})\cap B\neq \emptyset$.   
See Theorem  \ref{pb-limit} below.

Buff and Tan \cite[Proposition 6.3]{BT} prove part (2) of Theorem \ref{limit-A} under the condition in terms of stability of polynomial vector fields,
whereas our condition in part (1) is both natural from a dynamical system viewpoint and fundamentally necessary.

A sharper version of Theorem \ref{limit-A} holds for $*$-sequences, which calls for the notion of partial horocyclic convergence.
%For $*$-sequences,  Theorem \ref{limit-A}   has a more precise version. To state such a theorem, we shall introduce the notion of  partial horocyclic convergence.
 %Let $f_0$  have  a parabolic fixed point $\zeta$ with multiplier $1$ and   multiplicity $\nu+1$.
Let    $(f_n)_n$ be a   $*$-sequence of $f_0$ at $\zeta$.
% at its parabolic fixed point $\zeta$ with multiplier $1$ and multiplicity $\nu+1$. 
We say that $f_n\rightarrow f_0$ {\it $\ell$-horocyclically at $\zeta$} for some integer $\ell\in [0, \nu]$,  if the  $\nu$ nearby non-attracting (or non-repelling) fixed points of $f_n$ can be labeled as $\zeta_{n,1}, \cdots, \zeta_{n, \nu}$ such that their multipliers $\lambda_{n,j}=f'_n(\zeta_{n,j})$ satisfy 
\begin{itemize}
	\item  the first $\ell$ multiplier sequences converge to $1$  horocyclically: 
	$$\Big|{\rm Re}\Big(\frac{1}{1-\lambda_{n,j}}\Big)\Big|\rightarrow +\infty,  \ 1\leq j\leq \ell;$$   
	\item  the remaining $\nu-\ell$ multiplier sequences satisfy
	$$\sup_{n\geq 1}\Big|{\rm Re}\Big(\frac{1}{1-\lambda_{n, j}}\Big)\Big|< +\infty, \ \ell< j\leq \nu.$$
\end{itemize}

Note that any $*$-sequence $(f_n)_n$ of $f_0$ at $\zeta$ must admit a subsequence $(f_{n_j})_j$ such that $f_{n_j}\rightarrow f_0$ $\ell$-horocyclically at $\zeta$ for some integer $\ell \in [0, \nu]$. 

See Figure \ref{fig:cubic-near} for an example of $1$-horocyclic convergence.
 %whereas our condition in part (1)  is natural from the perspective of  dynamical systems %intuitive  (possibly weaker) 
%and more importantly is also necessary.
% both weaker
%does not require this assumption.

%Buff and Tan \cite[Proposition 6.3]{BT} prove Theorem \ref{limit-A} (2) under the  condition of a stability  in terms of  polynomial vector fields, rather than Theorem \ref{limit-A} (1). 
% rather than . Our proof in finding $A$ is a little bit tricky.

%A version of Theorem \ref{limit-A}   is proven by 

	\begin{figure}[h]
	\begin{center}
		\includegraphics[height=5cm]{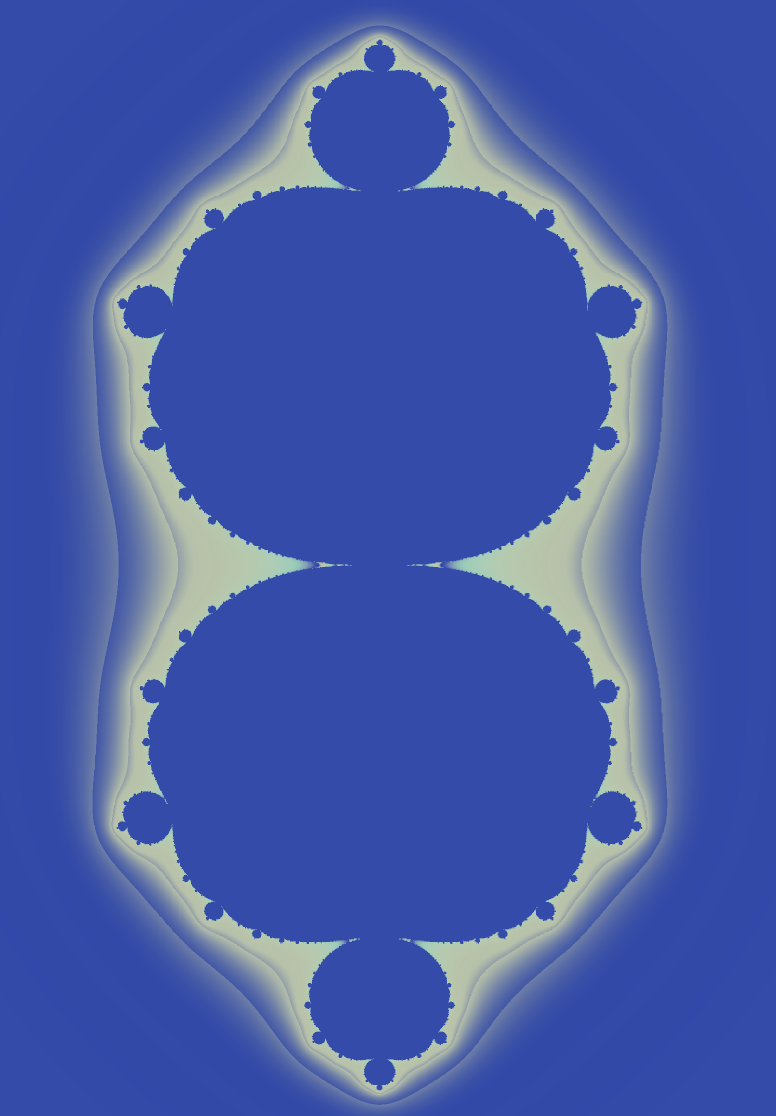} \includegraphics[height=5cm]{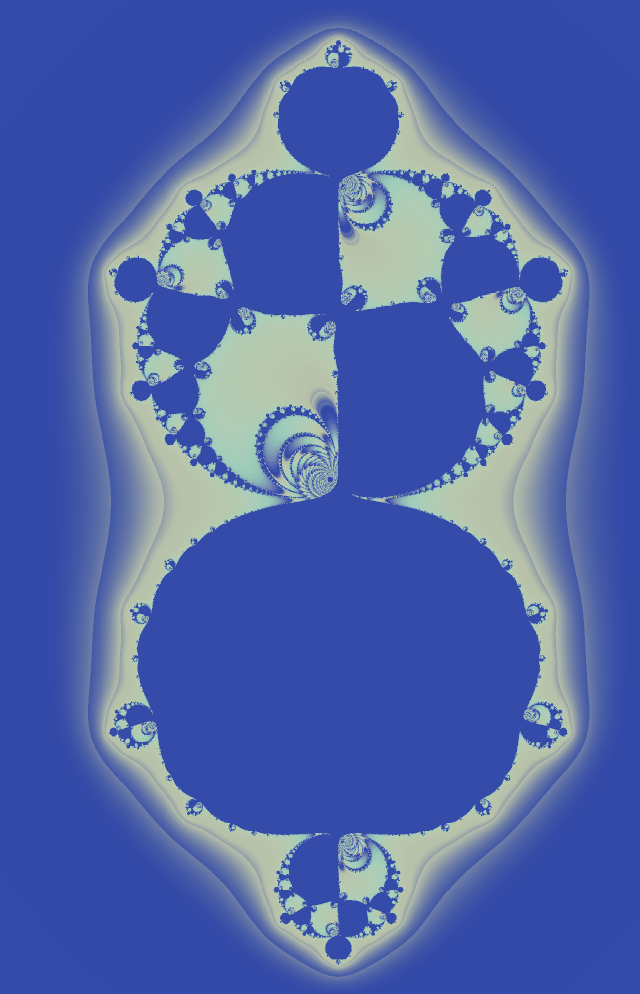}
	\end{center}
	\caption{Julia sets of  $f_{a,b}(z)=z+z((z-a)^2+b)$, where
		%  $J(z+z^3)$ (left)  and $J(f_{a,b})$ (right), where $f_{a,b}(z)=z+z((z-a)^2+b)$ and
		$(a,b)=(0,0)$ (left) and $(0.2i, -0.0005539+0.00012975i)$ (right).  This example  is provided by Jie Cao, and shows the case that $f_{a,b}\rightarrow f_{0,0}$    $1$-horocyclically at $0$.
	}
	\label{fig:cubic-near}
\end{figure}

\begin{thm} \label{pb-limit} 	Let $f_0$ have a parabolic fixed point $\zeta$ of multiplicity $\nu+1$, $f'_0(\zeta)=1$.   Let $f_n\rightarrow f_0$ be  a $*$-sequence of $f_0$ at $\zeta$. Let $\ell\in [0, \nu]$ be an integer.   The following three statements are equivalent:
	
(1).  $f_n\rightarrow f_0$ $\ell$-horocyclically at $\zeta$.
	% and let $1\leq m\leq \nu$.  Assume $m$  attracting (or repelling) multipliers sequences  converge to $1$  horocyclically,

(2).   For every subsequence of $(f_n)_n$, there exist a further subsequence $(f_{n_j})_j$ and a subset $\mathcal I \subset \{1, \dots, \nu\}$ with cardinality $\# \mathcal I  = \ell$ such that 
%Every subsequence of $(f_n)_n$ admits a further subsequence   $(f_{n_j})_j$, a subset $E\subset \{1,\cdots, \nu\}$ with cardinality  $\# E=\ell$ so that 
%\begin{equation} \label{lim-basin}
$$\begin{cases} X\cap A_k=\emptyset, \ \forall  X\in \mathcal{L}((J(f_{n_j}))_j), &\text{ if }  k\in \mathcal I; \\
 X\cap A_k\neq \emptyset,   \ \forall  X\in \mathcal{L}((J(f_{n_j}))_j),  & \text{ if }  k\notin \mathcal I.
\end{cases}$$
% \end{equation}

(3).   For every subsequence of $(f_n)_n$, there exist a further subsequence $(f_{n_j})_j$ and a subset $\mathcal I \subset \{1, \dots, \nu\}$ with  cardinality  $\# \mathcal I  = \ell$ such that 
%\begin{equation} \label{lim-curve}
$$	\begin{cases}   \gamma_{n_j,k} \rightrightarrows \gamma_k, &\text{ if }  k\in \mathcal I; \\
	 \gamma_{n_j, k}\not \rightrightarrows \gamma_k, & \text{ if }  k\notin \mathcal I.
\end{cases} $$
\end{thm}

%To Theorem \ref{pb-limit} (2)$\Longrightarrow$(1), we actually prove  \eqref{lim-basin}$\Longrightarrow$(1) and  \eqref{lim-curve}$\Longrightarrow$(1). 

We remark that the subsequence $(f_{n_j})_j$  and the index set $\mathcal I$ in  part (2) and part (3) of Theorem \ref{pb-limit}   can be the same.  
Applying Theorem \ref{pb-limit} to the case $\ell=\nu$, we have 

\begin{thm} \label{pb-limit1} 	 Let $f_0$ have a parabolic fixed point $\zeta$ of multiplicity $\nu+1$, $f'_0(\zeta)=1$.   Let $f_n\rightarrow f_0$ be  a $*$-sequence of $f_0$ at $\zeta$.   Then the following three statements are equivalent:
	
	(1).  $f_n\rightarrow f_0$ horocyclically at $\zeta$.
	% and let $1\leq m\leq \nu$.  Assume $m$  attracting (or repelling) multipliers sequences  converge to $1$  horocyclically,
	
	(2).   For any $1\leq k\leq \nu$, we have 
	$X\cap A_k=\emptyset$, $\forall X\in \mathcal{L}((J(f_{n}))_n)$.
	%$$\limsup_{j\rightarrow \infty} J(f_{n_j})\cap A_k\neq \emptyset, \  \forall k\notin E.$$ 
	
	(3).    For any $1\leq k\leq \nu$,  we have $\gamma_{n,k} \rightrightarrows \gamma_k$.
	%	    $$\limsup_{j\rightarrow \infty} J(f_{n_j})\cap A_k=\emptyset, \  \forall k\in E,$$ 
	%    $$\limsup_{j\rightarrow \infty} J(f_{n_j})\cap A_k\neq \emptyset, \  \forall k\notin E.$$ 
	%  $\ell$ immediate parabolic basins $A_1, \cdots, A_\ell$ of $\zeta$, such that 
	%	$$\limsup_{k\rightarrow \infty} J(f_{n_k})\cap (A_1\cup \cdots\cup A_\ell)=\emptyset.$$
\end{thm}

Theorem \ref{pb-limit1} (1)$\Longrightarrow$(3) generalizes the Tameness Theorem of Petersen and Zakeri \cite[Theorem B]{PZ24b}.  Note that in Petersen and Zakeri's work, the $f_0$-invariant curve $\gamma_k$  is required to converge to  $\zeta$ in a repelling direction, while in our work,  $\gamma_k$  converges to  $\zeta$ in an attracting direction.
This adjustment does not fundamentally change the conclusion, since the statements (1) and (3) and their supporting arguments apply equally to the case of holomorphic germs, 
 %This makes  no essential difference, because statements (1) and (3) and their arguments apply equally well to the  holomorphic germs,  
 and  the  repelling directions correspond to the attracting directions for the  inverse germ $f_0^{-1}$ at $\zeta$.   %and one may replace  $(f_n)_n$  by the sequence of the inverse germs $(f_n^{-1})_n$.

It is worth noting that horocyclic convergence is closely related to the Hausdorff convergence of external rays for polynomials. To illustrate this connection,
let  $\mathcal P_d$  denote the space  of monic and centered polynomials of degree $d\geq 2$.  The connected locus $\mathcal C_d$  consists of those $f\in \mathcal P_d$ whose Julia set $J(f)$ 
is connected. For each $f\in \mathcal C_d$, the external ray with angle $\theta \in \mathbb R/\mathbb Z$ is denoted by $R_{f}(\theta)$. 
 
Let $f_0\in \mathcal C_d$ and let  $\zeta\in J(f_0)$ be a  parabolic or repelling periodic point of $f_0$. 
It is known   \cite[\S 18]{Mil06}  that $\zeta$ is the landing point of finitely many external rays.   Let $\Theta$ be the set of angles $\theta \in \mathbb R/\mathbb Z$  for which  the external ray $R_{f_0}(\theta)$ lands at $\zeta$. If $\zeta$ is a $f_0$-repelling point, a classical result  \cite[Prop. 8.5]{DH} states that for any $\theta\in \Theta$ and any sequence $f_n\rightarrow f_0$   in  $\mathcal C_d$, we have  $\overline{R_{f_{n}}(\theta)}  \to  \overline{R_{f_0}(\theta)}$  in Hausdorff topology.
However, when $\zeta$ is a $f_0$-parabolic point, the situation becomes more subtle. Recently,  Petersen and Zakeri \cite{PZ24} show that any Hausdorff limit $ L\in \mathcal{L}((\overline{R_{f_{n}}(\theta)})_n)$   has the trichotomy:

\begin{itemize}
	\item The tame case: $L=\overline{R_{f_0}(\theta)}$.
	
	\item  The semi-wild case:  $L\supsetneq \overline{R_{f_0}(\theta)}$,  and  $L$ is the union of $\overline{R_{f_0}(\theta)}$ and at most countably many homoclinic arcs.
	
	\item  The wild case:  $L\supsetneq \overline{R_{f_0}(\theta)}$, and $L$ contains at least one and at most finitely many      heteroclinic arcs.
	\end{itemize}

%It is   natural to ask the insentric reason for the tame case. We show that this case is essentially caused by the horocyclic convergence:   

%It is natural to ask for the dynamical reason for the tame case.
The dynamical reason for the tame case is a natural question.
 We show that it is closely related  to the horocyclic convergence.

\begin{thm} \label{ext-ray0} Let $f_n, f_0\in \mathcal C_d$.  Let $\zeta$ be a  $f_0$-parabolic periodic  point, and assume  $(f_n)_n$ is a generic perturbation of $f_0$ at $\zeta$.     If   $\overline{R_{f_{n}}(\theta)}\rightarrow \overline{R_{f_0}(\theta)}$ in Hausdorff topology for all $\theta\in \Theta$, then  $f_n\rightarrow f_0$  horocyclically at $\zeta$.

	If  $(f_n)_n$ is a $*$-sequence of $f_0$ at $\zeta$,  then  $f_n\rightarrow f_0$  horocyclically at $\zeta$ if and only if $\overline{R_{f_{n}}(\theta)}\rightarrow \overline{R_{f_0}(\theta)}$ in Hausdorff topology for all $\theta\in \Theta$.
	\end{thm}
 
 Theorem \ref{ext-ray0} establishes an equivalence between the horocylic  convergence and the Hausdorff convergence of external rays, for $*$-sequences. It has  a more precise form for partial horocyclic convergence. To state such a result, assume
$f_0^{l}(\zeta)=\zeta$, $(f_0^{l})'(\zeta)=1$, and the multiplicity of $f_0^{l}$ at $\zeta$ is $\nu+1$.
%$f_0^{l}(\zeta)=$ there are $\nu$ immediate parabolic basins associated with $\zeta$.   
It is known   \cite[Theorem 18.13]{Mil06} that  $\# \Theta\geq \nu$. 
Let $\Theta_\nu \subset \Theta$ consist  of $\nu$ angles with the   property: for each $1 \leq k \leq \nu$, there exists   $\theta \in \Theta_\nu$ such that the external ray $R_{f_0}(\theta)$ lands at $\zeta$ in the $k$-th repelling direction.

%Let $\Theta_\nu \subset \Theta$ consist of $\nu$ angles so that for each $1\leq k\leq \nu$, there is $\theta\in \Theta_\nu $ so that   the external ray $R_f(\theta)$   lands at $\zeta$ in the $k$-the repelling direction.

%For $*$-sequences, the   horocyclic convergence  is closely related to the convergence of external rays for polynomials.

\begin{thm} \label{ext-ray} Let $f_0\in \mathcal C_d$ have a parabolic period point $\zeta$ of multiplicity $\nu+1$. 
	% Let $\Theta$ be the set of angles $\theta$  for which  the external ray $R_f(\theta)$ lands at $\zeta$ \footnote{Note that for each repelling direction, there is an angle $\theta\in \Theta$ such that the external ray $R_f(f)$ converges to $\zeta$ in this direction.}.   
Let  $(f_n)_n$ be a $*$-sequence  of $f_0$ at $\zeta$, and let $\ell\in [0, \nu]$ be an integer.   Then the following two statements
	  are equivalent:
	
	(1).  $f_n\rightarrow f_0$ $\ell$-horocyclically at $\zeta$.
	
	(2).   For every subsequence of $(f_n)_n$, there exist a further subsequence $(f_{n_j})_j$ and a subset  $\Theta_\ell \subset \Theta_\nu$ with $\# \Theta_\ell = \ell$ such that 
	%\begin{equation} \label{lim-curve}
	$$	\begin{cases}  \overline{R_{f_{n_j}}(\theta)}\rightarrow \overline{R_{f_0}(\theta)} \text{  in Hausdorff topology}, &\text{ if }  \theta\in \Theta_\ell; \\
 %\overline{R_{f_0}(\theta)} \subsetneq \liminf_n   \overline{R_{f_{n_j}}(\theta)},	
  \overline{R_{f_{n_j}}(\theta)} \not\to  \overline{R_{f_0}(\theta)} \text{  in Hausdorff topology}, 
 & \text{ if }   \theta\in\Theta_{\nu}\setminus \Theta_\ell.
	\end{cases} $$

\end{thm}

%We remark that the set $\Theta_\ell$ in (2) depends on the choice of the subsequence $(f_{n_j})_j$.

%Theorem \ref{ext-ray} is a supplement to  a classical result \cite[Proposition 8.5]{DH} on the continuity of external rays when the landing point is repelling.  
%Theorem  \ref{ext-ray} supplement a classical result on continuity of external rays when the landing points are repelling . 
%For generic perturbations, 
%Theorem  \ref{ext-ray} has a weaker form, see Remark \ref{ext-generic}.
% and \ref{ext-generic2}.
%In the case that  $\overline{R_{f_{n}}(\theta)} \not\to \overline{R_{f_0}(\theta)}$,
%   in Hausdorff topology, 
%all possible Hausdorff limits of $\overline{R_{f_{n}}(\theta)}$ are studied by  Petersen and Zakeri \cite{PZ24}. 

% \vspace{5pt}
%\noindent\textbf{Origanization.}  

 \vspace{5pt}
\noindent\textbf{Idea of the proof.}  
%The ideas behind the proofs of these results  rest on the implosion theory developed by Oudkerk.
The proofs of these results are grounded in the implosion theory developed by Oudkerk.
 Under a generic perturbation $(f_n)_n$,    the $f_0$-parabolic point  of $\zeta$ splits into finitely many  nearby non-parabolic periodic points of $f_n$,    and finitely many possible {\it gate structures} (introduced by Oudkerk \cite{Ou99}, see \S \ref{para-imp}) arise from    such a perturbation. These gate structures give  qualitative descriptions of the `egg-beater dynamics' for $f_n$'s.    By passing to a subsequence, we may assume  $(f_n)_n$ have the same  gate structure $\mathbf G$.  
%The  gate structure induces a   combinatorial tree $T$.
%In this way, the maps  $(f_n)_n$ have  the same combinatorial tree     induced by the
Each gate between the fixed points of $f_n$ will have
an associated complex number called the {\it lifted phase}.  The lifted phase is related to the holomorphic indices for the  $f_n$-fixed points above (or below) the gate.  Because of this relation, the (partial) horocyclic convergence of multipliers come into play.  Roughly speaking, the horocyclic convergence of multipliers allows us to control the $f_n$-orbit of the points passing through the gate. Combining with an algorithm designed by Oudkerk, we are able to trace the whole $f_n$-orbit of an initial point chosen  in the parabolic basin of $f_0$,     hence getting the  (partial)  continuity properties for the Julia sets 
%(or more generally, Fatou components) 
and the invariant rays.

%  It turns out that we are able to   and leads to the 

 \vspace{5pt}
\noindent\textbf{Structure of the paper.}   
%The idea of the proof  The paper is organized as follows:
In \S  \ref{para-imp},  we review the parabolic   implosion theory developed by Oudkerk.
In \S \ref{h-e},    we introduce Oudkerk's algorithm \cite{Ou02}, which plays a  crucial role in studying the  near parabolic perturbations. Some properties are given therein.  
In \S  \ref{generic-p}, we shall study generic perturbations using Oudkerk's algorithm,  
%relate Oudkerk's algorithm to an induced  dynamical system  in a  combinatorial tree,
 and prove 
Theorems \ref{hdc-horo} and  \ref{limit-A}.  In \S \ref{generic-p}, we will focus on the   $*$-sequences. 
%relate the $\ell$-horocyclic convergence to the edge dynamics on the combinatorial tree induced by the gate structure for $∗$-sequences. From this relation, we extract the partial continuity for the sequence of Fatou components and invariant rays.
We relate the $\ell$-horocyclic convergence to the edge dynamics on the combinatorial tree induced by the gate structure. 
This relation yields the partial continuity of both  the Julia sets and the invariant rays.
%From this relation,   the partial continuity for 
 %Julia sets  and invariant rays are obtained.
%the sequence of Fatou components and invariant rays.
In \S \ref{proofs}, we prove the remaining theorems in \S \ref{int}.

 \vspace{5pt}
\noindent\textbf{Acknowledgments.}  The author would like to thank Jie Cao,  Carsten  Petersen and  Saeed Zakeri for  helpful discussions and  inspiring examples.   
 
The research is supported by National Key R\&D Program of China (Grant No. 2021YFA1003200),  National Natural Science Foundation of China (Grant No. 12131016), and  the Fundamental Research Funds for the Central Universities 2024FZZX02-01-01.

 \vspace{5pt}
\noindent\textbf{Notations.}  	Throughout the paper we adopt the following notations:

\begin{itemize}
	\item $\mathbb C$ and $\C$:  the complex plane and the Riemann sphere
	
	\item $\mathbb N$ and $\mathbb Z$: the set of natural numbers $1,2,\cdots$ and  the set of integers
	
	\item $\mathbb D(a,r)=\{z\in \mathbb C; |z-a|<r\}$, $\mathbb D=\mathbb D(0,1)$.
%	\item For  a sequence of  sets $(A_n)_n$ in $\C$,  its accumulation set
%	\begin{equation}\label{accu}
%	\mathcal{A}((A_n)_n)=\bigcap_{k=1}^\infty\overline{\bigcup_{n\geq k} A_n}.
%	\end{equation}
	%, \ \liminf_{n\rightarrow \infty}A_n=\bigcup_{k=1}^\infty\bigcap_{n=k}^\infty A_n.$$
\end{itemize}

\section{Parabolic implosion following Oudkerk} \label{para-imp}

%The parabolic implosion theory is a t 

%In this section, we shall review the parabolic implosion theory  developed by Oudkerk \cite{Ou99}, see also  \cite{Lav, D, Sh98}. 

In this section, we review the theory of parabolic implosion as developed by Oudkerk \cite{Ou99}, along with related works in \cite{Lav, D, Sh98}.
We collect some  results used in this paper, following the presentation of \cite{Ou02} whose proofs  can be found in \cite{Ou99}.

 For an open subset $W\subset \C$,   let $\mathcal H$ denote the family of  holomorphic maps $f: \mathcal D(f)\rightarrow \C$, where $\mathcal D(f)\subset W$ is open. We equip $\mathcal H$ with a topology as follows. We say $f_n\rightarrow f$ in the {\it compact-open topology} if we have the following: if $K\subset \mathcal D(f)$ is compact, then
$K\subset \mathcal D(f_n)$ for large $n$ and $f_n\rightarrow f$ uniformly on $K$.

Let $\mathcal F$ be a  family of holomorphic germs defined in a small neighborhood of $0$. Suppose  $f_0\in \mathcal F$ has a parabolic fixed point at $0$ with multiplier $1$ and multiplicity $\nu+1$. Up to a conformal change of coordinates,   %$f_0$ takes the form 
  $f_0(z)=z+z^{\nu+1}+{O}(z^{\nu+2})$.  We fix a small $r_0>0$ so that every $f\in \mathcal F$ is defined in a neighborhood of  $K_0=\{|z|\leq 2r_0\}$.
  By Rouch\'e's Theorem, if  $f\in \mathcal F$  is close to $f_0$, then it will have 
  $\nu+1$
  fixed points in $K_0$ counted with multiplicity. 
  
  Let $z_{k,-}:=r_0 e^{2\pi i (k-1)/\nu}$  and $z_{k,+}:=e^{\pi i /\nu} z_{k,-}$ for $k\in \mathbb Z_\nu=\mathbb Z/\nu \mathbb Z$. Clearly each $z_{k,-}$ lies on a  repelling direction of $f_0$ and  each $z_{k,+}$ lies on an  attracting direction of $f_0$.
  
  For $\phi\in(-\pi/4, \pi/4)$ and $f\in \mathcal F$, consider the vector field 
  \begin{equation}\label{vfield}
  \dot{z}= e^{i\phi} \cdot i (f(z)-z).
  \end{equation}

 Let $\gamma_{k, \pm, f, \phi}: I_{k, \pm, f, \phi}\rightarrow K_0$   solve \eqref{vfield} with 
$\gamma_{k, \pm, f, \phi}(0)=z_{k,\pm}$, where $ I_{k, \pm, f, \phi}\ni 0$ is the largest interval in $\mathbb R$.
Let $\ell_{k, \pm, f, \phi}=\gamma_{k, \pm, f, \phi}(I_{k, \pm, f, \phi})$.

For $f=f_0$, we have the following

\begin{pro} For each $k\in \mathbb Z_\nu$, $s\in \{\pm\}$ and  $\phi\in(-\pi/4, \pi/4)$, we have
	
	1.  $I_{k, s, f_0, \phi}=\mathbb R$ and $\gamma_{k, s, f_0, \phi}(\pm\infty)=0$. Also $\overline{\ell_{k, s, f_0, \phi}}$ bounds an open Jordan disk denoted by $U_{k, s, f_0, \phi}$.
	
	2. $f_0(U_{k, +, f_0, \phi})\subset U_{k, +, f_0, \phi}$ and $U_{k, -, f_0, \phi}
	\subset f_0(U_{k, -, f_0, \phi})$.
	
	3. There is a univalent map $\Phi_{k, s, f_0, \phi}: U_{k, s, f_0, \phi}\rightarrow \mathbb C$ satisfying 
	$$\Phi_{k, s, f_0, \phi}(f_0(z))=\Phi_{k, s, f_0, \phi}(z)+1,$$
	whenever $z, f_0(z)\in U_{k, s, f_0, \phi}$. This $\Phi_{k, s, f_0, \phi}$ is unique up to addition by a constant.
		\end{pro} 
	
		\begin{figure}[h]
		\begin{center}
			\includegraphics[height=4cm]{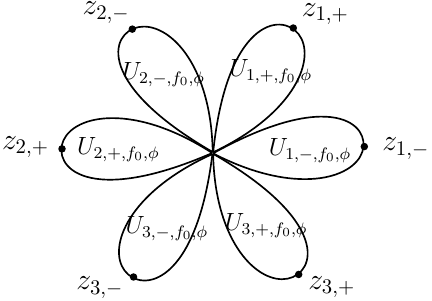}
		\end{center}
		\caption{The sets $U_{k, s, f_0, \phi}$, in the case $\nu=3$.}
		\label{fig:setsf0}
	\end{figure}
	
	For each $k\in \mathbb Z_\nu$, the sets $U_{k, +, f_0, \phi}$ and $U_{k, -, f_0, \phi}$ are called the $k$-th {\it attracting petal} and the $k$-th {\it repelling petal}  for $f_0$,  respectively.

Let  $f\in \mathcal F$ and  $\phi\in(-\pi/4, \pi/4)$.  If for each $k\in \mathbb Z_\nu$ and $s\in \{\pm\}$, we have $I_{k, s, f, \phi}=\mathbb R$ and values $t_{\pm}\in \mathbb R$ such that $t_-<0<t_+$ and 
$$\gamma_{k, \pm, f, \phi}(t)\notin \mathbb D(0, r_0/2)\Longleftrightarrow t\in [t_-,t_+],$$
then we say that $f$ is {\it $\phi$-well behaved}, and let
 $$\mathcal{WB}_\phi:=\{f\in \mathcal F; f \text{ is } \phi\text{-well behaved}\}.$$
In particular, $f_0\in \mathcal{WB}_\phi$ for each $\phi\in(-\pi/4, \pi/4)$.

%\begin{thm}
	%\end{thm}

\begin{pro} If $f\in \mathcal{WB}_\phi$ for some  $\phi\in(-\pi/4, \pi/4)$, then  for each $k\in \mathbb Z_\nu$ and $s\in \{\pm\}$, we have
	
	1.  The limits $\gamma_{k, s, f, \phi}(+\infty)$ and $\gamma_{k, s, f, \phi}(-\infty)$ exist.
	
	2.  $\gamma_{k, +, f, \phi}(+\infty)=\gamma_{k, -, f, \phi}(+\infty)$ and $\gamma_{k, +, f, \phi}(-\infty)=\gamma_{k-1, -, f, \phi}(-\infty)$.
	
	3. If $\gamma_{k, s, f, \phi}(+\infty)\neq \gamma_{k, s, f, \phi}(-\infty)$, then there exist some $j\in \mathbb Z_\nu$ and $s'\in \{\pm\}\setminus\{s\}$ such that
	$$\gamma_{k, s, f, \phi}(+\infty)=\gamma_{j, s', f, \phi}(+\infty), \ \gamma_{k, s, f, \phi}(-\infty)=\gamma_{j, s', f, \phi}(-\infty).$$
	
	4.  $\ell_{k, \pm, f, \phi}\cap f(\ell_{k, \pm, f, \phi})=\emptyset$.
	 
	Also for every $f$-fixed point $\sigma\in K_0$, there exists $k\in \mathbb Z_\nu$ and $s\in \{\pm\}$
	 for which $\sigma=\gamma_{k, s, f, \phi}(+\infty)$ or $\sigma=\gamma_{k, s, f, \phi}(-\infty)$. 
\end{pro} 

	\begin{figure}[h]
	\begin{center}
		\includegraphics[height=5cm]{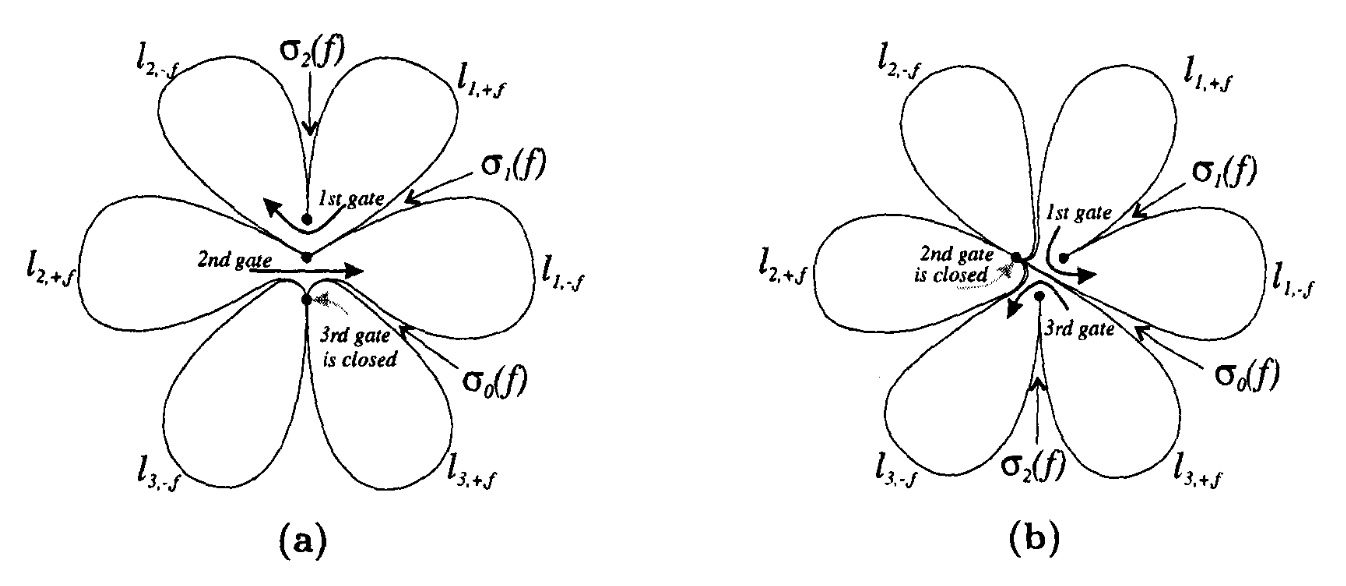}
	\end{center}
	\caption{Two cases of well behaved map $f$.   In both cases, $\nu=3$ and $\sigma_0(f), \sigma_1(f), \sigma_2(f)$  are fixed points of $f$.  This Figure comes from Oudkerk's thesis \cite[Figure 2.4]{Ou99}.  The gate structure  $\mathbf G=(2,1,*)$ for (a),  and  $\mathbf G=(1,*,3)$ for (b).}
	\label{fig:gate2}
\end{figure}

If $\gamma_{k, s, f, \phi}(+\infty)=\gamma_{k, s, f, \phi}(-\infty)$, let $U_{k, s, f, \phi}$ be the open Jordan domain bounded by $\overline{\ell_{k, s, f, \phi}}$ and call it a {\it single petal}. Otherwise, if 
$$\gamma_{k, +, f, \phi}(+\infty)=\gamma_{j, -, f, \phi}(+\infty)  \text{ and } \gamma_{k, +, f, \phi}(-\infty)=\gamma_{j, -, f, \phi}(-\infty),$$
we let $U_{k, +, f, \phi}=U_{j, -, f, \phi}$ denote the open Jordan domain bounded by $\overline{\ell_{k, +, f, \phi}}\cup \overline{\ell_{j, -, f, \phi}}$, and call this a {\it double petal}.

\begin{pro} If $f\in \mathcal{WB}_\phi$ for some  $\phi\in(-\pi/4, \pi/4)$, then  for each $k\in \mathbb Z_\nu$ and $s\in \{\pm\}$,  there is a  map $\Phi_{k, s, f, \phi}: U_{k, s, f, \phi}\rightarrow \mathbb C$ satisfying 
	$$\Phi_{k, s, f, \phi}(f(z))=\Phi_{k, s, f, \phi}(z)+1,$$
	whenever $z, f(z)\in U_{k, s, f, \phi}$. This $\Phi_{k, s, f, \phi}$ is unique up to addition by a constant, therefore we can normalize it so that $\Phi_{k, s, f, \phi}(z_{k,s})=0$.
\end{pro} 

For $f\in \mathcal{WB}_\phi$, let $S_{k, \pm, f, \phi}$ be the component of $\overline{U_{k, \pm, f, \phi}}\setminus f^{\pm 1}({U_{k, \pm, f, \phi}})$ which contains $z_{k, \pm}$. Then  
$$S'_{k, \pm, f, \phi}:=S_{k, \pm, f, \phi}\setminus \{\gamma_{k, \pm, f, \phi}(+\infty), \gamma_{k, \pm, f, \phi}(-\infty)\}$$
is the {\it fundamental domain}. The Fatou coordinate $\Phi_{k, \pm, f, \phi}$ is well defined on this set.

	\vspace{5pt}
\noindent \textbf{Gate structure and lifted phases.}	 For each map $f\in \mathcal{WB}_\phi$, the {\it gate structure}
for $f$ (with respect to $\phi$) is constructed as follows:
 \begin{itemize}
	
	\item Draw a  circle, and mark $2\nu$ points around its perimeter.
	
	\item Label these in anticlockwise order as $(1,-), (1,+), \cdots, (\nu, -), (\nu, +)$.
	
	\item For each $k$ such that $U_{k, +, f, \phi}=U_{j, -, f, \phi}$ is a double petal, draw an arrow from $(k, +)$ to $(j, -)$.
	\end{itemize}
For each  $f\in \mathcal{WB}_\phi$, the gate structure picture  satisfies the properties: 

 \begin{itemize}
	
	\item The arrows (can be drawn so that they) do not cross one another. 
	
	\item The arrows always go from `$+$'-points to `$-$'-points.

	\item No  marked points have more than one arrow going to or from it.
\end{itemize}

Any picture satisfying above three properties is called {\it admissible}. Clearly, there are only finitely many admissible pictures (for a fixed $\nu$).

Every  admissible gate structure   can be represented by a gate vector $\mathbf G=(\mathbf \g(1), \cdots, \g(\nu))\in (\mathbb Z_\nu\cup \{*\})^\nu$ so that
$$\g(k)=
\begin{cases} j &\text{ if there is an arrow from }(k, +) \text{ to } (j,-);\\
 * &\text{ if there is no arrow from }(k, +).
\end{cases}$$
We say the $k$-th gate of $\mathbf  G$ is {\it open} if $\g(k)\neq *$.

	\begin{figure}[h]
	\begin{center}
		\includegraphics[height=6cm]{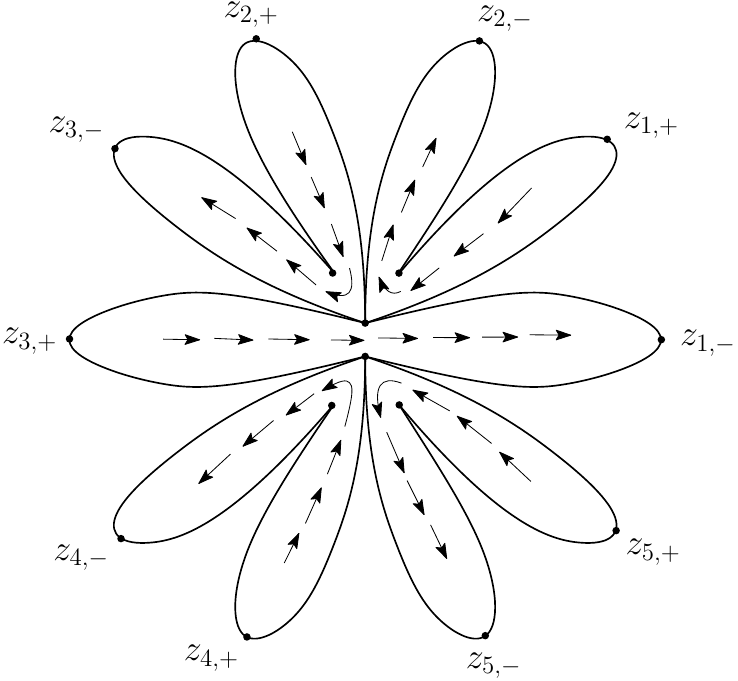}
	\end{center}
	\caption{An example of gate structure $\mathbf  G=(2,3,1,4,5)$, in this case $\nu=5$. }
	\label{fig: gate1}
\end{figure}

For $f\in \mathcal{WB}_\phi$ and  $k\in \mathbb Z_\nu$, we define
$$\tau_{k,\phi}(f)=\begin{cases} \Phi_{j, -, f, \phi}-\Phi_{k, +, f, \phi} &\text{ if  } U_{k, +, f, \phi}=U_{j, -, f, \phi} \text{ is a double petal}; \\
	\infty  &\text{ if  } U_{k, +, f, \phi}\text{ is a single petal}.
\end{cases}$$
We call $\tau_{k,\phi}$ the {\it lifted phase} for the $k$-th gate. Since  $\Phi_{k, s, f, \phi}(z_{k,s})=0$ for all $k\in \mathbb Z_\nu$  and $s\in \{\pm\}$, we see that when $\g(k)=j\neq *$,  
$$\tau_{k,\phi}(f)=\Phi_{j, -, f, \phi}(z_{k,+})=-\Phi_{k, +, f, \phi}(z_{j,-}).$$

%The {\it holomorphic index} $\iota(f, \sigma)$ of $f$ at a fixed point $\sigma$ is defined as
%\begin{equation} \label{h-ind}
%\iota(f, \sigma):=\frac{1}{2\pi i}\int_{\gamma}\frac{dz}{z-f(z)},
%\end{equation}
%where  $\gamma$ is a small loop  in the positive direction around $\sigma$.
% If $f'(\sigma)\neq 1$, then $\iota(f, \sigma)=\frac{1}{1-f'(\sigma)}$.

Suppose $f\in \mathcal{WB}_\phi$  and the $k$-th gate of $\mathbf  G$ is open.  All of the $\nu+1$ fixed points of $f$ in $K_0$ are assumed   in $\mathbb D(0, r_0/2)$, whilst $U_{k,+,f,\phi}$ contains no fixed points. The set $\mathbb D(0, r_0/2)\setminus U_{k,+,f,\phi}$ has two components.  The set of fixed points  in the component  which contains $\gamma_{k, +, f, \phi}(+\infty)$ (resp. $\gamma_{k, +, f, \phi}(-\infty)$) is denoted by ${\rm Fix}_{k, \phi}^{u}(f)$ (resp. ${\rm Fix}_{k, \phi}^{\ell}(f)$).  Essentially, ${\rm Fix}_{k, \phi}^{u}(f)$ (resp. ${\rm Fix}_{k, \phi}^{\ell}(f)$) consists of the fixed points above (resp. below) the $k$-th gate. 

Let $\mathbf G$ be an admissible gate vector and 
$$\mathcal{WB}_\phi(\mathbf G):=\{ f\in \mathcal{WB}_\phi; \ f \text{ has gate structure }  \mathbf G \text{ w.r.t. } \phi\}.$$

\begin{pro}     [\cite{Ou99} \S 2.4; \cite{Ou02} \S 6] \label{phase} For an admissible gate vector   $\mathbf G$, there are constants $c_k, c'_k$ such that if $f\in \mathcal{WB}_\phi(\mathbf G)$ and $\g(k)\neq *$,
	% the $k$-th gate of  $\mathbf G$ is open, 
	then
	\bess\tau_{k,\phi}(f)&=&2\pi i \sum_{\sigma\in {\rm Fix}_{k, \phi}^{u}(f)}\iota(f, \sigma)+c_k+o(1)\\
	&=&-2\pi i \sum_{\sigma\in {\rm Fix}_{k, \phi}^{\ell}(f)}\iota(f, \sigma)+c'_k+o(1),
	\eess
	as $f\rightarrow f_0$ in $\mathcal{WB}_\phi(\mathbf G)$, where $\iota(f, \sigma)$ is defined by \eqref{h-ind}.
\end{pro}

 \begin{pro} [Selection Principle] \label{sp} For any sequence $(f_n)_n$ in $\mathcal F$ so that $f_n\rightarrow f_0$ uniformly in $K_0$, there exist a subsequence  $(f_{n_j})_j$,  an angle $\phi\in(-\pi/4, \pi/4)$, and an admissible gate vector $\mathbf G$ so that 
 	$$f_{n_j}\in \mathcal{WB}_\phi(\mathbf G), \text{ for all } j\geq 1,$$
 	and  for each $k\in \mathbb Z_\nu$ with $\g(k)\neq *$,
 	$${\rm Re}(e^{-i\phi} \tau_{k,\phi}(f_{n_j}))\rightarrow -\infty  \text{ as } j\rightarrow \infty.$$
 	\end{pro}

 \begin{lem} [\cite{Ou02}, Lemma 13.1] \label{non-wb}  There are two constants   $0<M\ll L$  so that if $f\notin  \mathcal{WB}_\phi$ is close to $f_0$, then there is a subset $E\neq \emptyset$ of the fixed points of $f$ in $K_0$ so that
 	$$\sum_{\sigma\in E} \iota(f, \sigma)\in S(\phi):= e^{i\phi}\{w\in \mathbb C; \ {\rm Re }(w)<-L, |{\rm Im }(w)|<M\},$$
 	where $\iota(f, \sigma)$ is the holomorphic index of $f$ at the fixed point $\sigma$,  defined as
 	\begin{equation} \label{h-ind}
 		\iota(f, \sigma):=\frac{1}{2\pi i}\int_{|z-\sigma|=r}\frac{dz}{z-f(z)}, 
 	\end{equation}
 	for small $r>0$.
 	\end{lem}
 
  \begin{proof}[Proof of Proposition \ref{sp}]  For $f$ close to $f_0$,   there are 
  $\nu+1$ fixed points of $f$ in $K_0$. 	Let ${\rm Fix}(f)$ be the set of these fixed points, and let
  	$$X(f)=\Big\{\sum_{\sigma\in E} \iota(f, \sigma); \ \emptyset\neq  E\subset {\rm Fix}(f) \Big\}.$$
  	It clear that the cardinality  $\# X(f)<2^{\nu+1}$.
  	
  		\begin{figure}[h]
  		\begin{center}
  			\includegraphics[height=5cm]{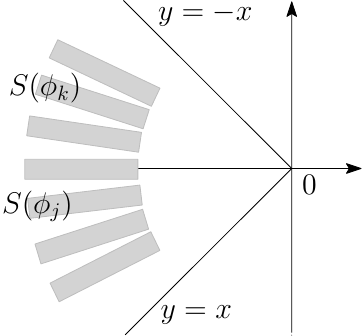}
  		\end{center}
  		\caption{}
  		\label{fig:selection}
  	\end{figure}
  	
  	Let $N=2^{\nu+1}+\nu$.
  	Choose angles $\phi_1, \cdots, \phi_{N}\in (-\pi/4, \pi/4)$ so that $S(\phi_k)$ (defined by Lemma \ref{non-wb}) lies to the left of $\{|y|=-x\}$  and $S(\phi_k)\cap S(\phi_j)=\emptyset$ for $k\neq j$, see Figure \ref{fig:selection}.  For any sequence $(f_n)_n$ so that  $f_n\rightarrow f_0$ uniformly in $K_0$, there exist a subset $\Theta\subset \{\phi_1, \cdots, \phi_{N}\}$ with cardinality $\# \Theta=\nu+1$ and a subsequence  $(f_{n_j})_j$ so that 
  	$$X(f_{n_j})\cap S(\phi)=\emptyset, \ \forall  j\geq 1, \  \forall \phi\in \Theta.$$
  	By Lemma \ref{non-wb}, $f_{n_j}\in\bigcap_{\phi\in \Theta} \mathcal{WB}_\phi$. Let $\mathbf G_{n_j}$ be the gate vector of $f_{n_j}$. Since there are only finitely many admissible gate vectors for  a fixed $\nu$, we may choose a further subsequence of   $(f_{n_j})_j$ so that $\mathbf G_{n_j}\equiv \mathbf G$. Since $\# \Theta=\nu+1$, by passing to a further subsequence, there is a $
  	\phi\in \Theta$ so that for  each $k\in \mathbb Z_\nu$ with $\g(k)\neq *$,
  	$${\rm Re}(e^{-i\phi} \tau_{k,\phi}(f_{n_j}))\rightarrow -\infty  \text{ as } j\rightarrow \infty.$$
  	The proof is completed.
  	\end{proof}
 
 %	\vspace{5pt}
% \noindent \textbf{Compact-open topology.}	

 \begin{pro}   [\cite{Ou99} Proposition 2.4.13; \cite{Ou02} Proposition 6.4] \label{cont-c}  Suppose for some   $\phi\in (-\pi/4, \pi/4)$, we have a sequence  $(f_n)_n$ in $\mathcal{WB}_\phi(\mathbf G)$ converging to   $f_0$,  and for each $k\in \mathbb Z_\nu$ with $\g(k)\neq *$,
 	$${\rm Re}(e^{-i\phi} \tau_{k,\phi}(f_{n}))\rightarrow -\infty  \text{ as } n\rightarrow \infty\footnote{This condition is necessary to guarantee the Hausdorff convergence of the  fundamental domains, see Figure 2.7 in \cite{Ou99}.}.$$  
 	Then for all $k\in \mathbb Z_\nu$ and $s\in \{\pm\}$,  the following hold.
 	
 	1. $S_{k, s, f_n, \phi}\rightarrow S_{k, s, f_0, \phi}$ in Hausdorff topology \footnote{A sequence of  compacta $(E_n)_n$ converges to a compactum $E$ in {\it Hausdorff topology}  if 
 		$d_{H}(E_n, E)\rightarrow 0$, where $d_H$ is the Hausdorff distance  defined by
 		$$d_H(A, B)=\max\Big\{ \max_{a\in A}\min_{b\in B} d(a,b),  \  \max_{b\in B}\min_{a\in A} d(a,b)\Big\},$$
 		and $d(a,b)$ is the Euclidean or spherical distance  depending on the situation.}.
 	
 	2.  We have the Hausdorff convergence 	%$f\mapsto \overline{U_{k, s, f, \phi}}$ is Hausdorff continuous on $\mathcal{WB}_\phi(\mathbf G)$.  
 	$$\overline{U_{k, s, f_n, \phi}}\rightarrow \begin{cases}\overline{U_{k, +, f_0, \phi}}\cup \overline{U_{j, -, f_0, \phi}} &\text{ if  } \g(k)=j, s=+; \\ 
 		\overline{U_{k, -, f_0, \phi}}\cup \overline{U_{j, +, f_0, \phi}} &\text{ if  } \g(j)=k, s=-;\\
 	\overline{U_{k, s, f_0, \phi}} &\text{ if  } \g(k)=*.  \end{cases}$$
 	
 %	2.  $f\mapsto {S_{k, s, f, \phi}}$ is Hausdorff continuous on $\mathcal{WB}_\phi(\mathbf G)$.
 	
 	3. $\Phi_{k, s, f_n, \phi}: U_{k, s, f_n, \phi}\rightarrow \mathbb C$ converges to 
 	$\Phi_{k, s, f_0, \phi}: U_{k, s, f_0, \phi}\rightarrow \mathbb C$  in the compact-open topology.
 	
 %	4. $f\mapsto (\Phi_{k, s, f, \phi}:  {\rm Int}(S_{k, s, f, \phi})\rightarrow \mathbb C)$   is   continuous on $\mathcal{WB}_\phi(\mathbf G)$ w.r.t.  compact-open topology.
 \end{pro}

	\vspace{5pt}
\noindent \textbf{Lifted Ecalle map.}  For each $k\in \mathbb Z_\nu$, the $k$-th {\it immediate parabolic basin} $\Omega^{(k)}_{f_0}$  of  $f_0$ at $0$ is   the connected component of $\bigcup_{l\geq0}f_0^{-l}(U_{k, +, f_0, \phi})$ containing  $U_{k, +, f_0, \phi}$.  %A  $\Omega^{(k)}_{f_0}\subset {\rm \mathcal D}(f_0)$ containing $U_{k, +, f_0, \phi}$ is called the $k$-th \emph{immediate parabolic basin} of  $(f_0,0)$.
The map $\Phi_{k, +, f_0, \phi}: U_{k, +, f_0, \phi}\rightarrow \mathbb C$  extends holomorphically to $\Omega^{(k)}_{f_0}$ by the equality $\Phi_{k, +, f_0, \phi}\circ f_0=\Phi_{k,+, f_0, \phi}+1$.
The  inverse  of $\Phi_{k, -, f_0, \phi}: U_{k, -, f_0, \phi}\rightarrow \mathbb C$ has a maximal analytic continuation $\Psi_{k, -, f_0, \phi}$ such that $\Psi_{k, -, f_0, \phi}({w}+1)=f_0\circ  \Psi_{k, -, f_0, \phi}(w)$ if one side is defined. 
For   $r\in \mathbb R$, define two half planes:
 $$\mathbb H^u(r):=\{ w\in \mathbb C; \ {\rm Im}(w)>r\}, \ \mathbb H^l(r):=\{ w\in \mathbb C; \ {\rm Im}(w)<r\}.$$

Note that  $\mathbb H^u(r_0)\subset \Psi_{k, -, f_0, \phi}^{-1}(\Omega^{(k)}_{f_0})$ and $\mathbb H^l(-r_0)\subset\Psi_{k, -, f_0, \phi}^{-1}(\Omega^{(k-1)}_{f_0})$ for some $r_0>0$.  Let  $\widetilde{B}^{(k,u)}_{f_0}$ be the component of $\Psi_{k, -, f_0, \phi}^{-1}(\Omega^{(k)}_{f_0})$ containing $\mathbb H^u(r_0)$, and  let $\widetilde{B}^{(k,l)}_{f_0}$ be the component of $\Psi_{k, -, f_0, \phi}^{-1}(\Omega^{(k-1)}_{f_0})$ containing $\mathbb H^l(-r_0)$. 
Clearly both $\widetilde{B}^{(k,u)}_{f_0}$ and $\widetilde{B}^{(k,l)}_{f_0}$ are invariant under the translation
 $T_1(w)=w+1$.
%\begin{itemize}
%	\item  
	%It is 
	%	\item  
	% It is invariant under the translation $T_1(\widetilde{B}^{(k,l)}_{f_0})=\widetilde{B}^{(k,l)}_{f_0}$.
%	\end{itemize}

We define two maps $\widetilde{\mathcal{E}}^{(k,u)}_{f_0}:  \widetilde{B}^{(k,u)}_{f_0}\to\mathbb C$ and $\widetilde{\mathcal{E}}^{(k,l)}_{f_0}:  \widetilde{B}^{(k,l)}_{f_0}\to\mathbb C$  by
$$\widetilde{\mathcal{E}}^{(k,u)}_{f_0}=\Phi_{k, +, f_0, \phi} \circ  \Psi_{k, -, f_0, \phi}, \ \widetilde{\mathcal{E}}^{(k,l)}_{f_0}=\Phi_{k-1, +, f_0, \phi} \circ  \Psi_{k, -, f_0, \phi}.$$
They satisfy that $\widetilde{\mathcal{E}}^{(k,*)}_{f_0}(w+1)=\widetilde{\mathcal{E}}^{(k,*)}_{f_0}(w)+1$ and $\widetilde{\mathcal{E}}^{(k,*)}_{f_0}(w)=w+c_{k,*}+o(1)$ when $|{\rm Im }(w)|$ is large, where $c_{k,*}\in \mathbb C$ is a constant, $*\in \{u, l\}$.
%It holds that $\t{\EEE}^{(k,u)}_{f_0}(\t{w}+1)=\t{\EEE}_{f_0}(\t{w})+1$ for $\t{w}\in\t{B}^{(k,u)}_{f_0}$. Hence one obtains a well-defined map
%\[\EEE_{f_0}^{(k,u)}=\pi\circ\t{\EEE}_{f_0}^{(k,u)}\circ \pi^{-1}:B^{(k,u)}_{f_0}\to \mathbb C^*\]
%via $\pi(\t{w}):=e^{2\pi i\t{w}}$ with $B^{(k,u)}_{f_0}:=\pi(\t{B}^{(k,u)}_{f_0})$. Moreover this map extends to $0$ holomorphically by defining $\EEE_{f_0}^{(k,u)}(0)=0$.

The maps $\widetilde{\mathcal{E}}^{(k,u)}_{f_0}$   and $\widetilde{\mathcal{E}}^{(k,l)}_{f_0}$ are called the {\it lifted Ecalle maps}.
% or {\it horn maps}.

  If  $f\in \mathcal{WB}_\phi(\mathbf G)$  is sufficiently close to  $f_0$, we can define $\widetilde{\mathcal{E}}^{(k,u)}_{f}$   and $\widetilde{\mathcal{E}}^{(k,l)}_{f}$ in the same way.  Let $w=\Phi_{k, -, f, \phi}(z)$ for $z\in U_{k,-, f, \phi}$.
  If ${\rm Im}(w)>\eta$ (resp. ${\rm Im}(w)<-\eta$) for some large $\eta>0$, then we can find a minimal   integer $m\geq 0$  so that $f^m(z)\in U_{k, +, f, \phi}$  (resp.  $f^m(z)\in U_{k-1, +, f, \phi}$), define 
  $$\E_{f}^{(k,*)}(w):=
   \begin{cases}  \Phi_{k, +, f, \phi}(f^m(z))-m, &\text{ if  } *=u;\\
  	  \Phi_{k-1, +, f, \phi}(f^m(z))-m, &\text{ if  } *=l.
  \end{cases}$$
 
It satisfies $\E_{f}^{(k,*)}(w+1)=\E_{f}^{(k,*)}(w)+1$. Hence $\E_{f}^{(k,u)}$ can extend to the entire half-plane $\mathbb H^u(\eta)$,   and $\E_{f}^{(k,l)}$ can extend to the entire half-plane $\mathbb H^l(-\eta)$. They have the expression $\widetilde{\mathcal{E}}^{(k,*)}_{f}(w)=w+c_{k,*}(f)+o(1)$ when $|{\rm Im}(w)|$ is large, where $c_{k,*}(f)\in \mathbb C$ is a constant.

By Proposition \ref{cont-c}  and using the same argument as that of \cite[Lemma 3.4.5]{Ou99},  we get
 
  \begin{pro}    \label{ecalle0}
  	Suppose for some   $\phi\in (-\pi/4, \pi/4)$, we have a sequence  $(f_n)_n$ in $\mathcal{WB}_\phi(\mathbf G)$ converging to   $f_0$,  and for each $k\in \mathbb Z_\nu$ with $\g(k)\neq *$,
  	$${\rm Re}(e^{-i\phi} \tau_{k,\phi}(f_{n}))\rightarrow -\infty  \text{ as } n\rightarrow \infty.$$

  	 % Let $k\in \mathbb Z_\nu$ and let $\mathbf G$ be an admissible gate vector.   
  	    Then the lifted Ecalle map $\E_{f}^{(k,*)}$ converges uniformly to $\E_{f_0}^{(k,*)}$     in  $\mathbb H^u(\eta)$ or $\mathbb H^l(-\eta)$  for large $\eta>0$.
\end{pro}

 \section{Controlling orbits for perturbations} \label{h-e}

 % $v'$ lies further than $v'$ to $v_0$.

In this section,  let $f_0$ be a rational map with a parabolic periodic point $\zeta$.   Let   $(f_n)_n$ be a generic perturbation of $f_0$ at $\zeta$. 
 Replacing  $f_0$ and $f_n$ by some iterate, we assume  $f_0(\zeta)=\zeta$,  $f'_0(\zeta)=1$ and the multiplicity of $f_0$ at   $\zeta$ is  $\nu+1\geq 2$.   
% Since  $f_n$ is a generic perturbation  of $f$, 
  Then the $f_0$-fixed point $\zeta$ splits into $\nu+1$ distinct fixed points of $f_n$, one of which is denoted by $\zeta_n$. By changing coordinate, we assume $\zeta_n=\zeta=0$ and $f_0(z)=z+z^{\nu+1}+O(z^{\nu+2})$ near $0$.
 
 %(1).    an attracting fixed point   $\zeta_n$ and $\nu$ repelling  fixed  points  of $f_n$, or 
 
 %(2).    a repelling  fixed  point   $\zeta_n$ and $\nu$ attracting  fixed  points  of $f_n$.
 
 %By choosing subsequences, we assume  case (1) holds for all $n$ (the argument below also works for case (2)).   
  
  %The situation that $\zeta$  splits into one repelling point say $\zeta_n$ and $\nu$ attracting points  for $f_n$  follows from the same argument.
  
For each $k\in \mathbb Z_\nu$,  let $A_k$ be the immediate parabolic basin of $f_0$ at $0$,  associated with the attracting direction  $\epsilon_{k,+}=e^{\pi i (2k-1)/\nu}$.    
 
By Proposition \ref{sp} and passing to a subsequence, % if necessary,  
we assume $f_n\in \mathcal{WB}_\phi(\mathbf G)$ for some $\phi, \mathbf G$ and for all $n\geq 1$.

 % To simplify the notations, we may drop the subscript $\phi$ in the corresponding notations.

%Passing to a subsequence, we  assume $v_0\in V$ is the vertex so that $v_0(f_n)$ is  $f_n$-attracting for  all $n$. 
%Let $\prec$  be a partial order on $V$ (with respect to $v_0$).

Since $(f_n)_n$ is a generic perturbation of $f_0$ at $0$,  there are no closed gate. Hence the gate  vector $\mathbf G=(\g(k))_{k\in \mathbb Z_\nu}$ can be viewed as a bijective self-map $\g: \mathbb Z_\nu\rightarrow \mathbb Z_\nu$.  Set $\widehat{\g}=\g-1:  \mathbb Z_\nu\rightarrow \mathbb Z_\nu$.

\vspace{5pt}
\noindent \textbf{Tree of gate structure.}  The gate  structure $\mathbf G=(\g(k))_{k\in \mathbb Z_\nu}$   can be represented by a finite set of disjoint directed hyperbolic geodesics in $\mathbb{D}$, constructed as follows:
%can be represented by finitely many disjoint hyperbolic geodesics of $\mathbb D$ as follows:  
for each $k\in \mathbb Z_{\nu}$, there is a directed hyperbolic geodesic    $\ell_k$ from $\epsilon_{k,+}=e^{\pi i (2k-1)/\nu}$ to $\epsilon_{j,-}=e^{2\pi i (j-1)/\nu}$, where  $j=\g(k)$.  
% if $\mathbf G_k=*$, set $\ell_k(\mathbf G)=\emptyset$.

The tree $T = (V, E)$ associated with $\mathbf{G}$ is defined as follows:
\begin{itemize}
	\item The vertex set $V$ corresponds to the connected components of $\mathbb{D} \setminus \bigcup_{k \in \mathbb{Z}_\nu} \ell_k$. For  $v \in V$, let $X_v$ denote the corresponding component.
	\item The edge set $E$  corresponds to the geodesics $\ell_k$: an edge $e = [v, v']$ exists if and only if $\partial X_v \cap \partial X_{v'} = \ell_k$ for some $k$.
\end{itemize}

Note that $V$ consists of $\nu + 1$ vertices, 
 and  $E$ consists of $\nu$ edges.
For each $e=[v,v']\in E$,   there is a unique
$\boldsymbol k(e)\in \mathbb Z_\nu$  so that $\partial X_v\cap \partial X_{v'}=\ell_{\boldsymbol k(e)}$.
This gives a bijection $\boldsymbol k: E\rightarrow \mathbb Z_\nu$.
The inverse of  $\boldsymbol k$ is denoted by $\boldsymbol e: \mathbb Z_\nu  \rightarrow  E$.

%The tree $T=(V, E)$ of   $\mathbf G$ consists of a vertex set $V$ and an edge set $E$, such that 
%\begin{itemize}
%	\item  each vertex $v\in V$ corresponds to a % connected 
%	component $X_v$ of $\mathbb D\setminus \bigcup_{k\in \mathbb Z_\nu} \ell_k$;
	
%	\item each edge $e=[v,v']$  corresponds to a geodesic  $\ell_k$ if $\partial X_v\cap \partial X_{v'}=\ell_k$.
%\end{itemize} 
%Note that $V$ consists of $\# V=\nu+1$ vertices.  

	\begin{figure}[h]
	\begin{center}
		\includegraphics[height=4cm]{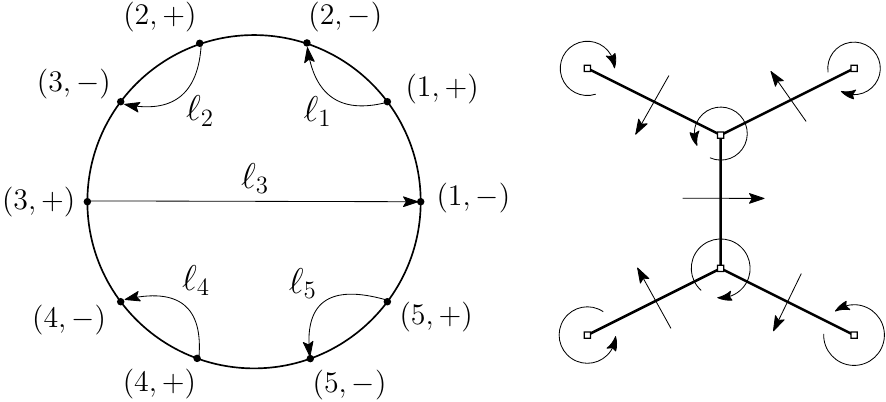}
	\end{center}
	\caption{Tree $T$ of the gate structure $\mathbf G$ in Figure \ref{fig: gate1}.
		Left: $\mathbf G$ is represented by  disjoint directed hyperbolic geodesics. Right:  the tree $T$ with orientations of edges and vertices. }
	\label{fig: tree-gate}
\end{figure}

For each edge $e\in E$, the directed geodesic $\ell_k$ induces a transversal  direction of $e$. 
%\textcolor{red}{ Let $v^{+}(e)$ (resp. $v^{-}(e)$) be the vertex of $e$ on the left (resp.  right) with respect to the direction of $e$.}
For each vertex $v\in V$, the  transversal  directions of its adjacent edges give  a natural orientation $\mathcal O(v)$ of $v$, which is either  positive cyclic order or negative cyclic order. See Figure \ref{fig: tree-gate}.
% whose  signature $\chi(v)\in \{\pm 1\}$  is
%$$\chi(v)=\begin{cases} +1 &\text{ if } \mathcal O(v) \text{ is positive cyclic order}; \\
%-1 &\text{ if } \mathcal O(v) \text{ is negative cyclic order}.
%\end{cases}$$

Since  the maps $f_n$'s  share the same gate structure,   we can mark the $\nu+1$ fixed points of $f_n$ near $0$ as $v(f_n), v \in V$.
%the $\nu+1$ fixed points of $f_n$ near $0$ can be labeled by the vertices $v \in V$, denoted as $v(f_n), v\in V$.
  The multiplier of $f_n$ at $v(f_n)$ is denoted by $\lambda_v(f_n)=f_n'(v(f_n))$.
% \textcolor{red}{ Clearly, for each $e=[v,v']\in E$,
%$$v(f_n), v'(f_n)\in \partial U_{\boldsymbol k(e), +, f_n, \phi}, \ \forall n\geq 1.$$}

\begin{lem} \label{asm-iota}
	For large $n$, 
	\begin{equation}\label{lim-iota}\sum_{v\in V} \iota(f_n,  v(f_n))=\iota(f_0,  0)+o(1). 	\end{equation}
	\end{lem}
\begin{proof} Since $f_n$ converges to $f_0$  uniformly in $\overline{\mathbb D(0, r_0)}$, we have % By the convergence 
	$$\frac{1}{2\pi i }\int_{|z|=r_0}\frac{dz}{z-f_n(z)}\rightarrow \frac{1}{2\pi i }\int_{|z|=r_0}\frac{dz}{z-f_0(z)}.$$
It follows that $\sum_{v\in V} \iota(f_n,  v(f_n))\rightarrow \iota(f_0,  0)$, implying \eqref{lim-iota}.
%	\begin{equation}\label{lim-iota}\sum_{v\in V} \iota(f_n,  v(f_n))\rightarrow \iota(f_0,  0).
%	\end{equation}
%The lemma follows immediately.
	\end{proof}

In the following, for simplicity, we omit the subscript $\phi$ from the notation and write $\tau_{k, \phi}(f)$, $U_{k, \pm, f, \phi}$, and $\Phi_{k, \pm, f, \phi}$ as $\tau_{k}(f)$, $U_{k, \pm, f}$, and $\Phi_{k, \pm, f}$, respectively.

%In the following,  for simplicity, we drop the subscript $\phi$ from the notations, and write $\tau_{k, \phi}(f), U_{k, \pm, f, \phi}, \Phi_{k, \pm, f, \phi}$ as $\tau_{k}(f), U_{k, \pm, f}, \Phi_{k, \pm, f}$ respectively. 

\vspace{5pt}
\noindent \textbf{Oudkerk's Algorithm \cite[\S 13]{Ou02}.}  
Fix $k\in \mathbb Z_\nu$, and apply the following algorithm which produces a (possibly infinite) sequence 
$a_1, a_2, \cdots$ in $\mathbb Z_\nu$.

(1). Let $a_1:=k$ and $r:=1$.

(2). Replace $(f_n)_n$ by a subsequence satisfying one of the following:

\ \ \ \ \ (a). ${\rm Im}\sum_{j=1}^r \tau_{a_j}(f_n)$  is bounded as $n\rightarrow +\infty$;

\ \ \ \ \ (b).  ${\rm Im}\sum_{j=1}^r \tau_{a_j}(f_n)\rightarrow +\infty$ as   $n\rightarrow +\infty$;

\ \ \ \ \ (c).  ${\rm Im}\sum_{j=1}^r \tau_{a_j}(f_n)\rightarrow -\infty$ as   $n\rightarrow +\infty$.

(3). If we have case (a), then EXIT  the algorithm.

(4). If we have cases (b) or (c), then set 
$$a_{r+1}:=\begin{cases} \g(a_r) &\text{ for  case (b)}; \\
 \widehat{\g}(a_r) &\text{  for  case (c)}.
\end{cases}$$

%(5). Replace $r$ by $r+1$ and 

(5). GO BACK to step (2).

\begin{pro} \label{a-preperiodic}
	If   the algorithm does not  terminate at step (3), then the infinite sequence $(a_j)_{j\geq 1}$ is  
	preperiodic: there are integers $m\geq 0$ and $p\geq 1$ so that $a_{m+p+j}=a_{m+j}$ for all $j\geq 1$.
	
	Further, if $m$ and $p$ are minimal, then
	\begin{itemize}
		\item 	the edges  $\boldsymbol e(a_{m+1}), \cdots, \boldsymbol e(a_{m+p})$ share a common  vertex  say $v$,  and these edges appear  sequentially  in the orientation  $\mathcal O(v)$ of $v$;
		
		\item the  number of the edges   incident to $v$   is precisely $p$, and 
		
		\item ${\rm Re } \  \iota(f_n, v(f_n))\rightarrow +\infty  \text{ as } n\rightarrow +\infty$ \footnote{This implies that $v(f_n)$ is $f_n$-attracting for large $n$ and $\lambda_v(f_n)\rightarrow 1$ horocyclically.}.
		\end{itemize}
	\end{pro}
\begin{proof} First note that for any $j\geq 1$, the edges $\boldsymbol e(a_j)$ and $\boldsymbol e(a_{j+1})$ share a common vertex.
		Since the sequence $(a_j)_{j\geq 1}$ takes values in the finite set $\mathbb Z_\nu$, there is a minimal  integer $l\geq 1$ so that 
	$a_{l+1}\in \{a_1, \cdots, a_l\}$.   Suppose $a_{l+1}=a_t$ for some $1\leq t\leq l$.
	We claim that the edges  $\boldsymbol e(a_t), \cdots, \boldsymbol e(a_l)$ share a common vertex say $v$. %There is nothing to prove if $l-s=0, 1$, so we assume
	If not, then  $l\geq t+2$. Since   $\boldsymbol e(a_t)=\boldsymbol e(a_{l+1})$ and $e(a_l)$ share a common vertex,  the union of the edges  $\boldsymbol e(a_t), \cdots, \boldsymbol e(a_l)$  would form  a loop in the tree $T$, giving a contradiction. See Figure \ref{fig:loop}.

	\begin{figure}[h]
		\begin{center}
			\includegraphics[height=3.5cm]{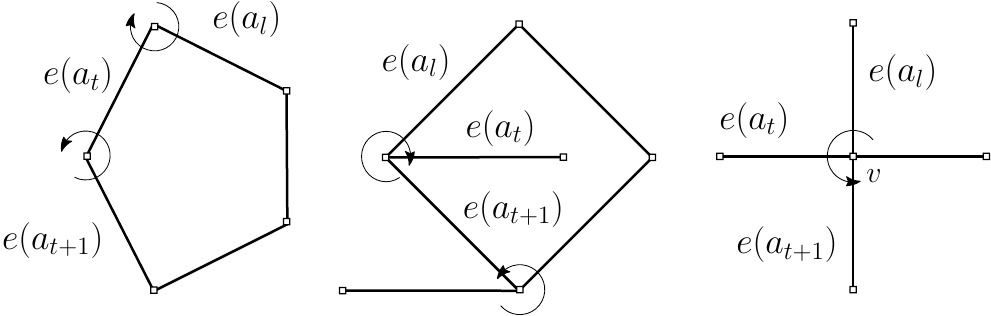}
		\end{center}
		\caption{Left two: the edges  $\boldsymbol e(a_t), \cdots, \boldsymbol e(a_l)$  form at least one loop. Right: the edges  $\boldsymbol e(a_t), \cdots, \boldsymbol e(a_l)$ share a common vertex $v$. }
		\label{fig:loop}
	\end{figure}

%	Let $v_r$ be the common vertex of $\boldsymbol e(a_{r})$ and $\boldsymbol e(a_{r+1})$ for 
%	$s\leq r\leq l$. If there are at least two different vertices in $v_s, \cdots v_{l}$, then  % If the vertices $v_s, \cdots v_{l}$ are pariwise different,

%Let $v$ be the common vertex of the edges  $\boldsymbol e(a_t), \cdots, \boldsymbol e(a_l)$.

	For each $j\geq 1$,  let 
	\begin{equation} \label{delta-j}
	\Delta_j(f_n):={\rm Im}\sum_{r=1}^j \tau_{a_r}(f_n) \in \mathbb R,
	\end{equation}
	and
	\begin{equation} \label{s-j} 
		s_j=\begin{cases}
 +1,  &\text{ if }  \Delta_j(f_n)\rightarrow +\infty \text{ as } n\rightarrow +\infty; \\
	-1,  & \text{ if }  \Delta_j(f_n)\rightarrow -\infty  \text{ as } n\rightarrow +\infty.
\end{cases}
	\end{equation} 
Clearly $s_j\Delta_j(f_n)\rightarrow +\infty$ for any $j\geq 1$. The properties that $a_{l+1}=a_t$ and  $\boldsymbol e(a_t), \cdots, \boldsymbol e(a_l)$ share a common vertex  imply that  $s_t=s_{t+1}=\cdots=s_{l}$.

If $t=1$, note that $a_{l+1}=a_1$ and $s_1=s_l$, we get
$$s_l  \Delta_{l+1}(f_n)= s_1 \Delta_{1}(f_n)+ s_l \Delta_{l}(f_n) \rightarrow +\infty.$$
This implies that $s_{l+1}=s_l$ and $a_{l+2}=a_2$. Applying the same argument inductively, we have  $a_{l+j}=a_j$ for all $j\geq 1$. 

If $t>1$, we first note that  {$s_{t-1}\neq s_t$} (if not, then $a_l=a_{t-1}$, contradicting the  minimality of $l$), which means that  $s_{t-1}=-s_t$. By the facts
\bess
s_{t-1}\Delta_{t-1}(f_n)&\rightarrow&+\infty,\\
s_t\Delta_{t}(f_n)=s_t{\rm Im}\tau_{a_t}(f_n)-s_{t-1}\Delta_{t-1}(f_n)  &\rightarrow& +\infty,
\eess
 we get $s_t{\rm Im}\tau_{a_t}(f_n)\rightarrow +\infty$. 
 Since $a_{l+1}=a_t$ and $s_t=s_l$, we have
 $$s_l  \Delta_{l+1}(f_n)= s_l \Delta_{l}(f_n)+ s_t {\rm Im}\tau_{a_t}(f_n)\rightarrow +\infty.$$ 
 This implies that $s_{l+1}=s_l$ and $a_{l+2}=a_{t+1}$. Applying the same argument inductively, we have  $s_{l+j}=s_l$ and  $a_{l+j}=a_{t-1+j}$ for all $j\geq 1$.  
 
 Note that  the edges  $\boldsymbol e(a_{t}), \cdots, \boldsymbol e(a_{l})$ appear  sequentially  in the orientation  $\mathcal O(v)$ of $v$, and they are all edges  incident to vertex $v$.

The above proof also implies that 
$$s_l {\rm Im}\sum_{r=t}^l \tau_{a_r}(f_n)=\begin{cases}  s_l  \Delta_{l}(f_n) \rightarrow +\infty &\text{ if }t=1;\\
	s_l  \Delta_{l}(f_n)+ s_{t-1}\Delta_{t-1}(f_n) \rightarrow +\infty &\text{ if }t>1.
\end{cases}$$

By  Proposition \ref{phase} and Lemma \ref{asm-iota},  
	\begin{equation}\label{sum-e}\sum_{r=t}^l \tau_{a_r}(f_n)= 2\pi  i s_l\cdot \iota(f_n,  v(f_n))+O(1).
			\end{equation}
It follows that ${\rm Re } \  \iota(f_n, v(f_n))\rightarrow +\infty$ as $n\rightarrow +\infty$.

The proof is completed by taking  $(m,p)=(t-1, l-t+1)$.
	\end{proof}

%	\begin{de}[Lavaurs map] For any $k,l\in \mathbb Z_\nu$ and any $\sigma\in \mathbb C$, the Lavaurs map $g_{\sigma}: A_k\rightarrow \mathbb C$ is defined by
%		$$g_{\sigma}:=g_{\sigma}^{(k,l)}= \Phi_{l, -, f, \phi}^{-1} \circ T_{\sigma}\circ \Phi_{k, +, f, \phi} $$
	%	\end{de}

\begin{pro} \label{ak-limit1}  Let $k\in \mathbb Z_\nu$ and let $E\subset U_{k,+, f_0} $ be  compact.  % We have the  dichotomy:
 {If the algorithm terminates at step (3), then any subsequence of $(f_n)_n$ admits a further subsequence $(f_{j_n})_n$  and a sequence of integers $(N_{n})_{n\geq 1}$ so that 
$f_{j_n}^{N_{n}}$ converges uniformly on $E$ to the Lavaurs map  
$$g_{\sigma}=\Phi_{l, -, f_0}^{-1} \circ T_{\sigma}\circ \Phi_{k, +, f_0} : A_k\rightarrow \mathbb C,$$ for some $l\in \mathbb Z_\nu$, where $T_\sigma(w)=w+\sigma$. This implies 
$$X\cap A_k\neq\emptyset, \ \forall X\in \mathcal{L}((J(f_{n}))_n).$$}
%	$$A_k\cap \liminf_{n\rightarrow \infty} J(f_n)\neq \emptyset.$$
\end{pro}

\begin{proof}  The argument here is due to Oudkerk \cite[\S 13.2]{Ou02}, for the readers' convenience, we include it here. Since the algorithm terminates at step (3), we get a finite sequence $a_1, \cdots, a_\ell$ for some $\ell\geq 1$.  
	For $1\leq j<\ell $, let $s_j$ be given by \eqref{s-j}.
%	$$s_j=\begin{cases} +1,  &\text{ if } {\rm Im}\sum_{l=1}^j \tau_{a_l}(f_n)\rightarrow +\infty; \\
	%	-1,  & \text{ if } {\rm Im}\sum_{l=1}^j \tau_{a_l}(f_n)\rightarrow -\infty.
	%\end{cases}$$
%	In the following, we drop the subscript $\phi$ in the notations.
	Note that   $U_{a_j, +, f_n}=U_{\g(a_j), -, f_n}$.  Fix  $R>0$, for $1\leq j<\ell$ and $f\in \{f_n, f_0\}$,  let 
	\begin{equation}\label{wr} W_{j, f}(R)=\{z\in U_{\g(a_j), -, f}; \ s_j {\rm Im } \Phi_{\g(a_j), -,  f}(z)> R\}.
		\end{equation}
%	For any $\epsilon>0$, there is an integer $N$ so that $ v(f_n)\in \mathbb D(0, \epsilon)$, $\forall v\in V$ and $n\geq N$.   We choose $R$ large enough so that for $n\geq N$, 
%	$$W_{j, f_n} \subset \begin{cases} \mathbb D(v_{a_j}^+(f_n), \epsilon) &\text{ if } s_j=+1; \\
%		\mathbb D(v_{a_j}^-(f_n), \epsilon) &\text{ if } s_j=-1.
%	\end{cases}$$

	\begin{figure}[h]
	\begin{center}
		\includegraphics[height=3cm]{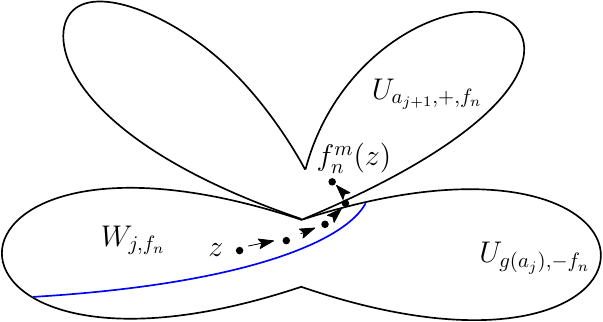}
	\end{center}
	\caption{An orbit crossing two adjacent double petals. }
	\label{fig: algo-finite}
\end{figure}
	
	 For any $1\leq j<\ell$ and any $z\in W_{j, f_n}(R)$, when $R$ is large enough (independent of $j$),
	 % for $w=\Phi_{\g(a_j), - f_n}(z)$, let
	 %$$\E_{j, f_n}(w)=\begin{cases} \widetilde{\mathcal{E}}^{(k,u)}_{f_n}(w),  &\text{ if }t=1;\\
	 %	\widetilde{\mathcal{E}}^{(k,l)}_{f_n}(w), &\text{ if }t>1.
	 %\end{cases}$$
	 there is an integer $r>0$ so that $f_n^r(z)\in U_{a_{j+1}, +, f_n}$ and $f_n^l(z)\in K_0$ for $0\leq l\leq r$.  See Figure \ref{fig: algo-finite}. For $w=\Phi_{\g(a_j), -, f_n}(z)$, we define
	\begin{equation} \label{ecalle}
	\E_{j, f_n}(w):=\Phi_{a_{j+1}, +, f_n}(f_n^r(z))-r.
	\end{equation}
	It satisfies $\E_{j, f_n}(w+1)=\E_{j, f_n}(w)+1$ and hence extends to the whole half-plane $\{w\in \mathbb C; s_j {\rm Im}(w)>R\}$.   {By the discussion before Proposition \ref{ecalle0}, 
  $\E_{j, f_n}(w)-T_{C_{j, n}}(w)\rightarrow 0$ uniformly as $s_j {\rm Im}(w)\rightarrow +\infty$}, where   $C_{j, n}$ is a constant. It satisfies $C_{j, n}\rightarrow C_j$ as $n\rightarrow +\infty$.
  
  By the same way,  we can define $\E_{j, f_0}$.

For the  compact set 
  $E\subset U_{k,+, f_0}$, by Proposition \ref{cont-c}, we have $E\subset U_{k,+, f_n}$ for large $n$. 
	Since $s_1  {\rm Im} \ \tau_{a_1}(f_n)\rightarrow +\infty$ as $n\rightarrow +\infty$, we have 
	$$s_1 {\rm Im } \ \Phi_{\g(a_1), -,  f_n}(E)=s_1  {\rm Im} \ \tau_{a_1}(f_n)+s_1 {\rm Im }\ \Phi_{a_1, -,  f_n}(E)\rightarrow +\infty.$$
This implies that	
	$E\subset W_{1, f_n}(R)$ for large $n$.  	We claim that for any large $n$, % there is a integer $N_0$ such that for each $n\geq N_0$, 
	there is a least integer $e_{1, n}\geq 1$ so that 
	\begin{equation}\label{orbit1}
	f_n^{e_{1,n}}(E)\subset U_{a_2, +, f_n} \text{ and }  f_n^{j}(E)\subset K_0, \  \forall \ 0\leq j\leq e_{1,n}.
	\end{equation}
	If not, for any large $n$, there would be an integer  $L_n$ so that $f_n^{L_n}(E)$  intersects both  $\overline{\ell_{a_2, +, f, \phi}}$ and $\overline{\ell_{\g(a_2), -, f, \phi}}$, which form the    boundary  $\partial U_{a_2, +, f_n}$.  Note that
	$$P_n:=T_{-L_n}\circ \Phi_{a_2, +, f_n} (f_n^{L_n}(E) \cap  U_{a_2, +, f_n})\subset \E_{1, f_n}\circ \Phi_{\g(a_1), -, f_n}(E).$$
 By Proposition \ref{ecalle0}, the right side tends to $\E_{1, f_0}\circ \Phi_{\g(a_1), -, f_0}(E)$ as $n\rightarrow +\infty$,  hence $P_n$ has uniformly bounded width.  {However  the width of $\Phi_{a_2, +, f_n}(U_{a_2, +, f_n})$ tends to $+\infty$  since ${\rm Re} \ e^{-i\phi} \tau_{a_2}(f_n) \rightarrow -\infty$  as $n\rightarrow \infty$ (guaranteed by Proposition \ref{sp})}.   Contradiction.
	
	This proves \eqref{orbit1}.
Therefore we have the equality
%	$$T_{-m_{1,n}}\circ \Phi_{a_2, +, f_n} \circ f_n^{m_{1,n}}=\E_{1, f_n}\circ   \Phi_{\g(a_1), -, f_n}=\E_{1, f_n}\circ  T_{\tau_{a_1}(f_n)}\circ \Phi_{a_1, +, f_n},$$
%	we get 
	\bess
	\Phi_{\g(a_2), -, f_n}\circ f_n^{e_{1, n}} &=&T_{\tau_{a_2}(f_n)}\circ  \Phi_{a_2, +, f_n} \circ f_n^{e_{1,n}}\\
	&=& T_{\tau_{a_2}(f_n)}\circ T_{e_{1,n}}\circ  \E_{1, f_n}\circ  T_{\tau_{a_1}(f_n)}\circ \Phi_{a_1, +, f_n}.
	\eess
	It follows that
	$$s_2 {\rm Im }  \Phi_{\g(a_2), -, f_n}\circ f_n^{e_{1, n}}|_E=s_2 {\rm Im} (\tau_{a_1}(f_n)+\tau_{a_2}(f_n)) +O(1)\rightarrow +\infty$$
	uniformly on the compact set $E$, hence $f_n^{e_{1, n}}(E)\subset W_{2, f_n}(R)$ for large $n$.
	
	%(there is $L$ large enough  independent of large $n$)

	By induction, there is a large integer $N_\ell>0$ such that for any $n\geq N_\ell$, there are   integers $0< e_{1,n}< \cdots< e_{\ell-1, n}$ with properties:  
	 \begin{itemize}
		\item for $1\leq j< \ell$,  $e_{j, n}$ is the  
		least integer with $f_n^{e_{j,n}}(E)\subset  U_{a_{j+1}, +, f_n}$
		(by choosing $n\geq N_\ell$, this implies  that $f_n^{e_{j,n}}(E)\subset W_{j+1, f_n}(R)$);
		
		\item  $f_n^{j}(E)\subset K_0, \  \forall \ 0\leq j\leq e_{\ell-1,n}$.
	\end{itemize} 

%For any  compact set $E$ in the $k$-th attracting basin $A_k$, we can  find integers 
%$0<m_{n,1}<m_{n,2}<\cdots < m_{n, \ell}$ so that $E\subset W_{0, f_n}$ and $f_n^{m_{n,j}}(E)\subset W_{j, f_n}$ for $1\leq j\leq \ell-1$.

For each $1\leq j\leq \ell-1$, let $q_{j,n}>e_{j, n} $ be the maximal integer so that 
$$f_n^{l}(E)\subset 
 U_{a_{j+1}, +, f_n}, \ \forall \ e_{j,n} \leq l\leq q_{j, n}.$$
%Then for $m_{n, \ell-1} \leq m\leq M_{n, \ell-1}$, 
Then on $E$, %$\Phi_{\g(a_\ell), -, f_n}\circ f_n^{M_n}$ 
one may verify that for  each $1\leq j\leq \ell-1$ and $e_{j,n} \leq l\leq q_{j, n}$, 

\begin{equation} \label{l-equality}  \Phi_{\g(a_{j+1}), -, f_n}\circ f_n^l 
=T_{l}\circ  T_{\tau_{a_{j+1}}(f_n)}\circ  \E_{j, f_n}\circ \cdots\circ  T_{\tau_{a_{2}}(f_n)}\circ   \E_{1, f_n}\circ T_{\tau_{a_1}(f_n)}\circ \Phi_{a_1, +, f_n}.
	\end{equation}
%\bess 
%&&\Phi_{\g(a_{j+1}), -, f_n}\circ f_n^l\\
%&=&T_{l}\circ  T_{\tau_{a_{j+1}}(f_n)}\circ  \E_{j, f_n}\circ \cdots\circ  T_{\tau_{a_{2}}(f_n)}\circ   \E_{1, f_n}\circ T_{\tau_{a_1}(f_n)}\circ \Phi_{a_1, +, f_n}.
%\eess
By the choice of $q_{\ell-1, n}$, $Y_n:= \Phi_{\g(a_\ell), -, f_n}\circ  f_n^{q_{\ell-1,n}}(E)$ is contained in the strip  $\Phi_{\g(a_\ell), -, f_n}(U_{a_{\ell}, +, f_n})$, and $Y_n+1$ intersects the right boundary of this  strip.  

By Proposition \ref{cont-c}, the set  
 $\Phi_{a_1, +, f_n}(E)$ is bounded  as $n\rightarrow \infty$. The imaginary part of the  right side of above equality is 
 $$ {\rm Im}\sum_{j=1}^\ell \tau_{a_j}(f_n)+O(1),$$
hence is bounded.  Note that 
$e^{-i\phi} \Phi_{\g(a_\ell), -, f_n}(U_{a_{\ell}, +, f_n})$ is almost a vertical strip, which has $0$ on its right boundary, we conclude that
  $Y_n$ is also bounded for large $n$.
It follows that $q_{\ell-1,n}+\sum_{j=1}^\ell \tau_{a_j}(f_n)$ is bounded as $n\rightarrow \infty$. By choosing a subsequence, we assume 
$$q_{\ell-1,n}+\sum_{j=1}^\ell \tau_{a_j}(f_n)+\sum_{j=1}^{\ell-1}C_{j,n}\rightarrow \sigma\in \mathbb C.$$
This   yields the uniform convergence 
$\Phi_{\g(a_\ell), -, f_n}\circ f_n^{q_{\ell-1, n}}\rightarrow T_\sigma \circ \Phi_{a_1, +, f_0}$ on $E$. 
Hence $ f_n^{q_{\ell-1, n}}\rightarrow g_\sigma:=\Phi_{\g(a_\ell), -, f_0}^{-1}\circ  T_\sigma \circ \Phi_{a_1, +, f_0}$ on $E$. It is clear that
  $g_\sigma$ can extend to $A_k$ by the equality $g_\sigma \circ f_0=f_0\circ g_\sigma$.
  % and is called the Lavaurs map. 

By the density of   repelling periodic points on Julia set and their stability, we get   
$X\supset J(f_0)$ for any $X\in \mathcal{L}((J(f_{n}))_n)$.	
It follows that 
$$A_k\cap X
%\supsetneq \liminf_{n\rightarrow \infty}(f_n^{M_n})^{-1}(J(f))\cap A
 \supsetneq  g_\sigma^{-1}(J(f_0))\cap A_k\neq \emptyset.$$
The proof is completed.  \end{proof}

\begin{pro} \label{ak-limit2}  Let $k\in \mathbb Z_\nu$ and let $E$ be a compact subset of $U_{k,+, f_0}$.   
	If the algorithm never terminates, then when $n$ is large,  the sequence $(f_n^{j})_{j\geq 1}$ converges uniformly on $E$ to the $f_n$-attracting point $v(f_n)$, where $v$ is given by Proposition \ref{a-preperiodic}.  This implies 
	$$A_k\cap X= \emptyset, \ \forall X\in  \mathcal{L}((J(f_{n}))_n).$$
\end{pro}

  Since the algorithm never terminates at step (3),   we get an infinite sequence $(a_j)_{j\geq 1}$.
 By Proposition \ref{a-preperiodic}, there are minimal   
   integers $m\geq 0$ and $p\geq 1$ so that $a_{m+p+j}=a_{m+j}$ for all $j\geq 1$.
   Let $(s_j)_{j\geq 1}$ be the sequence given by \eqref{s-j}. 
   By the proof of Proposition  \ref{a-preperiodic}, 
   $$s_m=-s_{m+1};  \ s_l= s_{m+1}:=s_*,   \forall  l \geq m+1.$$
  
 For each $l\geq 2$, and each   $R>0$, let
 $$Y_{l,g}(R)=\{z\in U_{a_{l}, +, g}; \ s_{l-1} {\rm Im } \Phi_{a_{l}, +, g}(z)> R\}, \ g\in \{f_n, f_0\},$$ 
 and set $X_{l, f_n}(R):=Y_{l, f_n}(R)\cap W_{l, f_n}(R)$, where $W_{l, f_n}(R)$ is  given by \eqref{wr}.

 \begin{figure}[h]
 	\begin{center}
 		\includegraphics[height=4.8cm]{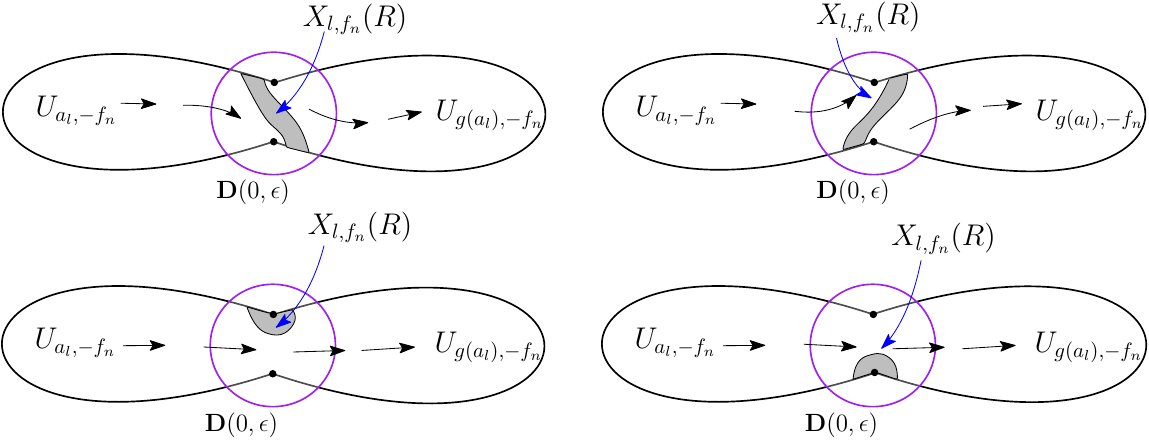}
 	\end{center}
 	\caption{The set $X_{l, f_n}(R)$ with $(s_{l-1}, s_l)=(1,-1)$ (UL), $(-1,1)$ (UR),
 		$(1,1)$ (LL) and $(-1,-1)$ (LR). }
 	\label{fig: curve-bound}
 \end{figure}
% $$X_{l, f_n}(R)=\{z\in U_{a_{l}, +, f_n}; \ s_{l-1} {\rm Im } \Phi_{a_{l}, +, f_n}(z)> R \text{ and } s_{l} {\rm Im } \Phi_{\g(a_{l}), -, f_n}(z)> R\}.$$

\begin{lem}   \label{hausdorff-limit} For any $l\geq2$ and $R>0$, the Hausdorff limit of any convergent subsequence  of $(\overline{X}_{l, f_n}(R))_{n\geq 1}$ is contained in  $\overline{Y}_{l, f_0}(R)\cup \overline{W}_{l, f_0}(R)$.
\end{lem}
\begin{proof}
	By Proposition \ref{cont-c} and the  equality $ \Phi_{\g(a_l), -, f_n}-\Phi_{a_l, +, f_n}=\tau_{a_l}(f_n)$,  passing  to a subsequences if necessary, there are $7$ cases:
	
	%	we have the Hausdorff convergence for 
	\begin{itemize}
		\item  If ${\rm Im } \tau_{a_l}(f_n)\rightarrow +\infty$ and $(s_{l-1}, s_l)=(1,1)$, then
		$$\overline{Y}_{l, f_n}(R)\rightarrow \overline{Y}_{l, f_0}(R), \ \overline{W}_{l, f_n}(R)\rightarrow \overline{W}_{l, f_0}(R)\cup \overline{U}_{a_l, +, f_0}.$$
		It follows that $\overline{X}_{l, f_n}(R)\rightarrow \overline{Y}_{l, f_0}(R)$.
		
		\item  If ${\rm Im } \tau_{a_l}(f_n)\rightarrow +\infty$ and $(s_{l-1}, s_l)=(-1,1)$, then
		$$\overline{Y}_{l, f_n}(R)\rightarrow \overline{Y}_{l, f_0}(R)\cup \overline{U}_{\g(a_l), -, f_0}, \ \overline{W}_{l, f_n}(R)\rightarrow \overline{W}_{l, f_0}(R)\cup \overline{U}_{a_l, +, f_0}.$$
		It follows that $\overline{X}_{l, f_n}(R)\rightarrow \overline{Y}_{l, f_0}(R)\cup \overline{W}_{l, f_0}(R)$.
		
		\item  If ${\rm Im } \tau_{a_l}(f_n)\rightarrow +\infty$ and $(s_{l-1}, s_l)=(-1,-1)$, then
		$$\overline{Y}_{l, f_n}(R)\rightarrow \overline{Y}_{l, f_0}(R)\cup \overline{U}_{\g(a_l), -, f_0}, \ \overline{W}_{l, f_n}(R)\rightarrow \overline{W}_{l, f_0}(R).$$
		It follows that $\overline{X}_{l, f_n}(R)\rightarrow  \overline{W}_{l, f_0}(R)$.
		
		%	By the Hausdorff convergence,  there is a large integer $N>0$ so that $f$ 

		\item  If ${\rm Im } \tau_{a_l}(f_n)\rightarrow -\infty$ and $(s_{l-1}, s_l)=(-1,-1)$, then
		$$\overline{Y}_{l, f_n}(R)\rightarrow \overline{Y}_{l, f_0}(R), \ \overline{W}_{l, f_n}(R)\rightarrow \overline{W}_{l, f_0}(R)\cup \overline{U}_{a_l, +, f_0}.$$
		It follows that $\overline{X}_{l, f_n}(R)\rightarrow  \overline{Y}_{l, f_0}(R)$.
		
		\item  If ${\rm Im } \tau_{a_l}(f_n)\rightarrow -\infty$ and $(s_{l-1}, s_l)=(1,-1)$, then
		$$\overline{Y}_{l, f_n}(R)\rightarrow \overline{Y}_{l, f_0}(R)\cup \overline{U}_{\g(a_l), -, f_0}, \ \overline{W}_{l, f_n}(R)\rightarrow \overline{W}_{l, f_0}(R)\cup \overline{U}_{a_l, +, f_0}.$$
		It follows that $\overline{X}_{l, f_n}(R)\rightarrow \overline{Y}_{l, f_0}(R)\cup \overline{W}_{l, f_0}(R)$.
		
		\item  If ${\rm Im } \tau_{a_l}(f_n)\rightarrow -\infty$ and $(s_{l-1}, s_l)=(1,1)$, then
		$$\overline{Y}_{l, f_n}(R)\rightarrow \overline{Y}_{l, f_0}(R)\cup \overline{U}_{\g(a_l), -, f_0}, \ \overline{W}_{l, f_n}(R)\rightarrow \overline{W}_{l, f_0}(R).$$
		It follows that $\overline{X}_{l, f_n}(R)\rightarrow   \overline{W}_{l, f_0}(R)$.
		
		\item    If ${\rm Im } \tau_{a_l}(f_n)\rightarrow c\in \mathbb R$, then $s_{l-1}=s_l$ and 
		\bess \overline{Y}_{l, f_n}(R)&\rightarrow& \overline{Y}_{l, f_0}(R)\cup \overline{W}_{l, f_0}(R+s_lc),\\
		\overline{W}_{l, f_n}(R)&\rightarrow& \overline{Y}_{l, f_0}(R-s_lc)\cup \overline{W}_{l, f_0}(R).
		\eess
		It follows that 
		$$\overline{X}_{l, f_n}(R)\rightarrow \overline{Y}_{l, f_0}(\min\{R, R-s_lc\})\cup \overline{W}_{l, f_0}(\min\{R, R+s_lc\}).$$
		%	\item  \textcolor{red}{  if ${\rm Im } \tau_{a_l}(f_n)\rightarrow c\in \mathbb R$, then  
			%		$$\overline{Y}_{l, f_n}(R)\rightarrow \overline{Y}_{l, f_0}(R)\cup \overline{W}_{l, f_0}(R+c), \ \overline{W}_{l, f_n}(R)\rightarrow \overline{Y}_{l, f_0}(R-c)\cup \overline{W}_{l, f_0}(R).$$
			%	It follows that $\overline{X}_{l, f_n}(R)\rightarrow \overline{Y}_{l, f_0}(R)\cup \overline{W}_{l, f_0}(R+c)$ if $c\leq 0$; $\overline{Y}_{l, f_0}(R-c)\cup \overline{W}_{l, f_0}(R)$ if $c\geq 0$. }
	\end{itemize}
In either of above cases, the   Hausdorff limit of $\overline{X}_{l, f_n}(R)$ is as claimed.
	\end{proof}

 \begin{proof}[Proof of Proposition \ref{ak-limit2}]

For the given compact set $E\subset U_{k,+,f_0}$, by Proposition \ref{cont-c}, there is a constant $M>0$ such that for large $n$, 
$$E\subset  U_{k,+,f_n}  \text{ and } -M\leq {\rm Im} \Phi_{k, +, f_n}|_{E},  {\rm Im} \Phi_{k, +, f_0}|_{E} \leq M.$$

  %For $l\geq 2$, let %$B_l=\sup_{n}|{\rm Im } \tau_{a_l}(f_n)|$, and define
%	$$Y_{l,g}(R)=\{z\in U_{a_{l}, +, g}; \ s_{l-1} {\rm Im } \Phi_{a_{l}, +, g}(z)> R\}, \ g\in \{f_n, f_0\},$$
%	then $X_{l, f_n}(R)=Y_{l, f_n}(R)\cap W_{l, f_n}(R)$, where $W_{l, f_n}(R)$ is  given by \eqref{wr}.

For $g\in \{f_n, f_0\}$,  	note that  the  sequence  $(\overline{Y}_{l, g}(R)\cup \overline{W}_{l, g}(R))_{l\geq 2}$ is pre-periodic. 
By this,  for any $\epsilon>0$,   there is a $R>0$ so that 

\begin{itemize}
	\item $\overline{Y}_{l, f_0}(R)\cup \overline{W}_{l, f_0}(R)\subset \mathbb D(0, \epsilon)$ for all $l\geq 2$.  By Lemma \ref{hausdorff-limit}, there is an integer $N$ so that  ${X}_{l, f_n}(R)\subset \mathbb D(0, \epsilon)$ for all $n\geq N$ and $l\geq 2$. 
	
	\item  for any $l\geq 1$ and any $w\in   \mathbb C
	$ satisfying that $s_l {\rm Im}(w)\geq R$, 
	$$|\E_{l, f_n}(w)-T_{C_{l, n}}(w)|\leq  1/2, \ |\E_{l, f_0}(w)-T_{C_{l}}(w)|\leq 1/2,$$
	where 
	$$C_{l,n}= \lim_{s_l {\rm Im}(w)\rightarrow +\infty}(\E_{l, f_n}(w)-w), C_{l}= \lim_{s_l {\rm Im}(w)\rightarrow +\infty}(\E_{l, f_0}(w)-w).$$
	
	\item for any $l\geq 1$ and any $n\geq N$, 
	$$s_l\Delta_l(f_n)\geq R_1:= R+M+\sum_{j=1}^{m+p}(1+|{\rm Im} C_j|).$$
	\end{itemize}

%	 \begin{itemize}
	% 	\item if $B_l<+\infty$, then , by  choosing $N$ large, we have $X_{a_l, f_n}(R+B_l)\subset \mathbb D(0, \epsilon)$ for $n\geq N$.
	 %	\item  if $B_l=+\infty$ and $s_{l-1}=s_{l}$, then for large $n$,  either  $Y_{a_l, f_0}(R)\subset W_{a_l, f_0}(R)$ or $W_{a_l, f_0}(R)\subset Y_{a_l, f_0}(R)$
	 %	\end{itemize}\\
 
 %We may also assume $R$ is large so that
 
 By the  equality $s_1\Phi_{\g(a_1), -, f_n}=s_1\tau_{a_1}(f_n)+s_1\Phi_{a_1, +, f_n}$, and enlarging $N$ if necessary,  we have $E\subset W_{1, f_n}(R_1)$  for all $n\geq N$.
 Let $E_{t, n}:=f_n^t(E)$ for all $n,t\geq 1$. 
% Let $C_{j,n}= \lim_{s_j {\rm Im}(w)\rightarrow +\infty}(\E_{j, f_n}(w)-w)$.
 By Proposition \ref{ecalle0}, $\lim_{n\rightarrow +\infty}C_{j,n}=C_{j}$.
 So we may assume $|C_{j,n}-C_{j}|<1/2$ for all $n\geq N$ and $j\geq 1$.
 
   %$ \lim_{s_j {\rm Im}(w)\rightarrow +\infty}(\E_{j, f_n}(w)- T_{C_{j, n}}(w))=0$  and
 By \eqref{ecalle},  for $e_{1,n} \leq t\leq q_{1, n}$,   we have
 \bess
 &\Phi_{a_{2}, +, f_n} \circ f_n^t|_E=T_t\circ \E_{1, f_n}\circ T_{\tau_{a_1}(f_n)}\circ \Phi_{a_1, +, f_n}|_E,&\\
  & \Phi_{\g(a_{2}), -, f_n} \circ f_n^t|_E=T_{\tau_{a_2}(f_n)}\circ  T_t\circ \E_{1, f_n}\circ T_{\tau_{a_1}(f_n)}\circ \Phi_{a_1, +, f_n}|_E.&
  \eess
 It follows that 
 \bess s_1{\rm Im} \Phi_{a_{2}, +, f_n} \circ f_n^t|_E&=&s_1 {\rm Im} \E_{1, f_n}\circ T_{\tau_{a_1}(f_n)}\circ \Phi_{a_1, +, f_n}|_E\\
 &\geq &  s_1\Delta_1(f_n) -s_1  {\rm Im} C_{1,n}-s_1/2-M\\
 &\geq& R_1-|{\rm Im} C_1|-1-M >R,
 \eess
  \bess s_2 {\rm Im}  \Phi_{\g(a_{2}), -, f_n} \circ f_n^t|_E&=&s_2 {\rm Im}  T_{\tau_{a_2}(f_n)}  \circ \E_{1, f_n}   \circ T_{\tau_{a_1}(f_n)}\circ \Phi_{a_1, +, f_n}|_E\\
 &\geq &  s_2\Delta_2(f_n) -s_1  {\rm Im} C_{1,n}-s_1/2-M\\
 &\geq& R_1-|{\rm Im} C_1|-1-M >R.
 \eess
 Hence $f_n^t(E)\subset Y_{2, f_n}(R)\cap W_{2, f_n}(R)=X_{2,f_n}(R)$.
 
 By \eqref{ecalle}, \eqref{l-equality} and 
 %$${\rm Im} \Phi_{a_{j+1}, +, f_n} \circ f_n^t={\rm Im} \Phi_{\g(a_j), -, f_n}+{\rm Im} C_{j,n}+ o_n(1), \ \forall j\geq 1$$
 %on $E_{q_{j-1,n}, n}$, where $t$ satisfies $e_{j,n}\leq t+q_{j-1,n}\leq q_{j,n}$, and $o_n(1)$ is a term with $\lim_{n\rightarrow +\infty}o_n(1)\rightarrow 0$.
 %By
  induction, enlarging $N$ if necessary,   we have 
 \begin{equation}\label{con-orb-e1}
 E_{t, n}\subset X_{j+1, f_n}(R)\subset \mathbb D(0, \epsilon),   \text{ for } e_{j, n} \leq t\leq q_{j,n}, 1\leq j\leq m+p, \ n\geq N.
 \end{equation}
 
 %We assume that $R$ is sufficiently large so that the region $\{w \in \mathbb{C}: s_j \operatorname{Im}(w) \geq R\}$ lies within the domain of definition of both $\E_{j, f_n}$ and $\E_{j, f_0}$.
 
  We assume $R$ is large so that $\{w\in \mathbb C;  s_j {\rm Im}(w)\geq R\}$  lies within  the common  domains  of $\E_{j, f_0}$ and  $\E_{j, f_n}$ for all $n\geq N$,   and $\E_{j, f_n}$ converges uniformly to $\E_{j, f_0}$ on this set  as $n\rightarrow +\infty$.  This implies that 
 {
   \begin{equation}\label{con-orb-e2}
  E_{t,n}\subset \mathbb D(0, \epsilon), \ \text{ for } q_{j-1, n} \leq t\leq e_{j, n}, 1\leq j\leq m+p, n\geq N,
   \end{equation} }
   where $q_{0, n}=q_0$ is independent of $n$  so that $f_n^{q_0}(E)\subset  \mathbb D(0, \epsilon)$ for  $n\geq N$.
 
% Again by the fact  $\E_{j, f_n}(w)-\E_{j, f_0}(w)\rightarrow 0$ uniformly as $s_j {\rm Im}(w)\rightarrow +\infty$, we have  
 %,  we need the following 
 %\begin{lem}   \label{c-orb2}  For any $j\geq 1$,   $\sup_{n\geq N}(e_{n, j}-q_{n, j-1})<+\infty$.
 %	there is an integer $m_j\geq 1$   independent of $n\geq N$ so that  $e_{n, j}-q_{n, j-1}\leq m_j$.
 %\end{lem}
%\begin{proof}
%	\end{proof}

  Note that $a_{m+p+1}=a_{m+1}$. 
On  $E_{q_{m, n}, n} \subset U_{a_{m+1},+, f_n}$,  by   \eqref{ecalle},  
$$\Phi_{a_{m+1}, +, f_n} \circ f_n^t=T_t\circ \E_{m+p, f_n}\circ T_{\tau_{a_{m+p}}(f_n)}\circ \dots \circ  \E_{m+1, f_n}\circ T_{\tau_{a_{m+1}}(f_n)}\circ \Phi_{a_{m+1}, +, f_n},$$
 where   $e_{m+p,n}\leq  t+q_{m,n}\leq q_{m+p,n}$. By Proposition \ref{phase}  and  Lemma \ref{asm-iota}, 
%\textcolor{red}{ we have  
\bess
 \Phi_{a_{m+1}, +, f_n} \circ f_n^t&=&   \Phi_{a_{m+1}, +, f_n}+\sum_{j=1}^p \tau_{a_{m+j}}(f_n)+\sum_{j=1}^p {C_{m+j, n}}+t+O(1)\\
&=&\Phi_{a_{m+1}, +, f_n}+2\pi i s_{*}\iota(f_n, v(f_n))+t+O(1).
\eess 
  Hence
  $$s_{*}{\rm Im}  \Phi_{a_{m+1}, +, f_n} \circ f_n^t=s_{*}{\rm Im}  \Phi_{a_{m+1}, +, f_n}+2\pi {\rm Re } \ \iota(f_n, v(f_n))+O(1).$$
  
  By Proposition  \ref{a-preperiodic},  ${\rm Re } \  \iota(f_n, v(f_n))\rightarrow +\infty  \text{ as } n\rightarrow +\infty$. Hence by enlarging $N$, we have $s_{*}{\rm Im}  \Phi_{a_{m+1}, +, f_n} \circ f_n^t\geq 2R$ on $E_{q_{m, n}, n}$ for $n\geq N$.
  By induction, for any given $n\geq N$, there is a sequence of integers $(t_j)_{j\geq 1}$ so that   $\lim_{j}t_j=+\infty$, $\sup_{j\geq 1}(t_{j+1}-t_j)<+\infty$, and
  $$s_{*}{\rm Im}  \Phi_{a_{m+1}, +, f_n} \circ f_n^{t_j}|_{E_{q_{m, n}, n}}\geq jR, \ \forall j\geq 1.$$
  Note that for any sequence $(z_j)_j$ in $U_{a_{m+1},+, f_n}$,   $$\lim_{j\rightarrow +\infty}s_{*}{\rm Im}  \Phi_{a_{m+1}, +, f_n}(z_j)\rightarrow +\infty  \Longrightarrow z_j\rightarrow v(f_n).$$
It follows that the subsequence $(f_{n}^{q_{m,n}+t_j})_{j\geq 1}$ converges uniformly on $E$ to the $f_n$-attracting fixed point $v(f_n)$. Since  $\sup_{j\geq 1}(t_{j+1}-t_j)<+\infty$, we conclude that $(f_{n}^{j})_{j\geq 1}$ converges uniformly  on $E$ to  $v(f_n)$.

Now for any open set $V$ with compact closure $\overline{V}$  in $A_k$,  there is an $m\geq 1$ so that $f_0^m(\overline{V})\subset U_{k,+,f_0}$. For all large $n$, the set $f_n^m(\overline{V})$
is contained in some compact set $E\subset U_{k,+,f_0}$. By what is  proven, the   sequence $(f_n^{j})_{j\geq 1}$ converges uniformly on $\overline{V}$ to the $f_n$-attracting point $v(f_n)$. Hence $V$ is  in the Fatou set $F(f_n)$. Since $V\subset A_k$ is arbitrary, we get $A_k\cap X= \emptyset, \ \forall X\in  \mathcal{L}((J(f_{n}))_n)$.
  	\end{proof}

 %Note that $f_n^{q_{n, m}}(E)\subset X_{{m+1}, f_n}(R)$.
%Take an  $n\geq N$. For all $l\geq q_{n,m}$, we have $f_n^{l}(E)\subset \mathbb D(0, \epsilon)$. 

 %$l\in A:=\{a_{m+1}, a_{m+2}, \cdots, a_{m+p}\}$, 

%\begin{lem}  [Controlling the orbit]  \label{c-orb} For any $\epsilon>0$, there exist  $R>0$ and an integer $N>0$ so that for all $n\geq N$,
	
%	(1).  $X_{l, f_n}(R)\subset \mathbb D(0, \epsilon)$,  $\forall  l\geq 2$;

	%(2).    $f_n^{t}(E)\subset \mathbb D(0, \epsilon)$ for all $t\geq q_{0}$,   where $q_0$ is an integer depending only on $\epsilon$ and independent of $n\geq N$;
	
	% (2). $f_n^{t}(E)\subset X_{a_l, f_n}(R)$, $\forall  e_{n, l} \leq t\leq q_{n, l}, 1\leq l\leq m+p$; 
	
	%(3).   $f_n^{t}(E)\subset \mathbb D(0, \epsilon)$,  $\forall q_{n, l-1} \leq t\leq e_{n, l}, 1\leq l\leq m+p$,  
	% here $q_{n,0}=q_0$ is am integer depending solely on $\epsilon$ and independent of $n$.  
	
	%(3).  $(f_n^{j})_{j\geq 1}$ converges uniformly on $E$ to the $f_n$-attracting point $v(f_n)$, where $v$ is given by Proposition \ref{a-preperiodic}.  
%\end{lem}

%Let $f$ be a rational map with parabolic fixed point $\zeta$ of multiplicity $\nu+1$.   Let $f_n\rightarrow f$ be  a $*$-sequence of $f$ at $\zeta$. 
For $k\in \mathbb Z_\nu$, let $\gamma_k: [0, +\infty)\rightarrow \mathbb D(0, r_0)$ be an $f_0$-invariant curve  $f_0(\gamma_k(t))=\gamma_k(t+1)$, so that $\gamma_k(t)$ converges to $0$ along the $k$-th attracting direction as $t\rightarrow +\infty$. 
%$\lim_{t\rightarrow +\infty}\gamma_k(t)=0$ and $\gamma_k$ is asymptotic to the $k$-th attracting.
 Let $\gamma_{n,k}: [0, +\infty)\rightarrow \mathbb C$ be an $f_n$-invariant curve  $f_n(\gamma_{n,k}(t))=\gamma_{n,k}(t+1)$ so that $\gamma_{n,k}$ converges locally and uniformly to $\gamma_k$ in $[0, +\infty)$ as $n\rightarrow +\infty$.
% If $\gamma_{n,k}$ converges  uniformly to $\gamma_k$ in $(-\infty, 0]$, we denote it by $\gamma_{n,k} \rightrightarrows \gamma_k$.

\begin{pro} \label{lic}   Applying Oudkerk's algorithm to $k\in \mathbb Z_\nu$,   we have
	
		1.	If the algorithm never terminates, then $\gamma_{n,k}$ converges  uniformly to $\gamma_k$ in $[0, +\infty)$.
%	$$\gamma_{n,k} \rightrightarrows \gamma_k.$$
	This implies that $\overline{\gamma_{n,k}}$ converges to $\overline{\gamma_k}$ in Hausdorff topology.

	2.	If the algorithm terminates at step (3), then there is no subsequence of $(\gamma_{n,k})_n$ converging to $\gamma_k$ uniformly in $[0, +\infty)$. Precisely, % $\gamma_k\subsetneq \liminf_n \gamma_{n,k}$.
	$$\overline{\gamma_k}\subsetneq  X, \ \forall X\in  \mathcal{L}((\overline{\gamma_{n,k}})_n).$$

\end{pro}  
\begin{proof} 
	Replacing $\gamma_k$ by  $f_0^{m}(\gamma_k)$ for some   $m\geq 0$, we may
assume $\gamma_k\subset U_{k,+, f_0}$.  Since $\gamma_{n,k}$ converges uniformly to $\gamma_{k}$ in $[0,1]$, there is an integer $n_0>0$ so that   $E=\gamma_k[0,1] \bigcup \bigcup_{n\geq n_0} \gamma_{n,k}[0,1]$ is a compact subset of $A_k$. 
	
	1.  Assume the algorithm never terminates.
	For any $\epsilon>0$,  by \eqref{con-orb-e1} and \eqref{con-orb-e2} and  the proof of  Proposition \ref{ak-limit2}, there are two integers $N\geq n_0, q_0>0$ so that for all $n\geq N$,  
		$$ \gamma_{n,k}[q_0, +\infty)\subset \mathbb D(0, \epsilon).$$
	 Since $\gamma_{n,k}$ converges  locally and uniformly to $\gamma_k$ in $[0, +\infty)$, we may enlarge $N$ so that $|\gamma_{n,k}(t)-\gamma_{k}(t)|\leq \epsilon$ for $t\in [0,q_0]$ and $n\geq N$. 
	It follows that  $|\gamma_{n,k}(t)-\gamma_{k}(t)|\leq 2\epsilon$ for all $t\geq 0$ and all $n\geq N$.   This shows  the uniform convergence $\gamma_{n,k} \rightrightarrows \gamma_k$.

2. Assume we get a finite sequence $a_1:=k, \cdots, a_\ell$.  By Proposition \ref{ak-limit1} and replacing $(f_n)_n$ by a subsequence,   there is a sequence of integers $(N_{n})_{n\geq 1}$ so that 
	$f_n^{N_{n}}$ converges locally and uniformly in $A_k$ to the Lavaurs map  $g_{\sigma}= \Phi_{\g(a_\ell), -, f_0}^{-1} \circ T_{\sigma}\circ \Phi_{a_1, +, f_0}: A_k\rightarrow \mathbb C$.
	Note that $\gamma_{n,k}[0,1]\subset E$ for $n\geq n_0$.
	By the Hausdorff convergence $f_n^{N_{n}}\gamma_{n,k}[0,1]=\gamma_{n,k}[N_n,N_n+1]\rightarrow g_\sigma(\gamma_k[0,1])$, we get that 
	$$\liminf_{n\rightarrow +\infty}{\rm diam}\gamma_{n,k}[N_n, +\infty)\geq  {\rm diam}g_\sigma(\gamma_k[0,1])>0.$$
	Hence  no subsequence of $(\gamma_{n,k})_n$ converges to $\gamma_k$ uniformly in $[0, +\infty)$.
	%	For large $n$,   This implies 
%	$$A_k\cap \limsup_{n\rightarrow \infty} J(f_n)\neq \emptyset.$$  
	\end{proof}

\begin{rmk}\label{germ} Our discussion in this section applies equally well  to the general setting that $f_0, f_n$ are local holomorphic germs.  %rather than globally defined rational maps. 
	\end{rmk}

%- \#\{k\in \mathbb Z_{\nu}; \mathbf G_k=*\}$ 

\section{Continuity for generic perturbation} \label{generic-p}

	This section investigates the relationship  between the
	horocyclic convergence and Hausdorff convergence of invariant curves and Julia sets.  We shall prove Theorems  \ref{hdc-horo} and  \ref{limit-A}. 

%In this section, we study the  relation between the horocyclic convergence and the kernel convergence of Fatou components and the Hausdorff convergence of invariant curves.  
% continuity for Julia sets and Fatou components.  

%\section{Julia continuity implies horocyclic convergence}

%$\vartheta$$\varepsilon$$\chi$$\eta$

Let   $(f_n)_n$ be a generic perturbation of $f_0$.   Let $\zeta=0$ be a parabolic fixed point of $f_0$ with multiplicity $\nu+1$,    $f_0'(0)=1$.    
In what follows,  we adopt the notations introduced at the beginning of \S \ref{h-e}.  
%For each parabolic fixed point $\zeta$ of $f$, we use the same notations 
 %satisfying . 

\vspace{5pt}
\noindent \textbf{Proof of Theorem \ref{hdc-horo}.}   Assume $J(f_n)\rightarrow J(f_0)$.  To show 	$f_n\rightarrow f_0$ horocyclically, it is equivalent to show that for any $v\in V$, 
\begin{equation} \label{h-limit}
|{\rm Re } \  \iota(f_n, v(f_n))|\rightarrow +\infty.
\end{equation}

For each $k\in \mathbb Z_\nu$, let $V^+(k)$ (resp.  $V^-(k)$ ) be the set of vertices $v\in V$ so that 
$X_v$ is on the upper (resp. lower) component of $\mathbb D\setminus \ell_k$ with respect to the direction of $\ell_k$.  It is clear that $V=V^+(k)\sqcup V^-(k)$.

For any $k\in \mathbb Z_\nu$,  recall that $A_k$ is the $k$-th immediate parabolic basin of $f_0$ at $0$. The assumption $J(f_n)\rightarrow J(f_0)$ implies that 
$A_k\cap X= \emptyset, \ \forall X\in  \mathcal{L}((J(f_{n}))_n)$.
% $A_k\cap \limsup_{n\rightarrow \infty} J(f_n)= \emptyset$.
	  By Proposition \ref{ak-limit1}, Oudkerk's algorithm can not   terminate at step (3).
	  It follows that 
	  \begin{itemize}
	  	\item we get two infinite sequences
	  	$$\boldsymbol{a}(k)=(a_j(k))_{j\geq 1}\in   (\mathbb Z_\nu)^{\mathbb N},  \ \boldsymbol{s}(k)=(s_j(k))_{j\geq 1}\in \{\pm 1\}^{\mathbb N}$$
	  	such that $a_1(k)=k$ and $\boldsymbol{s}(k)$ is given by \eqref{s-j}. 
	  	
	  	\item For any $k\in \mathbb Z_\nu$ and $s\in \{\pm\}$, 
	  	\begin{equation} \label{h-limit1} \sum_{v\in V^{s}(k)} {\rm Re } \ \iota(f_n,  v(f_n))\rightarrow +\infty \text{ or }  -\infty.
	  	\end{equation}
	  	\end{itemize}
	  	
	  To see \eqref{h-limit1}, it suffices to note that since Oudkerk's algorithm for $k$ can not   terminate at step (3), we get  ${\rm Im}\  \tau_{k}(f_n)\rightarrow +\infty$ or $-\infty$ as   $n\rightarrow \infty$.  By Proposition \ref{phase}, 
	  %	\begin{equation}\label{p-formula}\tau_{k}(f_n)=2\pi i \sum_{v\in V^+(k)}  \iota(f_n,  v(f_n))+O(1)=-2\pi i\sum_{v\in V^-(k)}    \iota(f_n,  v(f_n))+O(1).
	  	%	\end{equation}
	  \bess {\rm Im}\  \tau_{k}(f_n)&=&2\pi \sum_{v\in V^+(k)} {\rm Re } \ \iota(f_n,  v(f_n))+O(1)\\
	  &=&-2\pi \sum_{v\in V^-(k)} {\rm Re } \ \iota(f_n,  v(f_n))+O(1).
	  \eess
	  Then \eqref{h-limit1} follows immediately.

%for any $r\geq 1$, as $n\rightarrow \infty$.

% then mark $v$ in the tree $T$ as for the former case, and as   (means repelling)  for the  latter case.  $\ast$ $\star$ $\odot$ $\diamond$ $\times$ 

%\begin{lem}  
%\end{lem}

\begin{lem} \label{two-sequences}  Let $k, k'\in \mathbb Z_{\nu}$.  If 
	$$a_{r+1}(k)=a_1(k'), \ \ s_r(k)=s_1(k')=\cdots =s_m(k')$$
	for some integers $r\geq 1$ and $m\geq 1$, then 
	$$a_{r+s}(k)=a_{s}(k'), \ \forall 1\leq s\leq m+1;  \ s_r(k)=s_{r+1}(k)=\cdots =s_{r+m}(k).$$
\end{lem}
\begin{proof} By the assumption $a_{r+1}(k)=a_1(k')$ and $s_r(k)=s_1(k')$, we have
	%$a_{r+1}(k)=a_{2}(k')$. Combining with the facts 
	$$s_1(k') \cdot {\rm Im}  \tau_{a_1(k')}(f_n)=s_r(k) \cdot {\rm Im}  \tau_{a_{r+1}(k)}(f_n) \rightarrow +\infty.$$ 
	Adding to $s_r(k)   {\rm Im}\sum_{j=1}^{r}\tau_{a_j(k)}(f_n)\rightarrow +\infty$, we get $s_r(k)  {\rm Im}\sum_{j=1}^{r+1}\tau_{a_j(k)}(f_n)\rightarrow +\infty$,
	this implies that $s_{r+1}(k)=s_r(k)=s_2(k')$.   Combining with the assumption $a_{r+1}(k)=a_{1}(k')$, we get 
	$a_{r+2}(k)=a_{2}(k')$. 
	
	Repeating the argument, we get the conclusion.
\end{proof}

In the following, 
% a  vertex $v\in V$ in the tree $T$ is marked as “$\bullet$".  
%If ${\rm Re } \  \iota(f_n, v(f_n))\rightarrow +\infty$, then $v(f_n)$ is $f_n$-attracting for large $n$, in this case  $v$ is marked as “$\circ$" (means {\it horocyclically  attracting});     if ${\rm Re } \  \iota(f_n, v(f_n))\rightarrow -\infty$, then $v(f_n)$ is $f_n$-repelling for large $n$, in this case  $v$ is marked as  “$\mb$”(means  {\it horocyclically  repelling}).
  for any  nonempty subset $X\subset V$,  define its  holomorphic index  $\iota(f_n, X)$ of $f_n$ by 
 $$\iota(f_n, X):=\sum_{v\in X} \  \iota(f_n, v(f_n)).$$

 We say that $X$ is  {\it horocyclically  attracting}   if ${\rm Re } \  \iota(f_n, X)\rightarrow +\infty$, 
  {\it horocyclically  repelling} if ${\rm Re } \  \iota(f_n, X)\rightarrow -\infty$.

In the following, to show Theorem \ref{hdc-horo}, we need to verify that  for any $v\in V$, we    have \eqref{h-limit}.
 If $v$ is an endpoint of  $T$ (that is, there is only one edge incident to $v$), then there exist $k\in \mathbb Z_{\nu}$ and $s\in \{\pm\}$ such that  $V^{s}(k)=\{v\}$.  In this case,  \eqref{h-limit1} implies \eqref{h-limit}. 
	
	In what follows, we assume  at least two edges are incident to $v$.
		Let $v_1, \cdots$, $v_{\delta}$ be the vertices adjacent to $v$, labeled so that the edges $[v, v_1], \cdots, [v, v_{\delta}]$ appear in the cyclic order given by the orientation $\mathcal{O}(v)$ of $v$.
	%	Let $v_1, \cdots, v_{\delta}$ be the adjacent vertices of $v$, numbered  so that the adjacent  edges $[v, v_1],  \cdots, [v, v_{\delta}]$ are in   accordance with the orientation $\mathcal O(v)$ of  
	  The set $T\setminus \{v\}$ consists of $\delta$ connected components. 
	The  vertex set  in the component $T\setminus \{v\}$ containing $v_j$  
	%	For each $v_j$ which satisfies the condition of Lemma \ref{h2}, let $[v_j, u_j]$ be the minimal path avoiding the vertices in $\bigsqcup_{j\geq 3}V_j$ with $u_j\in V_1$. The vertex set along $[v_j, u_j]$
	 can be written as $V^s(k_j)$ for a unique  $k_j\in \mathbb Z_\nu$, and a unique $s\in \{\pm\}$ independent of $j$.  By \eqref{h-limit1},   $V^s(k_j)$ is either    horocyclically  attracting or 
 horocyclically  repelling.   %Since all adjacent edges  have same  orientation, this $e$ is independent of $j$. 
	Note that if  the orientation $\mathcal O(v)$ is positive cyclic order, then $s=-$;
	if   $\mathcal O(v)$ is negative cyclic order, then $s=+$.

%	Assume all adjacent vertices of $v$ satisfy the condition of Lemma \ref{h2}.  
By \eqref{lim-iota}, we have  
\begin{equation} \label{asm3}
\sum_{j=1}^{\delta}\iota(f_n, V^s(k_j))+\iota(f_n, v(f_n))=O(1),   \text{ for large } n.
\end{equation} %for large $n$.

	If  $V^s(k_j)$'s   are all  horocyclically  attracting or all  horocyclically  repelling,  then  $v$ satisfies \eqref{h-limit}  and the proof is done.  So  we may assume some of $V^s(k_j)$'s  are  horocyclically  attracting  and some are
 horocyclically  repelling.

\begin{lem} \label{se-h-rep}  Let   $j\in \mathbb Z_{\delta}$ and $1\leq m<\delta$.  Assume 
	$$V^s(k_{j}),  V^s(k_{j})\sqcup V^s(k_{j+1}), \cdots,  V^s(k_{j})\sqcup\cdots\sqcup V^s(k_{j+m-1})$$ are all horocyclically  repelling, here $k_{s}$ means $k_{s-\delta}$ if $s>\delta$.
	
	1. The  sequence $\boldsymbol{a}(k_j)=(a_l(k_j))_{l\geq 1}$  satisfies that
	\begin{equation} \label{seq-m} a_l(k_j)=k_{j+l-1}, 2\leq l\leq m+1.
		\end{equation}  
	
	2. The set   $V^s(k_{j})\sqcup \cdots\sqcup V^s(k_{j+m})$ is either  horocyclically  repelling or horocyclically  attracting. \footnote{In contrast, if $V^s(k_{j})$ is  horocyclically  attracting, it is possible that $V^s(k_{j})\sqcup V^s(k_{j+1})$ is neither horocyclically  repelling nor horocyclically  attracting.}
%	2.  If the two sequences  $\boldsymbol{a}(k_j)=(a_l(k_j))_{l\geq 1}$ and  $\boldsymbol{a}(k)=(a_l(k))_{l\geq 1}$ satisfy
%	$$a_r(k)=a_1(k_j), \ 	{\rm Im}\sum_{j=1}^r \tau_{a_j(k)}(f_n)\rightarrow +\infty$$
%	for some $r\geq 1$, then 
%	$$a_{r+s}(k)=a_{1+s}(k_j), \ \forall \ 0\leq s\leq m.$$
	\end{lem}

 % Let $k\in \mathbb Z_\nu$, suppose the infinite sequence  $\boldsymbol{a}(k)=(a_l(k))_{l\geq 1}$ given by Oudkerk's algorithm 
	
	%	{\bf Claim: }
	
	\begin{proof} 1. If the   orientation $\mathcal O(v)$ % of the adjacent  edges $[v, v_1],  \cdots, [v, v_{\delta(v)}]$ 
		is positive cyclic order, then 
	$\g(k_{j})=k_{j+1}$ 	for each $j\in \mathbb Z_{\delta}$.  The assumption that $V^s(k_{j})$ is  horocyclically  repelling  implies that ${\rm Im}\tau_{a_1(k_j)}(f_n)\rightarrow +\infty$. Hence $s_1(k_j)=+1$ and $a_2(k_j)=\g(k_{j})=k_{j+1}$. By the assumption and induction, we get the conclusion.

		The same reasoning works for the case that $\mathcal O(v)$ is negative cyclic order.

		2.   By  \eqref{seq-m} and Proposition \ref{phase},   
		$${\rm Im}\sum_{l=1}^{m+1} \tau_{a_l(k_j)}(f_n)=2\pi {\rm Re } \ \iota(f_n,  V^s(k_{j})\sqcup \cdots\sqcup V^s(k_{j+m}))+O(1).$$
		Since Oudkerk's algorithm can not   terminate at step (3), the above quantity tends to $+\infty$ or $-\infty$ as $n\rightarrow \infty$. This implies  that $V^s(k_{j})\sqcup \cdots\sqcup V^s(k_{j+m})$  is either horocyclically  repelling  or horocyclically  attracting.
		\end{proof}
	
By relabeling, we assume $V^s(k_{1})$ is  horocyclically  repelling. By Lemma \ref{se-h-rep}, we may rewrite $V^s(k_{1})\sqcup\cdots\sqcup V^s(k_{\delta})$ %decompose $V\setminus \{v\}$ in 
as the disjoint union $X_1\sqcup \cdots \sqcup X_\ell$ for some $\ell\geq 1$ so that
%Each $X_j$ with takes the form  $V^e(k_{s_j})\sqcup \cdots\sqcup V^e(k_{s_j+m_j})$ with $0\leq m_j<\delta(v)$ satisfying the following properties 
\begin{itemize}
	\item For each $1\leq l \leq \ell$,  % there are integers $m_l\geq 0$ so that 
	 the set $X_l$ takes the form
		$$X_l=V^s(k_{t_l})\sqcup V^s(k_{t_l+1})\sqcup \cdots\sqcup V^s(k_{t_l+m_l}),$$
	 for some  $1\leq t_l\leq \delta$ and $m_l\geq 0$.

		\item $1={t_1}<{t_2}<\cdots< {t_\ell}\leq {t_\ell+m_\ell}=\delta$.
		
	\item   For each $1\leq l\leq \ell$,    either 
	
	\begin{itemize}
		\item $m_l=0$ and 
	$V^s(k_{t_l})$ is horocyclically attracting, or 

\item  $m_l\geq 1$  and
	$$V^s(k_{t_l}),  V^s(k_{t_l})\sqcup V^s(k_{t_l+1}), \cdots,  V^s(k_{t_l})\sqcup\cdots\sqcup V^s(k_{t_l+m_l-1})$$ are all horocyclically  repelling.
	\end{itemize}

	\item  $X_1, \cdots, X_{\ell-1}$ are all  horocyclically  attracting (if $\ell>1$),  $X_{\ell}$ is either horocyclically  repelling or horocyclically  attracting.
	\end{itemize}

	\begin{figure}[h]
	\begin{center}
		\includegraphics[height=7cm]{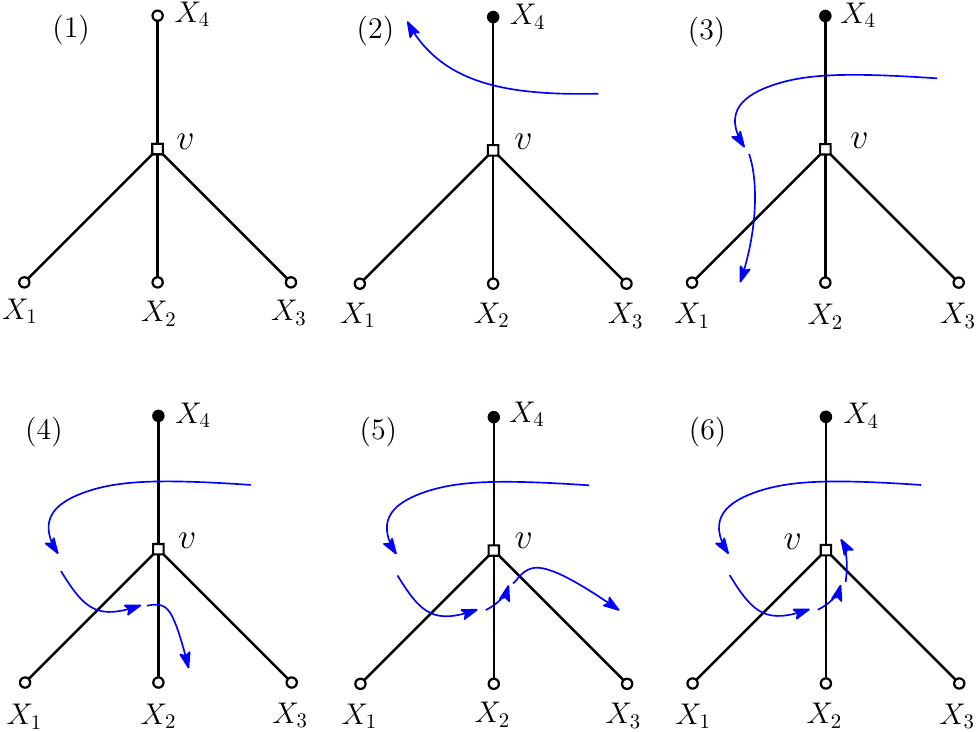}
	\end{center}
	\caption{The case $\ell=4$. Horocyclically  repelling set is marked  $\bullet$, horocyclically   attracting set is marked $\circ$.}%  The case (2) is impossible.}
	\label{fig:reduced}
\end{figure}

There are three cases:

(1). $\ell=1$. In this case, $X_1$ is   either horocyclically  repelling or horocyclically  attracting.
It follows from \eqref{asm3} that   $v$ satisfies \eqref{h-limit}.

(2). $\ell>1$ and  $X_1, \cdots, X_{\ell}$ are all horocyclically  attracting (see Figure \ref{fig:reduced}  (1)). Again by \eqref{asm3},  $v$ is horocyclically  repelling.

(3). $\ell>1$ and $X_1, \cdots, X_{\ell-1}$ are all horocyclically  attracting and   $X_{\ell}$ is   horocyclically  repelling.

%In this case, we consider the infinite  sequence  $\boldsymbol{a}(k_{s_\ell})=(a_l(k_{s_\ell}))_{l\geq 1}$  given by Oudkerk's algorithm.
Now we consider case (3).
By the above properties of $X_\ell$ and  Lemma \ref{se-h-rep},
$$a_l(k_{t_\ell})=k_{t_\ell+l-1}, \ 1\leq l \leq m_\ell+1.$$

There are two possibilities for $a_{m_\ell+2}(k_{t_\ell})$:

If $a_{m_\ell+2}(k_{t_\ell})\neq k_{t_1}$ (see Figure \ref{fig:reduced} (2)), then $X_1\sqcup \cdots \sqcup X_{\ell-1}\sqcup\{v\}$ is  horocyclically  repelling. Since $X_\ell$ is  horocyclically  repelling,  the whole vertex set $V$ is necessarily   horocyclically  repelling.  However this contradicts \eqref{asm3}.

Hence $a_{m_\ell+2}(k_{t_\ell})=k_{t_1}=a_1(k_{t_1})$.  It follows that $s_{m_\ell+1}(k_{t_\ell})=s_1(k_{t_1})$.     By the construction of $X_1$, the sets 
$V^s(k_{t_1}),  V^s(k_{t_1})\sqcup V^s(k_{t_1+1})$, $\cdots$,  $V^s(k_{t_1})\sqcup\cdots\sqcup V^s(k_{t_1+m_1-1})$ are all horocyclically  repelling.  Hence $s_1(k_{t_1})=\cdots =s_{m_1-1}(k_{t_1})$.   By Lemma \ref{se-h-rep},
$$a_{m_\ell+1+l}(k_{t_\ell})=a_l(k_{t_1})=k_{t_1+l-1}, \ 1\leq l \leq m_1+1.$$
%$s_{s_\ell+m_\ell+1}(k_{s_\ell})=s_1(k_{s_1})=\cdots =s_{m_1-1}(k_{s_1})$.
By Proposition \ref{phase},   
$${\rm Im}\sum_{l=1}^{m_\ell+m_1+2} \tau_{a_l(k_{t_\ell})}(f_n)=2\pi \cdot {\rm Re } \ \iota(f_n, X_2\sqcup \cdots \sqcup X_{\ell-1}\sqcup\{v\})+O(1).$$

If $a_{(m_\ell+1)+(m_1+1)+1}(k_{t_\ell})\neq k_{t_2}$  (see Figure \ref{fig:reduced} (3)), then the left part of above equality tends to $-\infty$ as $n\rightarrow -\infty$. This implies that $X_2\sqcup \cdots \sqcup X_{\ell-1}\sqcup\{v\}$ is  horocyclically  repelling. Since $X_2\sqcup \cdots \sqcup X_{\ell-1}$ is  horocyclically   attracting, we conclude that $v$ is    horocyclically  repelling. 

If $a_{(m_\ell+1)+(m_1+1)+1}(k_{t_\ell})=k_{t_2}$ (see Figure \ref{fig:reduced} (4)), we may  continue this process. By the same reasoning,   we see that $v$ is either   horocyclically  repelling (see Figure \ref{fig:reduced} (4)(5)) or horocyclically   attracting (see Figure \ref{fig:reduced}  (6)).

The proof of Theorem \ref{hdc-horo} is completed. 
%\end{proof}

\vspace{5pt}
\noindent \textbf{Proof of Theorem \ref{limit-A}.}      Let $\zeta=0$ be a parabolic fixed point of $f_0$ with multiplicity $\nu+1$,    $f_0'(0)=1$.     Let   $(f_n)_n$ be a generic perturbation of $f_0$ at $\zeta$.
% By Proposition \ref{sp},  
%we may assume $f_n\in \mathcal{WB}_\phi(\mathbf G)$ for   all $n\geq 1$. 

%	Let $f_0$ have a parabolic fixed point $\zeta=0$,  $f_0'(0)=1$, and the  multiplicity of $f_0$ at $0$ is $\nu+1$.   Let   $(f_n)_n$ be generic perturbation of $f_0$ at $0$. 

Assume for each $n$ there is an $f_n$-fixed point  $\zeta_n$ so that   
$$\zeta_n\rightarrow \zeta,  \text{ and } f_n'(\zeta_n)\rightarrow 1  \text{ horocyclically}.$$    By Proposition \ref{sp} and  Lemma \ref{asm-iota},   any subsequence of $(f_n)_n$ admits a further  subsequence $(f_{n_j})_j$ with $f_{n_j}\in \mathcal{WB}_\phi(\mathbf G)$ for   all $j$,   and a vertex $v_0\in V$  (where $V$ is induced by $\mathbf G$, cf. \S \ref{h-e}),
 such that  
 $${\rm Re } \  \iota(f_{n_j}, v_0(f_{n_j}))\rightarrow +\infty  \text{ as } j\rightarrow +\infty.$$
Passing to a further subsequence of $(n_j)_j$ if necessary, we assume for each non-empty subset $X\subset V$,   one of the following holds

\begin{itemize}
	\item ${\rm Re } \  \iota(f_{n_j}, X)\rightarrow +\infty$, as $j\rightarrow +\infty$;
	
	\item  ${\rm Re } \  \iota(f_{n_j}, X)\rightarrow -\infty$, as $j\rightarrow +\infty$;
	
	\item $\sup_{j\geq 1}|{\rm Re } \  \iota(f_{n_j}, X)|< +\infty$.  
	\end{itemize}

%As we did in the  proof of  Theorem \ref{hdc-horo},
Let $v_1, \cdots, v_{\delta}$ be the vertices adjacent to  $v_0$, labeled so that the   edges   $[v_0, v_1],  \cdots, [v_0, v_{\delta}]$  are arranged in the cyclic order induced by the orientation 
 $\mathcal O(v_0)$ of $v_0$.  The set $T\setminus \{v_0\}$ consists of $\delta$ connected components. 
  The vertex set in the component %of $T\setminus \{v_0\}$  
  containing $v_s$  
  is denoted by  $Y_s$.
  For any $0\leq l\leq \delta-1$, let $Y_s^{(l)}=Y_s\sqcup Y_{s+1}\sqcup\cdots\sqcup Y_{s+l}$, here $Y_{m}$ means $Y_{m-\delta}$ if $m>\delta$.

%

%In the following, we may pass to a subsequence of   $(n_j)_j$  if necessary.

%Henceforth, we may take $(n_j)_j$ to be a subsequence satisfying the desired properties.

 {\it Claim:  There is an integer $1\leq s\leq \delta$ so that 
 	$${\rm Re } \  \iota(f_{n_j},  Y_s^{(l)} )\rightarrow -\infty, \ \forall \ 0\leq l\leq \delta-1.$$}

To prove the claim, first note  that by Lemma \ref{asm-iota} and the choice of $v_0$, % we have 
\begin{equation} \label{sum-rep}
	\sum_{s=1}^{\delta}{\rm Re } \  \iota(f_{n_j},  Y_s)=-{\rm Re } \  \iota(f_{n_j}, v_0(f_{n_j}))+O(1)  \rightarrow -\infty.
	\end{equation}

By passing to a subsequence,  the set $$Z:=\{1\leq s\leq \delta; \  {\rm Re } \  \iota(f_{n_j},  Y_s) \rightarrow -\infty \text{ as } j\rightarrow +\infty\} \neq \emptyset.$$  
By renumbering the indices, we may assume that $1\in Z$. Set 
$s_1 =1$ and let $0\leq m_1\leq \delta-1$ be the maximal integer so that
$Y_{s_1}^{(0)}, Y_{s_1}^{(1)}, \cdots, Y_{s_1}^{(m_1)}$
 %$$Y_{j_1}, Y_{j_1}\sqcup Y_{j_1+1}, \cdots, Y_{j_1}\sqcup Y_{j_1+1}\sqcup\cdots\sqcup Y_{j_1+m_1-1}=: Y_{j_1}^*$$
are all horocyclically  repelling for the sequence $(f_{n_j})_j$.
 
If $s_1+m_1=\delta$, the proof is completed. If $s_1+m_1<\delta$,   then by the choice of $m_1$,  we get
$\inf_{j\geq 1}{\rm Re } \  \iota(f_{n_j},  Y_{s_1}^{(m_1+1)})>-\infty$. This combined with   \eqref{sum-rep} implies that  $s_1+m_1+1< \delta$ and $[s_1+m_1+2, \delta]\cap Z\neq\emptyset$.

Take the minimal $s_2\in [s_1+m_1+2, \delta]\cap Z$.
%Let $s_2\geq s_1+m_1+2$ be the minimal integer in $Z$.  
Let $0\leq m_2\leq \delta-s_2$ be   maximal so that  $Y_{s_2}^{(0)}, Y_{s_2}^{(1)}, \cdots, Y_{s_2}^{(m_2)}$
%$$Y_{j_2}, Y_{j_2}\sqcup Y_{j_2+1}, \cdots, Y_{j_2}\sqcup Y_{j_2+1}\sqcup\cdots\sqcup Y_{j_2+m_2-1}$$
are all horocyclically  repelling for the sequence $(f_{n_j})_j$.  If $m_2=\delta-s_2$,  this step is completed. 
Otherwise, we have $\inf_{j\geq 1} {\rm Re } \  \iota(f_{n_j},  Y_{s_2}^{(m_2+1)})>-\infty$ and  $s_2+m_2+1<\delta$.  

Repeating this procedure, we get finitely many pairs $(s_l, m_l), 1\leq l\leq N$ so that $Y_1\sqcup Y_{2}\sqcup\cdots\sqcup Y_{\delta}$ can be decomposed into disjoint subsets
$$Y_{s_1}^{(m_1+1)}, Y_{s_1+m_1+2}, \cdots, Y_{s_2-1}, Y_{s_2}^{(m_2+1)}, Y_{s_2+m_2+2}, \cdots, Y_{s_N-1}, Y_{s_N}^{(m_N)},$$
where $s_N+m_N=\delta$.  
Note that for each set $A$ in the list, with the exception of the final one,  we have  $\inf_{j}{\rm Re } \  \iota(f_{n_j},  A)>-\infty$. By   \eqref{sum-rep} and the choices of the pairs $(s_l, m_l), 1\leq l\leq N$, we have that 
$${\rm Re } \  \iota(f_{n_j},  Y_{s_N}^{(m_N+l)})\rightarrow -\infty, \ \forall \ 0\leq l\leq s_N-1.$$

%Note that for each set $A$ in the list except the last one,   Hence  ${\rm Re } \  \iota(f_{n_k},  Y_{j_N}^{(m_N)})\rightarrow -\infty$. It follows that 
%$ Y_{j_N}^{(0)},  Y_{j_N}^{(1)}, \cdots,  Y_{j_N}^{(\delta-1)}$ are  all horocyclically  repelling.  
The proof of the claim is completed by setting $s=s_N$.

Let $k=\boldsymbol k([v_0, v_s])$, where $s$ is given by the Claim.  The choice of $s$ implies that, when applying Oudkerk's Algorithm to $k$, we get an infinite sequence. This sequence  is periodic and can be written as $(\boldsymbol k([v_0, v_{s-1+l \ {\rm mod }\ \delta}]))_{l\geq 1}$. By Proposition \ref{ak-limit2}, we have
$A_k\cap X=\emptyset, \ \forall X\in \mathcal{L}((J(f_{n_j}))_j)$. 
This finishes the proof that (1)$\Longrightarrow$(2).  By Proposition \ref{lic}, we have the uniform convergence $\gamma_{n_j,k} \rightrightarrows \gamma_k$, this shows (1)$\Longrightarrow$(3).

To show (2)$\Longrightarrow$(1) or  (3)$\Longrightarrow$(1), assume by contradiction that (1) is false,   then there is a subsequence $(n_j)_j$ so that    $\sup_{j\geq 1}|{\rm Re } \  \iota(f_{n_j}, v(f_{n_j}))|< +\infty$ for all $v\in V$. For any $k\in \mathbb Z_\nu$,  applying Oudkerk's Algorithm to $k$, we get a finite sequence (in fact the algorithm terminate at the first step).  By Proposition \ref{ak-limit1}, we have $A_k\cap X\neq \emptyset, \ \forall X\in \mathcal{L}((J(f_{n_j}))_j)$,  contradicting (2).  By Proposition \ref{lic},  there is no subsequence of $(\gamma_{n_j,k})_j$ converging to $\gamma_k$ uniformly in $[0, +\infty)$,  contradicting (3).    \hfill\fbox

% By Lemma \ref{asm},   $V^s(k_j)$ is either   horocyclically  attracting or 
%horocyclically  repelling.   %Since all adjacent edges  have same  orientation, this $e$ is independent of $j$. 
%It is clear that if  the orientation $\mathcal O(v)$ is positive cyclic order, then $s=-$,
%if   $\mathcal O(v)$ is negative cyclic order, then $s=+$.

%\section{Proof of  Theorems \ref{h-equiv} and \ref{pb-limit}}
\section{ $*$-sequence}

In this section, we analyze the (partial) horocyclic convergence for $*$-sequences.

%shall prove Theorems \ref{h-equiv} and \ref{pb-limit}.  
%Let   $(f_n)_n$ be a generic perturbation of $f_0$.   
 As before, let $\zeta=0$ be a parabolic fixed point of $f_0$ with multiplicity $\nu+1$,    $f_0'(0)=1$.      We assume   $(f_n)_n$  is a  $*$-sequence of $f_0$ at $\zeta$  and $f_n\in \mathcal{WB}_\phi(\mathbf G)$.
 %In what follows,  we adopt the notations introduced at the beginning of \S \ref{h-e}.  
 We use the same notations as we did in  \S \ref{h-e}. 
 Passing to a subsequence, we may assume that the $f_0$-fixed point $\zeta$ splits into an attracting fixed point (still assumed to be $0$) and $\nu$ distinct non-attracting fixed points for $f_n$
 %Passing to a subsequence, we assume the  $f_0$-fixed point $\zeta=0$  splits into an attracting fixed point say $0$ and  $\nu$ distinct non-attracting  fixed points of $f_n$
  (the situation that the  $f_0$-fixed point $\zeta$  splits into a repelling fixed point and  $\nu$ distinct non-repelling  fixed points of $f_n$ will be discussed in Remark \ref{repelling-plus}).  % one of which is denoted by $\zeta_n$
 % The immediate parabolic basins of $f$ at $0$ are numbered by  $A_1, \cdots, A_{\nu}$, where $A_k$   corresponds to the $k$-th attracting direction $\epsilon_{k,+}=e^{\pi i (2k-1)/\nu}$.

%For each $k \in \mathbb Z_\nu$ with $G_k\neq *$, let $v_k^+$ (resp.  $v_k^-$) be the vertex corresponding to the component $W_{v_k^+}$  (resp. $W_{v_k^-}$) which  is on the left (resp. right) when moving on the geodesic   $\ell_k(\mathbf G)$ from $\epsilon_{k,+}$ to $\epsilon_{j,-}$. 

Recall that $T = (V, E)$ is the tree  associated with $\mathbf{G}$ introduced in  \S \ref{h-e}. 
By passing to a subsequence of $(f_n)_n$, we may assume 
  $v_0$ is a distinguished vertex in $V$ so that $v_0(f_n)=0$ (hence $v_0(f_n)$ is  $f_n$-attracting). % so that $v_0(f_n)$ is the attracting point of $f_n$ for all $n$.  
Each $v\in V\setminus\{v_0\}$ can be connected to $v_0$ by a {\it minimal path}, which is the union of edges $[u_0, u_1]\cup[u_1, u_2]\cup\cdots \cup [u_{l-1}, u_l]$ with  $u_0=v, u_l=v_0$ and $l\geq 1$ being minimal.  
There is a partial order $\prec$ on $V$ (with respect to $v_0$):  we say $v'\prec v$ if either $v'=v$ or $v'$ lies on the minimal path connecting $v_0$ and $v$. Intuitively $v$ is farther from $v_0$ than $v'$ is. Clearly $v_0\prec v$ for all $v\in V$.

For any $v\in V$, let $V(v)=\{v'\in V; v\prec v'\}$. Clearly $V(v_0)=V$ and $v\in V(v)$.
  Let $T_v\subset T$ be the   subtree   whose vertex set is $V(v)$.

%\textcolor{red}{
%For each $v\in V\setminus\{v_0\}$, there is a unique $v'\in V$ so that $v'\prec v$ and $v,v'$ form two vertices of some edge.}

%Let $V_*\subset V\setminus\{v_0\}$ consists of $m$ vertices so that for each $v\in V_*$, $v$ is horocyclically  repelling.

\begin{pro} \label{find-e}
	 For each non-empty subset $V_*\subset V \backslash \{v_0\}$, satisfying that 
	$${\rm Re } \  \iota(f_n,  v(f_n))\rightarrow -\infty, \ \forall v\in V_*,$$
	there is subset of edges $E_*\subset E$ such that $\# E_*=\# V_*$ and for each $e\in E_*$, 
	Oudkerk's algorithm for  $\boldsymbol k(e)$   gives an infinite sequence.
	
	If we further require that
		$$\sup_n|{\rm Re } \  \iota(f_n,  v(f_n))|<+\infty, \ \forall v\in  V\backslash(\{v_0\}\cup V_*),$$
	then	for each  $e\in E\backslash E_*$, 
		Oudkerk's algorithm for  $\boldsymbol k(e)$   gives a finite sequence.
		\end{pro}
	\begin{proof} The proof consists of four steps.
		% we first construct the edge set $E_*$ with  cardinality  $\# E_*=\# V_*$, then show that for  each $e\in E_*$,  	Oudkerk's algorithm for  $\boldsymbol k(e)$   gives an infinite sequence.
		
		\text{\it 1.  Construction of the set $E_*$.}
		
		Let $v_1, \cdots, v_{\delta}$ be the vertices adjacent to  $v_0$, labeled so that the   edges   $[v_0, v_1],  \cdots, [v_0, v_{\delta}]$  are arranged in the cyclic order induced by the orientation 
		$\mathcal O(v_0)$ of $v_0$. Let $T^j\subset T$ be the  subtree whose vertex set is  $V(v_j)\sqcup\{v_0\}$.   
		
		%  Let $v_1, \cdots, v_{\delta}$ be the adjacent vertices of $v_0$.
		  %, numbered  so that the adjacent  edges $[v_0, v_1],  \cdots, [v_0, v_{\delta(v_0)}]$ are in   accordance with the orientation $\mathcal O(v_0)$. 
		 % The set $T\setminus \{v_0\}$ consists of $\delta$ connected components. 
		  %The  vertex set  in the component  containing $v_j$ is denoted by $X_j$. 
		   % Assume $T^1, \cdots, T^{\delta(v_0)}$ attached at $v_0$ in the  orientation of $\mathcal O(v_0)$. 
		    
		    For each  $1\leq j\leq \delta$ so that $V(v_j)\cap V_*\neq \emptyset$, mark the edge $[v_0, v_j]$. Move along the tree $T^j$ from the root point $v_0$ downward  until meeting  the  first  vertex  $v$ which is either {\it branched}  (i.e.  at least three edges meeting at this vertex) or $v\in  V_*$.
		    % (in this case, it is possible that $\delta(v)=2$).  
		      Let $u_1, \cdots, u_{\ell}$  be the adjacent vertices of $v$ in $V(v)$, labeled in the cyclic order induced by the orientation  $\mathcal O(v)$ of $v$. For  each $1\leq s\leq \ell$, let $T_v^s\subset T$ be the subtree whose vertex set is $V(u_s)\sqcup\{v\}$.   
		   % Assume $T_v^1, \cdots, T_v^{\ell}$ attach at $v$  
		    
		  %  If $V(v)\cap V_*=\emptyset$, there is nothing to do. If $V(v)\cap V_*\neq \emptyset$, we treat two cases.
		  
		  	\begin{figure}[h]
		  	\begin{center}
		  		\includegraphics[height=5cm]{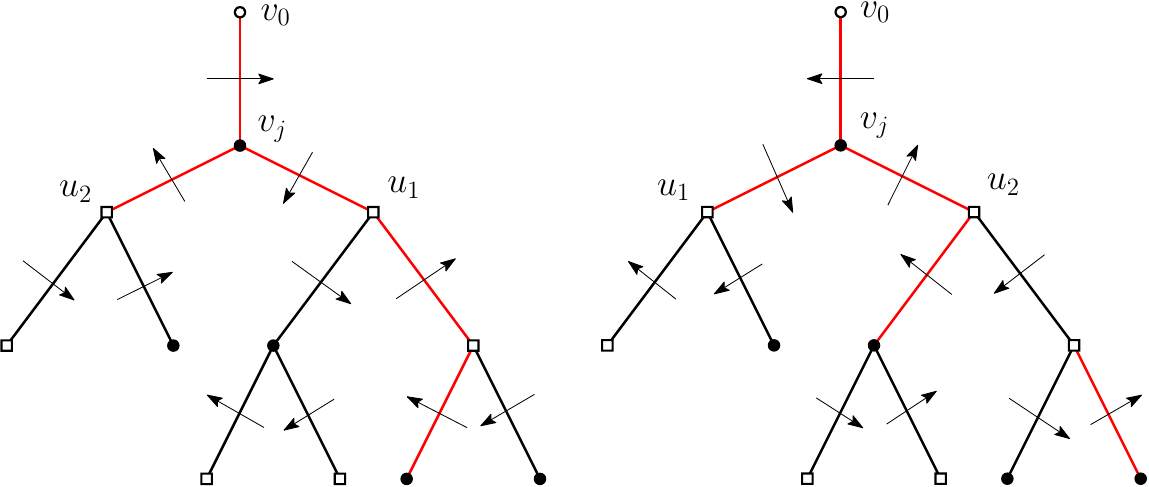}
		  	\end{center}
		  	\caption{The marked edges (in red) in the subtree $T^j$, the arrows are the orientation of the edges.  $v_0$ is marked $\circ$, horocyclically  repelling vertices in $V(v_j)\cap V_*$ are marked $\bullet$ and all other vertices are marked as small squares.}
		  	\label{fig:marked}
		  \end{figure}

		    \textbf{Case 1: $v$ is branched and $v\notin  V_*$.}
		    
		    Since $V(v)\cap V_*\neq \emptyset$ and  $v\notin  V_*$, we get $(V(u_1)\sqcup\cdots \sqcup V(u_\ell))\cap V_*\neq \emptyset$. 
		     Let $1\leq j_1<\cdots< j_m\leq \ell$ be all indices so that $V(u_{j_k})\cap V_*\neq  \emptyset$, where $1\leq m\leq \ell$.
		    If $m=1$, there is nothing to do. If $m\geq 2$, we mark the following edges 
		    $$[u_{j_2}, v], [u_{j_3}, v], \cdots, [u_{j_m}, v].$$
		    
		      \textbf{Case 2:  $v\in  V_*$.}
		      
		      If  $V(v)=\{v\}$ or  $(V(u_1)\sqcup\cdots \sqcup V(u_\ell))\cap V_*=\emptyset$, there is nothing to do. Otherwise,    let $1\leq j_1<\cdots< j_m\leq \ell$ be all indices so that $V(u_{j_k})\cap V_*\neq  \emptyset$, where $1\leq m\leq \ell$.  We then mark the following edges 
		    $$[u_{j_1}, v], [u_{j_2}, v], \cdots, [u_{j_m}, v].$$
		    
		    Now we get some additional marked edges in this step. Move along the subtrees  $T_v^s$ with marked edges from the root point $v$ downward until meeting the first     vertex    which is either branched  or in $V_*$. Repeating the marking procedure as above, 
		    the process terminates after finitely many steps. %, and we get all marked edges.
		    
		    Let $E_*$ be the collection of all marked edges. For each $e=[v,v']\in E_*$, let $\widehat{e}=e\setminus\{v,v'\}$ be  the interior of $e$.  By the construction of $E_*$, the set $T\setminus \bigcup_{e\in E_*}\widehat{e}$ consists of $\# E_*+1$ connected components. Note that each component is either the one containing $v_0$ and avoiding the set $V_*$, or the one    containing precisely one vertex in $V_*$, hence   $\# E_*+1=\#V_*+1$.  Equivalently $\# E_*=\#V_*$. See Figure \ref{fig:marked}.
		    
	%	 \textbf{Claim:   The set $T\setminus \bigcup_{e\in E_*}\widehat{e}$ consists of $\# V_*$ connected components. Each component contains precisely one vertex in $V_*$. }
	
		\text{\it 2. Edge sequences.}
		
		 For each $e\in E$, let $v(e)$ and $u(e)$ be two endpoints of $e$ so that $v(e)\prec u(e)$.
		 % where $u$ is the other endpoint of $e$. 
		  We may associate a sequence   of edges 
		 $$\bs(e)=(b_j(e))_{j\geq 1}=(e_j)_{j\geq 1} \in E^{\mathbb N}$$
		 so that $e_1=e$, and for each $j\geq 1$, 
\begin{itemize}
	\item  $e_j$ and $e_{j+1}$ have a common endpoint $v(e_{j})$;
	
	\item  $e_j$ and $e_{j+1}$  are adjacent and attach at $v(e_{j})$   in cyclic order induced by the orientation  $\mathcal O(v(e_{j}))$;
	
	\item $v(e_{j+1})\prec v(e_{j})$.
	\end{itemize}		 
See Figure \ref{fig:es1}.
		 
		 	\begin{figure}[h]
		 	\begin{center}
		 		\includegraphics[height=5cm]{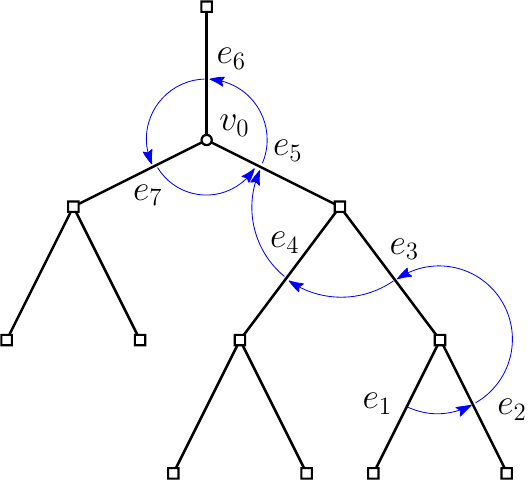}
		 	\end{center}
		 	\caption{An example of edge sequence $(e_j)_{j\geq 1}=(e_1, e_2, e_3, e_4, \overline{e_5, e_6, e_7})$. $v_0$ is marked $\circ$, and all other vertices are marked as small squares.}
		 	\label{fig:es1}
		 \end{figure}
		 
	By definition, the edge sequence   $\bs(e)$ is preperiodic:  there are integers $m\geq 0$ and $p\geq 1$ so that $e_{m+p+j}=e_{m+j}$ for all $j\geq 1$.
		 The minimal $m, p$ are called the {\it pre-period} and the  {\it period} of $\bs(e)$, respectively.
		 It's clear that for the minimal $p$,  $e_{m+1}, \cdots, e_{m+1+p}$ are precisely the   edges attached at $v_0$, in the orientation $\mathcal O(v_0)$.   Hence the periods of all edge sequences are same.
		 %Hence we always have $p=\delta(v_0)$.
	%	 The periodic part of $\bs(e)$ is formed by the finite .

		  	{\it 3.  For each $e\in E_*$, Oudkerk's algorithm for  $\boldsymbol k(e)$   gives an infinite 	sequence $\boldsymbol{a}(\boldsymbol k(e))=(a_j(\boldsymbol k(e)))_{j\geq 1}$, related to the edge sequence  $\bs(e)$ as follows
		  		$$a_j(\boldsymbol k(e))=\boldsymbol k(b_j(e)), \ \forall j\geq 1.$$}
	%  		For each edge $e\in E$,  %if $v(e)$ is branched and $v(e)\notin  V_*$, 
	  %		 following the notations in the beginning, let $u_1, \cdots, u_{\ell}$  be the adjacent vertices of $v(e)$ in $V(v(e))$, numbered in the  orientation of $\mathcal O(v(e))$, here $\ell=\delta(v)-1$, and assume $e=[v(e), u_{l(e)}]$ for some $l(e)\in [1, \ell]\cap \mathbb N$.
	  	%	Define $X(e)=(V(u_1)\sqcup\cdots\sqcup V(u_{l(e)}))\cap V_*$.
	  	
	  	For each $e\in E_*$, let $m_e$ be the pre-period of   $\bs(e)=(e_j)_{j\geq 1}$ and let $X(e_j):=\big(\bigcup_{1\leq s\leq j}V(u(e_s))\big)\cap V_*$.  Note that $X(e_1)\neq \emptyset$ and $X(e_j)\subset X(e_{j+1})$, hence  $X(e_j)\neq\emptyset$ for all $j$.
	  		
	%  		We first prove the following
	  		
	  	{\it 	Claim:  For each $1\leq j\leq m_e+1$,   there is a non-empty set $Y(e_j)\subset X(e_j)$ and $c_j\in \{\pm 1\}$ so that 
	  		\begin{equation} \label{equality-tau}
	  			\sum_{s=1}^j \tau_{\boldsymbol k(e_s)}(f_n)= 2\pi i c_j \cdot \iota(f_n, Y(e_j))+O(1),   \text{ for large } n.
	  				\end{equation} }
	  	%	\begin{itemize}
	  	%		\item $X(e_j):=\big(\bigcup_{1\leq s\leq j}V(u(e_s))\big)\cap V_*\neq\emptyset$;
	  			
	  	%		\item
	  		%
	  			%		\item    $Y(e_j)\neq X(e_j)$ if $v(e_j)\notin V_*$.
	  		%	\end{itemize}
  			
  				\begin{figure}[h]
  				\begin{center}
  					\includegraphics[height=5cm]{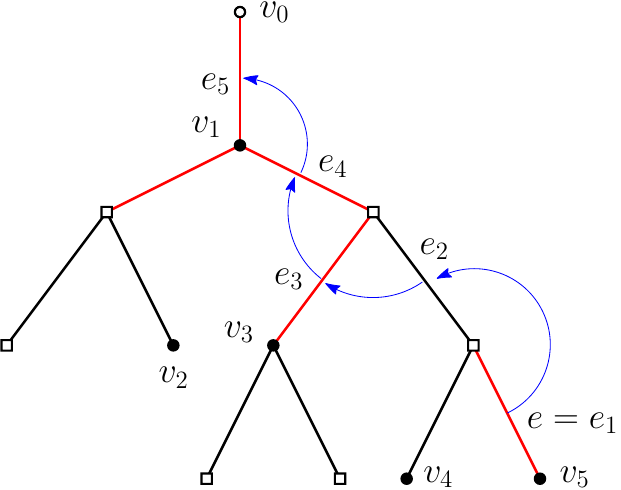}
  				\end{center}
  				\caption{In this example, $V_*\cap V(v_1)=\{v_1, v_2, v_3, v_4, v_5\}$. The pre-period for the edge sequence $(e_j)_{j\geq 1}$ of $e$ is $m_e=4$.  $X(e_1)=Y(e_1)=\{v_5\}$; $X(e_2)=\{v_4, v_5\}$, $Y(e_2)=\{v_4\}$;  $X(e_3)=\{v_3, v_4, v_5\}$, $Y(e_3)=\{v_3, v_4\}$; $X(e_4)=\{v_3, v_4, v_5\}$, $Y(e_4)=\{v_5\}$; $X(e_5)=\{v_1, v_2, v_3, v_4, v_5\}$, $Y(e_5)=\{v_1, v_2, v_3, v_4\}$.}
  				\label{fig:es}
  			\end{figure}
  			
  		%	To se, it suffices to
  			 % by the choice of the marked edges.  
  			
  			Let $Y(e_1)=V(u(e_1))\cap V_*=X(e_1)$. Define $Y(e_j)$ for $1\leq j\leq m_e+1$ inductively as follows:
  			if $v(e_{j+1})=v(e_j)$,  set 
  			$$Y(e_{j+1})=(V(u(e_{j+1}))\cap V_*)\sqcup Y(e_j); $$
  			if $v(e_{j+1})\neq v(e_j)$, then $ v(e_j)=u(e_{j+1})$,  set
  				$$Y(e_{j+1})=(V(u(e_{j+1}))\cap V_*)\setminus Y(e_j).$$
  				By this definition and Proposition \ref{phase}, \eqref{equality-tau}   can be verified inductively. 
  				
  				It remains to show that $Y(e_j)\neq \emptyset$. By the choice of the marked edge,   $Y(e_1)=X(e_1)$, and  the relation   $Y(e_j)=X(e_j)$ preserves until we meet the first $j'\geq 1$ so that $v(e_{j’+1})\neq v(e_{j'})=v(e)$.  By  definition, $Y(e_{j'+1})=X(e_{j'+1})\setminus Y(e_{j'}) \subsetneq  X(e_{j'+1})$.
  					If $v(e)\notin V_*$, then  $Y(e_{j'+1})\neq \emptyset$ (see Case 1 in Step 1); if $v(e)\in V_*$, then 
  				$Y(e_{j'+1}) 	\supset \{v(e)\}\neq \emptyset$.
  				In either case, repeating the arguments, the conclusion follows.   See Figure \ref{fig:es}.

  			%	$$Y(e_{j+1})=X(e_{j+1})\setminus Y(e_j) \begin{cases}\neq \emptyset,  &\text{ if }v(e)\notin V_*; \\
  		%		\supset \{v(e)\}\neq \emptyset,  & \text{ if } v(e)\in V_*.
  			%	\end{cases} $$

  	%		\textcolor{red}{Begin here}

  		%		If $v(e)\in V_*$,  by the choice of the marked edge,   $\emptyset  \neq Y(e) \subset X(e)$. 
%  			The relation   $\emptyset  \neq Y(e_j) \subset X(e_j)$ preserves until we meet the first $j$ so that $v(e_{j+1})\neq v(e_j)=v(e)$.  By the definition of $Y(e_{j+1})$
  %			$$\emptyset\neq \{v(e_j)\} \subset Y(e_{j+1})\subset X(e_{j+1})\setminus Y(e_j) \subsetneq  X(e_{j+1}).$$

  		%	Since $v(e_j)=v(e)$ for $1\leq j\leq \ell-l(e)+1$, we get  \eqref{equality-tau}  and $Y(e_j)\neq X(e_j)$. Note that when $j= \ell-l(e)+1$, we have
  		%\bess &&\sum_{s=1}^j \tau_{\boldsymbol k(e_s)}(f_n)+\tau_{\boldsymbol k(e_{j+1})}(f_n)\\
  	%	&=& 2\pi i c_j \cdot \iota(f_n, Y(e_j))-2\pi i c_j \cdot \iota(f_n, V(u(e_{j+1}))\cap V_*) +O(1)\\
  	%	&=& -2\pi i c_j \cdot \iota(f_n, Y(e_{j+1})) +O(1).\eess
%Hence   	\eqref{equality-tau}  holds for $j+1$ and $Y(e_{j+1})\neq X(e_{j+1})$.		

		  To finish, we show that when $e\in E_*$, Oudkerk's algorithm for  $\boldsymbol k(e)$   gives an infinite 	sequence  
		  %we show that Oudkerk's algorithm for  $\boldsymbol k(e)\in \mathbb Z_\nu$   gives an infinite 	sequence 
		  $\boldsymbol{a}(\boldsymbol k(e))=(a_j(\boldsymbol k(e)))_{j\geq 1}$ satisfying the required properties.   
		  By \eqref{equality-tau},   for each $1\leq j\leq m_e+1$, 
		  $${\rm Im}\sum_{s=1}^j \tau_{\boldsymbol k(e_s)}(f_n)=2\pi c_j   {\rm Re } \ \iota(f_n, Y(e_j))+O(1).$$
		 % Hence $Y(e_j)$ is horocyclically  repelling. 
	
	 For $j>m_e+1$,  by \eqref{equality-tau} and  Proposition \ref{phase},
	 $${\rm Im} \sum_{s=1}^j \tau_{\boldsymbol k(e_s)}(f_n)= 2\pi  c_{m_e+1}  {\rm Re }
	  \Big(\iota(f_n, Y(e_{m_e+1}))+\sum_{s=m_e+2}^j \iota(f_n,  V(u(e_s))\cap V_*) )\Big)+O(1).$$
	  
	 One may verify that $c_j=-1$ if the orientation  $\mathcal O(v(e_{j}))$ is positive cyclic order,  $c_j=+1$ if the orientation  $\mathcal O(v(e_{j}))$ is negative cyclic order. Note also that 
	  $c_j=c_{m_e+1}$ for  $j\geq m_e+1$.  Hence as $n\rightarrow +\infty$,  
	  $${\rm Im} \sum_{s=1}^j \tau_{\boldsymbol k(e_s)}(f_n)\rightarrow \begin{cases}
	  	+\infty,  &\text{ if }  \mathcal O(v(e_{j}))  \text{ is positive cyclic order}; \\
	  	-\infty,  &\text{ if }  \mathcal O(v(e_{j}))  \text{ is negative cyclic order}.
	  \end{cases}$$
	
	By the definition of the $\boldsymbol{a}$-sequence and  $\boldsymbol{b}$-sequence,  we see that
	     $a_j(\boldsymbol k(e))=\boldsymbol k(b_j(e))$ for all $j$ by induction. 
	 
		% 
		  %	For each $v_j$ which satisfies the condition of Lemma \ref{h2}, let $[v_j, u_j]$ be the minimal path avoiding the vertices in $\bigsqcup_{j\geq 3}V_j$ with $u_j\in V_1$. The vertex set along $[v_j, u_j]$
		 % can be written as $V^e(k_j)$ for some unique $k_j\in \mathbb Z_\nu$ and for some $e\in \{\pm\}$. 
		 
			{\it  4.   Edges in	 $E\backslash E_*$.}

		We assume $\inf_n{\rm Re } \  \iota(f_n,  v(f_n))>-\infty$ for $v\in  V\backslash(\{v_0\}\cup V_*)$. By the construction of $E_*$, for any $e\in E\backslash E_*$, there are two possibilities
		 
		 (1).  $V(u(e))\cap V_*= \emptyset$, or
		 
		 (2). $V(u(e))\cap V_*\neq \emptyset$, $v(e)\notin V^*$.
		 
		 In case (1),   by Proposition \ref{phase}, 
		 \bess  \tau_{\boldsymbol k(e)}(f_n)=\pm 2\pi i  \cdot \iota(f_n,  V(u(e)))+O(1)=O(1).
		 \eess
		 This implies that Oudkerk's algorithm terminates at $r=1$.
		 % and we get a finite sequence for $\boldsymbol k(e)$.
		 
		  In case (2), note that $e$ corresponds to the edge $[u_{j_2}, v]$ of Case 1 in Step 1.
		  We  consider the edge sequence $\boldsymbol b(e)=(e_j)_{j\geq 1}$. There is a minimal integer $\ell\geq 1$ so that
		  $v(e_1)=\cdots=v(e_\ell)$ and $v(e_{\ell+1})\neq v(e_\ell)$. By   assumption $V(u(e))\cap V_*\neq \emptyset$ and applying Oudkerk's algorithm  to $\boldsymbol k(e)$,  we get a finite sequence $a_1(\boldsymbol k(e)), \cdots,  a_{\ell+1}(\boldsymbol k(e))$ with   $a_j(\boldsymbol k(e))=\boldsymbol k(b_j(e)), \ \forall 1\leq j\leq \ell+1$. 
		  By the assumption $e\in E\backslash E_*$ and $v(e)\notin V^*$, we have that 
		  $$V(u(e_{\ell+1}))\cap V_*=(V(u(e_1))\sqcup\cdots \sqcup V(u(e_\ell)))\cap V_*.$$
		  Again by Proposition \ref{phase}, 
		 $$ \sum_{j=1}^{\ell+1} \tau_{a_j(\boldsymbol k(e))}(f_n)=\pm 2\pi i \cdot \iota\Big(f_n, V(u(e_{\ell+1}))\backslash \bigsqcup_{1\leq s\leq \ell}V(u(e_s))\Big)+O(1)=O(1).
		$$
		Hence Oudkerk's algorithm terminates at $r=\ell+1$.
		% and we get a finite sequence.
		
		The proof is completed.
		 % we shall show that for each
		 %Oudkerk's algorithm for  $\boldsymbol k(e)$   gives 
		\end{proof}

	\begin{rmk}\label{repelling-plus}   If the  $f_0$-fixed point $\zeta=0$  splits into a repelling fixed point $0$ and  $\nu$ distinct non-repelling  fixed points of $f_n$, we still let
		$v_0$ denote the distinguished vertex in $V$ so that $v_0(f_n)=0$.   Proposition \ref{find-e} reads as:   for any non-empty subset $V_*\subset V \backslash \{v_0\}$ satisfying that 
		$${\rm Re } \  \iota(f_n,  v(f_n))\rightarrow +\infty, \ \forall v\in V_*,$$
		there is a subset of edges $E_*\subset E$ such that $\# E_*=\# V_*$ and for each $e\in E_*$,  
		Oudkerk's algorithm for  $\boldsymbol k(e)$   gives an infinite sequence. 
		
			\begin{figure}[h]
			\begin{center}
				\includegraphics[height=5cm]{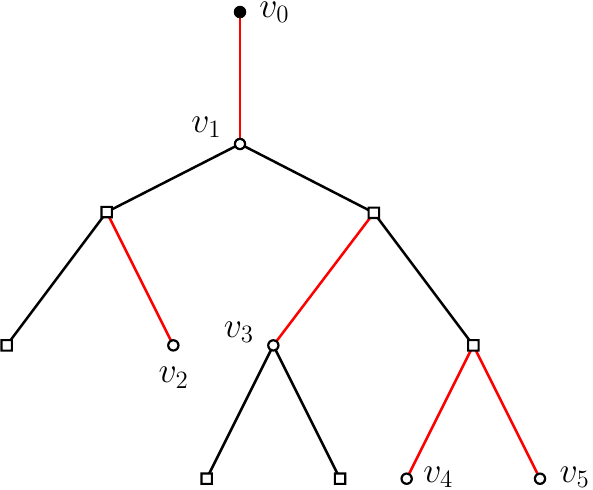}
			\end{center}
			\caption{In this example, $V_*=\{v_1, v_2, v_3, v_4, v_5\}$,  the set $E_*$ consists of five red edges.}
			\label{fig:e-choice}
		\end{figure}
		
		The construction of $E_*$ is as follows:  for each $v\in V_*$, there is a unique $v'\in V$ so that $v'\prec v$ and $[v', v]$ is an edge.   Let $E_*=\{[v',v]; v\in V_*\}$. 
		See Figure \ref{fig:e-choice}.		One may verify that  for each $e\in E_*$, Oudkerk's algorithm for  $\boldsymbol k(e)$   gives an infinite 	sequence $\boldsymbol{a}(\boldsymbol k(e))=(a_j(\boldsymbol k(e)))_{j\geq 1}$.
	\end{rmk}
	
	\section{The proofs}\label{proofs}
	
	This section gives the proofs of Theorems  \ref{h-equiv},  \ref{pb-limit}, \ref{pb-limit1}, \ref{ext-ray} and \ref{appe} in \S \ref{int}.
	
	In the following, when we are talking about a  parabolic periodic point $\zeta$ of $f_0$, we make the assumptions: 
   replacing $f_0$ by some iterate $f_0^l$, we assume  
	$f_0(\zeta)=\zeta$, $f_0'(\zeta)=1$, and the multiplicity of $f_0$ at   $\zeta$ is  $\nu+1\geq 2$.   By changing coordinate, we assume $\zeta=0$.  Recall that $A_k$ is the $k$-th immediate parabolic basin of $\zeta$.
	We shall use the notations given in \S \ref{h-e}.

\begin{proof}[Proof of Theorem \ref{h-equiv}]  Let $(f_n)_n$ be a  $*$-sequence of $f_0$. 
	By Theorem \ref{hdc-horo},  the convergence of Julia sets $J(f_{n})\rightarrow J(f_0)$ implies that $f_n\rightarrow f_0$  horocyclically.
	
	% is a type of generic sequence,  we have $J(f_{n})\rightarrow J(f_0) \Longrightarrow  f_n\rightarrow f_0   \  \text{ horocyclically}$  by 
	
	Since $f_0$ has no rotation domains, each periodic  Fatou component of $f_0$ is either attracting or parabolic \cite[\S 16]{Mil06}. Note that   each attracting Fatou component $U$ of $f_0$ has no intersection with any $X\in \mathcal{L}((J(f_{n}))_n)$ \cite{D}.

	To show that $f_n\rightarrow f_0$ horocyclically implies that $J(f_{n})\rightarrow J(f_0)$,   it suffices to show that  each parabolic Fatou component
	 of $f_0$ has no intersection with any $X\in \mathcal{L}((J(f_{n}))_n)$.
	% $A$, $A\cap \limsup_{n\rightarrow \infty} J(f_n)= \emptyset$. 
	
	%$  f_n\rightarrow f_0   \  \text{ horocyclically} \Longrightarrow J(f_{n})\rightarrow J(f_0)$. 
	For the $f_0$-parabolic  point $\zeta=0$, let $V,  A_k$ be given by \S \ref{h-e}.
  Applying Proposition \ref{find-e} to the case $V_*=V\backslash \{v_0\}$, we get $E_*=E$. Combining with  Proposition   \ref{ak-limit2},  we have that for any $k\in \mathbb Z_\nu$,  $A_{k}\cap X= \emptyset, \ \forall X\in  \mathcal{L}((J(f_{n}))_n)$. 
	
		The conclusion follows by  letting $\zeta$ range over all $f_0$-parabolic points.
%	Theorem \ref{pb-limit} follows by taking the basins $A_{\boldsymbol k(e)}, e\in E_*$.
\end{proof}

\begin{proof}[Proof of Theorem   \ref{pb-limit}]   (1)$\Longrightarrow$ (2) and (3).  Suppose $f_n\rightarrow f_0$ $\ell$-horocyclically at $\zeta$.  For every subsequence of $(f_n)_n$, there exists a further subsequence    $(f_{n_j})_j$ satisfying that 
	\begin{itemize}
		\item  $f_{n_j}\in \mathcal{WB}_\phi(\mathbf G)$ for some $\phi, \mathbf G$ and for all $j\geq 1$ (by Proposition \ref{sp}). 
		
		\item $v_0\in V$ is  horocyclically attracting (resp. repelling) for $(f_{n_j})_j$.
		
		\item there is a subset $V_*\subset V\backslash \{v_0\}$ with cardinality $\# V_*=\ell$ so that all $v\in V_*$ are horocyclically repelling (resp.  attracting) for $(f_{n_j})_j$.
		\end{itemize}
	
	%(by Proposition \ref{find-e} and Remark \ref{repelling-plus}).
	
	 By Propositions \ref{find-e} (as well as Remark \ref{repelling-plus}), \ref{ak-limit1}, \ref{ak-limit2} and \ref{lic}, 
	 %passing to a further   subsequence   of $(f_{n_j})_j$ if necessary,  
	  there is a set $E_*\subset E$  with  cardinality $\# E_*=\ell$ so that  for any $e\in E_*$, 
	$$A_{\boldsymbol k(e)}\cap X= \emptyset, \ \forall X\in  \mathcal{L}((J(f_{n_j}))_j); \  \gamma_{n_j, \boldsymbol k(e)} \rightrightarrows \gamma_k,$$
	and for any $e\in E\setminus E_*$, 
	$$A_{\boldsymbol k(e)}\cap X\neq \emptyset, \ \forall X\in  \mathcal{L}((J(f_{n_j}))_j); \  \gamma_{n_j, \boldsymbol k(e)} \not\rightrightarrows \gamma_k.$$
	The statements (2) and (3) follow by taking   $\mathcal I=\{\boldsymbol k(e); e\in E_*\}$.
	
	(2) or (3) $\Longrightarrow$ (1). If $f_n\rightarrow f$ is not $\ell$-horocyclically at $\zeta$, then there exist a subsequence $(f_{n_j})_j$  and an integer $\ell'\neq \ell$ so that 	  $f_{n_j}\rightarrow f$   $\ell'$-horocyclically at $\zeta$. 
	By the previous argument, there exist a further subsequence $(f_{m_j})_j$ of $(f_{n_j})_j$ and an index set $\mathcal{I}'$ with cardinality $\ell'$, satisfying that
	 %By the previous argument, there exist a further sequence $(f_{m_j})_j$ of $(f_{n_j})_j$ and an  index set $\mathcal I'$ with  cardinality $\# \mathcal I'=\ell'$ so that
	$$\begin{cases}   A_k\cap X=\emptyset, \ \forall X\in  \mathcal{L}((J(f_{m_j}))_j);   \gamma_{m_j, k} \rightrightarrows \gamma_k, &\text{ if }  k\in \mathcal I'; \\
		A_k\cap X\neq \emptyset,   \ \forall X\in  \mathcal{L}((J(f_{m_j}))_j);  
		\gamma_{m_j, k} \not\rightrightarrows \gamma_k, & \text{ if }  k\notin \mathcal I'.
	\end{cases}$$
The sequence $(f_{m_j})_j$ contains no further subsequence satisfying (2) and (3). This is a contradiction.
%	However this contradicts both (2) and (3). 
%	  for any
%	$k\in \mathcal I'$, 
%	$$A_{\boldsymbol k(e)}\cap \limsup_{j\rightarrow \infty} J(f_{n_j})= \emptyset, \  \gamma_{n_j, \boldsymbol k(e)} \rightrightarrows \gamma_k;$$
%	for any $e\in E\setminus E_*$, 
%	$$A_{\boldsymbol k(e)}\cap \liminf_{j\rightarrow \infty} J(f_{n_j})\neq  \emptyset, \  \gamma_{n_j, \boldsymbol k(e)} \not\rightrightarrows \gamma_k;.$$
	\end{proof}

\begin{proof}[Proof of Theorem \ref{pb-limit1}]   (1)$\Longrightarrow$ (2) and (3). 
Note that 	$f_n\rightarrow f_0$ horocyclically at $\zeta$ is equivalent to say that 	$f_n\rightarrow f_0$ $\nu$-horocyclically at $\zeta$.
	 By  Theorem   \ref{pb-limit},   any subsequence  of $(f_n)_n$ contains a further subsequence $(f_{n_j})_j$ so that 
	$$X\cap (A_1\cup \cdots \cup A_\nu)=\emptyset, \ \forall X\in  \mathcal{L}((J(f_{n_j}))_j);  \  \gamma_{n_j, k} \rightrightarrows \gamma_k, \ \forall k\in \mathbb Z_\nu.$$ 
	It follows that the whole sequence   $(f_{n})_n$ satisfies that 
		$$X\cap (A_1\cup \cdots \cup A_\nu)=\emptyset, \ \forall X\in  \mathcal{L}((J(f_{n}))_n);  \  \gamma_{n, k} \rightrightarrows \gamma_k, \ \forall k\in \mathbb Z_\nu.$$ 
%	$$ \limsup_{j\rightarrow \infty} J(f_{n})\cap (A_1\cup \cdots \cup A_\nu)=\emptyset;  \  \gamma_{n, k} \rightrightarrows \gamma_k, \forall k\in \mathbb Z_\nu.$$ 
	
Applying Theorem   \ref{pb-limit} to the case $\ell=\nu$, we get	(2) or (3) $\Longrightarrow$ (1). %follows from Theorem   \ref{pb-limit}.
\end{proof}

\begin{proof}[Proof of Theorem \ref{ext-ray}] 
	%  By Mobius conjugation and replacing $f_0$ by some iterate $f_0^m$,  we assume $f_0(z)=z+z^{\nu+1}+O(z^{\nu+1})$ near $0$, and 
	We assume the parabolic fixed point $\zeta=0$ of $f_0$ splits into the fixed points of $f_n$, one of which is $0$.  Take two small disks $U, V$ centered at $0$ so that $f_n|_U, f_0|_U$ are univalent, and their images $f_n(U), f_0(U)$ contain $V$.
% Consider the inverse sequence $(f_n|_{V}^{-1})_n$.
 Note that   $\lambda_n\rightarrow 1$ horocyclically if and only if    $1/\lambda_n\rightarrow 1$ horocyclically.  This implies that $f_n\rightarrow f_0$ $\ell$-horocyclically at $0$ if and only if the inverse sequence  $f_n^{-1}|_{V} \rightarrow f_0^{-1}|_{V}$ $\ell$-horocyclically at $0$.
 %  of the local germs of $f_n$ near parabolic fixed points. 

%The tails of the invariant rays.
	Let $\gamma_{n,k}$ and $\gamma_{k}$ be the tails (i.e. whose Green potentials $\leq \epsilon_0$ for some small $\epsilon_0>0$) of the external rays $R_{f_n}(\theta_k)$ and  $R_{f_0}(\theta_k)$ respectively, so that $\gamma_{k}\subset V$  (it is possible that $\gamma_{n,k}$ is not contained in $V$) and $\gamma_{n,k}$ converges locally and uniformly to $\gamma_k$ (with suitable parameterization). 
	The conclusion follows from Theorem \ref{pb-limit}.
	\end{proof}

\begin{rmk} \label{ext-generic} Let $f_n, f_0\in \mathcal C_d$. Replacing $f_0$ by some iterate, we assume $\zeta$ is an $f_0$-parabolic fixed point. Recall that $\Theta$ consists of all angles $t\in \mathbb R/\mathbb Z$ for which $R_{f_0}(t)$ lands at $\zeta$.
If $(f_n)_n$ is  a generic perturbation of $f_0$ at $\zeta$, then the following two statements are equivalent:
	
	(1).  For each $n$,  there is an $f_n$-fixed point  $\zeta_n$ so that   
	$$\zeta_n\rightarrow \zeta,  \text{ and } f_n'(\zeta_n)/f'_0(\zeta)\rightarrow 1  \text{ horocyclically}.$$
	
	(2). For every subsequence of $(f_n)_n$, there exist a further subsequence $(f_{n_j})_j$ and a  $\theta\in \Theta$  such that $\overline{R_{f_{n_j}}(\theta)}\rightarrow \overline{R_{f_0}(\theta)}$ in Hausdorff topology.
	
This implies,	in particular, that $\overline{R_{f_0}(\theta)}\subsetneq X,  \forall X\in \mathcal{L}((\overline{R_{f_n}(\theta)})_n)$ for all   $\theta\in \Theta$ if and only if for any $f_n$-fixed point  $\zeta_n$ satisfying that $\zeta_n\rightarrow \zeta$,  we have
$$\sup_{n\geq 1}\Big|{\rm Re}\Big(\frac{f'_0(\zeta)}{f'_0(\zeta)-f_n'(\zeta_n)}\Big)\Big|< +\infty.$$
\end{rmk}

\begin{proof} Consider the tails $\gamma_{n,k}$ and $\gamma_{k}$  of  the external rays $R_{f_n}(\theta_k)$ and  $R_{f_0}(\theta_k)$  respectively, the conclusion follows from Theorem \ref{limit-A}  (1)$\Longleftrightarrow$ (3).
\end{proof}

%\begin{rmk} \label{ext-generic2} Let $f_n, f_0\in \mathcal C_d$.  Let $\zeta$ be a  $f_0$-parabolic periodic  point, and assume  $(f_n)_n$ is a generic perturbation of $f_0$ at $\zeta$.   Recall that  $\Theta$ consists of all angles $t\in \mathbb R/\mathbb Z$ for which $R_{f_0}(t)$ lands at $\zeta$.  If   $f_n\rightarrow f_0$  horocyclically at $\zeta$, then $\overline{R_{f_{n}}(\theta)}\rightarrow \overline{R_{f_0}(\theta)}$ in Hausdorff topology for all $\theta\in \Theta$.
	
%	If  $(f_n)_n$ is a $*$-sequence of $f_0$ at $\zeta$,  then  $f_n\rightarrow f_0$  horocyclically at $\zeta$ if and only if $\overline{R_{f_{n}}(\theta)}\rightarrow \overline{R_{f_0}(\theta)}$ in Hausdorff topology for all $\theta\in \Theta$.
%	If  $(f_n)_n$ is a generic perturbation of $f_0$ at $\zeta$,  then 
	
%	 $\overline{R_{f_{n}}(\theta)}\rightarrow \overline{R_{f_0}(\theta)}$ in Hausdorff topology for all $\theta\in \Theta_{\nu}$, here $\Theta_\nu$ is defined in \S \ref{int} (before Theorem \ref{ext-ray}), then  $f_n\rightarrow f_0$  horocyclically at $\zeta$.
%\end{rmk}

%We remark that by Theorem \ref{ext-ray}, if  $(f_n)_n$ is a $*$-sequence of $f_0$ at $\zeta$,  then  $f_n\rightarrow f_0$  horocyclically at $\zeta$ if and only if $\overline{R_{f_{n}}(\theta)}\rightarrow \overline{R_{f_0}(\theta)}$ in Hausdorff topology for all $\theta\in \Theta$.

\begin{proof} [Proof of Theorem \ref{ext-ray0}]  For each $k\in \mathbb Z_\nu$,  consider the tails $\gamma_{n,k}$ and $\gamma_{k}$  of  the external rays $R_{f_n}(\theta_k)$ and  $R_{f_0}(\theta_k)$  respectively, where $\theta_k$ is an angle in $\Theta_\nu$ so that $R_{f_0}(\theta_k)$ lands at $\zeta$ in the $k$-th repelling direction.
	By Proposition \ref{lic} and considering  inverse germs $f_n^{-1}, f_0^{-1}$, 
 Oudkerk's algorithm for $k\in \mathbb Z_\nu$ never terminates. The conclusion follows from the same argument as that of Theorem   \ref{hdc-horo}.
 The last statement follows from Theorem \ref{ext-ray}.
\end{proof}

\begin{proof}[Proof of Theorem \ref{appe}] Note that  $f_{\lambda_0}$ has precisely one $q$-cycle of immediate parabolic basins,  since $f_{\lambda_0}$ has only one critical point in $\mathbb C$.
 It follows that $g_0=f_{\lambda_0}^q$ takes the form 
  $g_0(z)=z+az^{q+1}+O(z^{q+2})$ near $0$, where $a\neq 0$. For a sequence $\lambda_n\rightarrow \lambda_0$ with $\lambda_n\neq \lambda_0$, note that $g_n=f_{\lambda_n}^q$   is a generic perturbation of $g_0$. Moreover,  $g_n$ has a fixed point $0$ with multiplier $\lambda_n^q$, and $q$ fixed points with  same mulitipliers $\mu_n$ (these $q$ fixed points correspond to a $q$-periodic orbit of $f_{\lambda_n}$). Hence by \eqref{asm-iota},
  \begin{equation}\label{sum-i}
  \frac{q}{1-\mu_n}+\frac{1}{1-\lambda_n^q}\rightarrow \iota(g_0, 0).
  \end{equation}

  If $J(f_{\lambda_n})\rightarrow J(f_{\lambda_0})$, by Theorem \ref{hdc-horo}, $g_n\rightarrow g_0$ horocyclically.   In particular,   $|{\rm Re}(1/({\lambda_n^q-1}))|\rightarrow +\infty$.
  
  On the other hand, if $|{\rm Re}(1/({\lambda_n^q-1}))|\rightarrow +\infty$, then  \eqref{sum-i} implies that $|{\rm Re}(1/(1-\mu_n))|\rightarrow +\infty$. Hence  $g_n\rightarrow g_0$ horocyclically.  Note that in this setting,  $(g_n)_n$ is a $*$-sequence  of $g_0$. By Theorem \ref{h-equiv},  %we have 
  $J(f_{\lambda_n})\rightarrow J(f_{\lambda_0})$.
\end{proof}

\begin{rmk} \label{leaned} %Let $f_0$ have a parabolic fixed point $\zeta$ of multiplicity $\nu+1$, $f'_0(\zeta)=1$. Let $A_1, \cdots, A_\nu$ be the immediate parabolic basins of $f_0$ at $\zeta$.   
	Let $f_n\rightarrow f_0$ be  a leaned sequence of $f_0$ at $\zeta$. 
	Then %for any $k\in \mathbb Z_\nu$, 
		$$X\cap A_k\neq \emptyset,  \ \forall X\in  \mathcal{L}((J(f_{n}))_n), \ \forall k\in \mathbb Z_\nu.$$
		This implies in particular that $J(f_0)\subsetneq X, \ \forall X\in  \mathcal{L}((J(f_{n}))_n)$.
		\end{rmk}
	\begin{proof} We may assume $f_{n}\in \mathcal{WB}_\phi(\mathbf G)$ for some $\phi, \mathbf G$ and for all $n\geq 1$ (by Proposition \ref{sp}).  Let $V$ be the vertex set  induced by $\mathbf G$ (see \S \ref{h-e}).  Since $(f_n)_n$ is  leaned sequence of $f_0$ at $\zeta$,  by Lemma  \ref{asm-iota}, 
		for any $v\in V$,  
		\begin{equation} \label{lean-eq}
			\sup_{n}|{\rm Re} \ \iota(f_n, v(f_n))|<+\infty.
			\end{equation}    Hence for any $k\in \mathbb Z_\nu$, Oudkerk's algorithm for $k$ can   terminate at step (3) for $r=1$.  The conclusion follows from Proposition \ref{ak-limit1}.
		\end{proof}
	
	\begin{rmk}  In Remark \ref{leaned},  we further assume that $f_n, f_0$ are polynomials with connected Julia sets (this can happen, for example, when $f_n$ approaches $f_0$ along the bifurcation locus in $\mathcal P_d$). 
		
		 Then \eqref{lean-eq} implies that for any $f_n$-fixed point  $\zeta_n$ so that  $\zeta_n\rightarrow \zeta$,   the multiplier sequence  $(f_n'(\zeta_n))_n$  does not converge to $1$ horocyclically. By  Remark \ref{ext-generic}, 
		for any external ray $R_{f_0}(\theta)$ landing at $\zeta$,  
	$$\overline{R_{f_0}(\theta)}\subsetneq	 X, \ \forall X\in  \mathcal{L}((\overline{R_{f_{n}}(\theta)})_n).$$
	\end{rmk}

\end{document}